\documentclass[11pt]{smfart}
\usepackage{amsmath, amssymb, enumerate,  graphicx,tikz, stmaryrd, eufrak,enumitem,MnSymbol}
\usepackage[french,english]{babel}
\usepackage[T1]{fontenc}
\usepackage{thmtools}

\usepackage{cleveref}
\usepackage{bbm}
\usepackage{cases}
\usepackage{array}
\usepackage{tikz-cd}
\usepackage[all]{xy}

\usepackage[hmarginratio={1:1},vmarginratio={1:1},lmargin=80.0pt,tmargin=90.0pt]{geometry}

\usepackage{tikz}
\tikzset{anchorbase/.style={baseline={([yshift=-0.5ex]current bounding box.center)}}}
\usetikzlibrary{decorations.markings}
\usetikzlibrary{decorations.pathreplacing}
\usetikzlibrary{arrows,shapes,positioning,backgrounds}
\tikzstyle directed=[postaction={decorate,decoration={markings,
    mark=at position #1 with {\arrow{>}}}}]
\tikzstyle rdirected=[postaction={decorate,decoration={markings,
    mark=at position #1 with {\arrow{<}}}}]
    
\usepackage{graphicx}
\usepackage[vcentermath]{youngtab}
\usepackage{nicefrac}

\numberwithin{equation}{section}

\newtheorem{theorem}[subsubsection]{Theorem}
\newtheorem{lemma}[theorem]{Lemma}
\newtheorem{prop}[theorem]{Proposition}
\newtheorem{corollary}[subsubsection]{Corollary}

\newtheorem{conjecture}[theorem]{Conjecture}

\theoremstyle{definition}
\newtheorem{definition}[subsubsection]{Definition}

\newtheorem{remark}[theorem]{Remark}
\newtheorem{problem}[theorem]{Problem}

\newtheorem{example}[subsubsection]{Example}

\newtheorem{question}[theorem]{Question}

\newcommand{\Pol}{\mathrm{Pol}}
\newcommand{\SPol}{\mathrm{SPol}}

\newcommand{\UFun}{\mathrm{UFun}}

\newcommand{\bA}{\mathbf{A}}
\newcommand{\bB}{\mathbf{B}}
\newcommand{\bC}{\mathbf{C}}

\newcommand{\bI}{\mathbf{I}}

\newcommand{\bT}{\mathbf{T}\mathbf{C}}

\newcommand{\Tak}{\operatorname{\mathcal{T}{ak}}}
\newcommand{\mA}{\mathbf{A}}
\newcommand{\mB}{\mathbf{B}}
\newcommand{\mS}{\mathbf{S}}

\newcommand{\cP}{\mathcal{P}}

\newcommand{\cD}{\mathcal{D}}
\newcommand{\cE}{\mathcal{E}}
\newcommand{\uC}{\underline{\mathcal{C}}}
\newcommand{\co}{\mathrm{co}}
\newcommand{\ev}{\mathrm{ev}}

\newcommand{\GL}{\operatorname{GL}}
\newcommand{\SL}{\operatorname{SL}}

\newcommand{\Fun}{\mathsf{Fun}}

\newcommand{\Mod}{\mathsf{Mod}}
\newcommand{\Rep}{\mathsf{Rep}}

\newcommand{\Tilt}{\operatorname{\mathsf{Tilt}}}

\newcommand{\mo}{\mbox{{\rm mod}}}
\newcommand{\cA}{\mathcal{A}}

\newcommand{\cB}{\mathcal{B}}

\newcommand{\bk}{\mathbf{k}}

\newcommand{\gr}{\mathrm{gr}}

\newcommand{\tto}{\twoheadrightarrow}
\newcommand{\cO}{\mathcal{O}}
\newcommand{\cJ}{\mathcal{J}}

\newcommand{\mN}{\mathbb{N}}
\newcommand{\mG}{\mathbb{G}}
\newcommand{\mZ}{\mathbb{Z}}
\newcommand{\Frg}{\mathrm{Frg}}

\newcommand{\ex}{\mathrm{ex}}
\newcommand{\rex}{\mathrm{rex}}

\newcommand{\mC}{\mathbb{C}}
\newcommand{\mR}{\mathbb{R}}

\newcommand{\mF}{\mathbb{F}}

\newcommand{\End}{\mathrm{End}}

\newcommand{\Ob}{\mathrm{Ob}}

\newcommand{\Hom}{\mathrm{Hom}}
\newcommand{\Ann}{\mathrm{Ann}}
\newcommand{\Aut}{\mathrm{Aut}}
\newcommand{\Nat}{\mathrm{Nat}}
\newcommand{\Sym}{\mathrm{Sym}}

\newcommand{\Young}{\mathbf{Young}}
\newcommand{\UnEn}{\mathbf{UnEn}}
\newcommand{\SYoung}{\mathbf{SYoung}}
\newcommand{\Spin}{\mathbf{Spin}}
\newcommand{\Fr}{\mathrm{Fr}}

\newcommand{\op}{\mathrm{op}}
\newcommand{\Ind}{\mathrm{Ind}}
\newcommand{\Res}{\mathrm{Res}}

\newcommand{\Id}{\mathrm{Id}}
\newcommand{\Vecc}{\mathsf{Vec}}
\newcommand{\sVec}{\mathsf{sVec}}
\newcommand{\Ver}{\mathsf{Ver}}

\newcommand{\mT}{\mathbb{T}}

\newcommand{\ICS}{\mathbf{ICS}}

\newcommand{\unit}{{\mathbbm{1}}}
\newcommand{\cC}{\mathcal{C}}

\newcommand{\Proj}{\operatorname{\mathsf{Proj}}}
\newcommand{\Inj}{\operatorname{\mathsf{Inj}}}

\newcommand{\Perm}{\mathbf{Perm}}
\newcommand{\Irr}{\operatorname{Irr}}
\newcommand{\Indec}{\operatorname{Idc}}
\newcommand{\mI}{\mathbf{I}}
\newcommand{\cS}{\mathcal{S}}
\newcommand{\Triv}{\mathrm{Triv}}

\newcommand{\mD}{\mathbb{D}}
\newcommand{\uHom}{\underline{\mathrm{Hom}}}

\newcommand{\uEnd}{\underline{\mathrm{End}}}

\newcommand{\cM}{\mathcal{M}}
\newcommand{\uM}{\underline{\mathcal{M}}}
\newcommand{\cN}{\mathcal{N}}
\newcommand{\uN}{\underline{\mathcal{N}}}

\newcommand{\RT}{R\hspace{-0.4mm}T}

\begin{document}

\selectlanguage{english} 

\title[Polynomial functors on tensor categories]{Inductive systems of the symmetric group, polynomial functors and tensor categories

---

Systèmes inductifs du groupe symétrique, foncteurs polynomiaux et catégories tensorielles}
\author{Kevin Coulembier}
\address{K.C.: School of Mathematics and Statistics, University of Sydney, NSW 2006, Australia}
\email{kevin.coulembier@sydney.edu.au}

%\thanks{.}

%\subjclass[2020]{}

\begin{abstract}
We initiate the systematic study of modular representations of symmetric groups that arise via the braiding in (symmetric) tensor categories over fields of positive characteristic. We determine what representations appear for certain examples of tensor categories, develop general principles and demonstrate how this question connects with the ongoing study of the structure theory of tensor categories. We also formalise a theory of polynomial functors as functors that act coherently on all tensor categories. We conclude that the classification of such functors is a different way of posing the above question of which representations of symmetric groups appear. Finally, we extend the classical notion of strict polynomial functors from the category of (super) vector spaces to arbitrary tensor categories, and show that this idea is also a different packaging of the same information.

\medskip
\medskip

\hrule

\medskip
\selectlanguage{french}

Nous entreprenons l’étude systématique des représentations modulaires des groupes symétriques qui apparaissent par le tressage des catégories tensorielles (symétriques) sur des corps de caractéristique positive. Nous déterminons quelles représentations interviennent pour certains exemples de catégories tensorielles, développons des principes généraux et montrons comment cette question s’articule avec l’étude en cours de la théorie structurelle des catégories tensorielles. Nous formalisons également une théorie des foncteurs polynomiaux en tant que foncteurs agissant de manière cohérente sur toutes les catégories tensorielles. Nous concluons que la classification de tels foncteurs constitue une autre manière de poser la question précédente de quelles représentations des groupes symétriques apparaissent. Enfin, nous généralisons la notion classique de foncteurs polynomiaux stricts, de la catégorie des (super)espaces vectoriels aux catégories tensorielles arbitraires, et montrons que cette idée constitue également une reformulation différente de la même information.
\selectlanguage{english} 

\end{abstract}

\maketitle
%\today

%\tableofcontents

\section*{Introduction}

The structure theory of tensor categories has been an active topic in the last decade. We use the term tensor category over a field $\bk$ in the sense of \cite{Del90, EGNO}, for a {\em symmetric} rigid monoidal $\bk$-linear abelian category with some finiteness assumptions.

Over fields of characteristic zero, classical results of Deligne \cite{Del90, Del02} show that tensor categories `of moderate growth' must be representation categories of groups or supergroups. Without the moderate growth assumption, see for instance \cite{HS, HSS}, or over fields of positive characteristic, see for instance \cite{BE, BEO, Tann, AbEnv, CEO, CEO2, CF, EOf, Os}, exploring the structure theory remains ongoing.

To motivate the work in the current paper we give some background on the state of the art regarding tensor categories of moderate growth over fields of positive characteristic. The principle of tannakian reconstruction of \cite{Del90} reduces the structure theory problem to the classification of `incompressible' tensor categories, see \cite[Theorem~5.2.1]{CEO2}. By \cite{Del90, Del02} the only incompressible categories of moderate growth over a field $\bk$ of characteristic zero are thus the categories of vector spaces $\Vecc=\Vecc_{\bk}$ and supervector spaces $\sVec=\sVec_{\bk}$ over $\bk$. In contrast, for characteristic $p>0$, in \cite{BE, BEO, AbEnv}, a chain of incompressible categories of moderate growth
$$\Vecc\;\subset\;\sVec\;\subset\;\Ver_p\;\subset\;\Ver_{p^2}\;\subset\; \Ver_{p^3}\;\subset\;\cdots $$
was constructed. In \cite{BEO} it was conjectured that every tensor category of moderate growth admits a tensor functor to $\Ver_{p^\infty}=\cup_n\Ver_{p^n}$; or equivalently that all incompressible categories of moderate growth are subcategories of $\Ver_{p^\infty}$. Working towards this conjecture, in \cite{CEO} it was proved that a tensor category of moderate growth admits a tensor functor to $\Ver_p$ if and only if it is `Frobenius-exact', meaning the Frobenius functor
\begin{equation}\label{eqFr}\Fr:\cC\to\cC\boxtimes \Ver_p
\end{equation}
is exact. The functor $\Fr$ is defined in \cite{Os, EOf, Tann} as follows. One sends $X\in\cC$ to $X^{\otimes p}$, which can be interpreted as an $S_p$-representation internal to $\cC$ via the braiding, and subsequently one manipulates $X^{\otimes p}$ via a non-exact symmetric monoidal functor $\Rep S_p\to\Ver_p$. These results lead to three (interrelated) questions:

(Q1) By Takeuchi's theorem, for a tensor category $\cC$ there exists a faithful exact (non-monoidal) functor $\omega:\cC\to\Vecc$, so that $\omega(X^{\otimes d})$ is an ordinary $S_d$-representation over $\bk$. Attempts at constructing `higher' versions of $\Fr$ in \cite{CF}, relating to $\Ver_{p^n}$ rather than $\Ver_p$, lead to the question of which $S_d$-representations can occur in this way for $d=p^n$. A concrete example of such questions for $d=p$ already appeared in \cite[Question~7.3]{CEO}.

(Q2) The functor $\Fr$ in \eqref{eqFr} `commutes with tensor functors'. By taking direct summands of $\Fr$, one obtains functors $\cC\to\cC$ that commute with tensor functors. More familiar examples are the (skew) symmetric powers $\Sym^d$ and $\bigwedge^d$, and all of the above functors are subquotients of the functor $X\mapsto X^{\otimes d}$. In characteristic zero, such functors simply produce (direct sums of) Schur functors $\mS_\lambda$, see \cite{Del02}. In positive characteristic, many interesting questions, such as their `classification', are open. We will start by developing a rigorous theory and definition of such functors, and prove that the classification question is just a reformulation of (Q1).

(Q3) A logical term for the functors in (Q2) would be `polynomial' functors. Classical strict polynomial functors were introduced in \cite{FS}, as a tool for proving finite generation of the cohomology ring of finite group schemes. In the description of \cite{Kr}, the category of strict polynomial functors of degree $d$ is the category of $\bk$-linear functors
\begin{equation}\label{Equiv0}
\Fun_{\bk}(\Gamma^d\Vecc,\Vecc)\;\simeq\;\Rep^d\GL_V,
\end{equation}
where the equivalence is \cite[Theorem~3.2]{FS} for a vector space $V$ of dimension at least $d$ and $\Rep^d\GL_V$ is the category of polynomial representations of degree $d$. Here $\Gamma^d\Vecc$ is the category with the same class of objects as $\Vecc$, but with morphism spaces
$$\Gamma^d\Vecc(U,V)\;:=\;\Hom_{\bk S_d}(U^{\otimes d},V^{\otimes d}),$$
so that linear functors out of $\Gamma^d\Vecc$ indeed correspond to `polynomial of degree $d$' functors out of $\Vecc$. In \cite{Ax}, a `super' version of polynomial functors was introduced, and this was used in \cite{Dr} to prove cohomological finite-generation for finite supergroup schemes. One can follow this template for $\Vecc$ and $\sVec$ and define a category of strict polynomial functors based on every (incompressible) tensor category, and search for equivalences as in \eqref{Equiv0}. In the long term, one would hope to use this towards proving \cite[Conjecture~2.18]{EOconj} regarding cohomological finite-generation for finite tensor categories. In the current paper we only show that this study is also equivalent to (Q1) and (Q2).

\subsection*{Prelude: Schur-Weyl duality} The double centraliser property between the symmetric group $S_d$ and the general linear group $\GL_V$ of a complex vector space $V$, say with $\dim_{\mC}V\ge d$, acting on $V^{\otimes d}$, leads to an equivalence of $\mC$-linear categories
\begin{equation}\label{SW1}
\Rep_{\mC}S_d\;\xrightarrow{\sim}\; \Rep^d_{\mC}\GL_V,\quad M\mapsto V^{\otimes d}\otimes_{\mC S_d}M.
\end{equation}
Since $\Rep_{\mC}S_d$, as a $\mC$-linear category, is a finite direct sum of $\Vecc_{\mC}$, it is equivalent to the category of $\mC$-linear functors from $\Rep_{\mC}S_d$ to $\Vecc_{\mC}$. The resulting equivalence
\begin{equation}\label{SW2}\Fun_{\mC}(\Rep_{\mC}S_d,\Vecc_{\mC})\;\xrightarrow{\sim}\;\Rep_{\mC}^d\GL_V,\quad F\mapsto F(V^{\otimes d})=\bigoplus_{\lambda\vdash d}F(S^\lambda)\otimes \mS_\lambda(V)
\end{equation}
is less common, but perhaps more natural as we explain below. Firstly, note that $S^\lambda$ denotes the (simple) Specht module of $S_d$, so that, as a bimodule,
$$V^{\otimes d}\;\simeq\;\bigoplus_{\lambda\vdash d} S^\lambda\boxtimes\mS_\lambda(V).$$

While equivalence \eqref{SW1} does not extend to positive characteristic, its incarnation \eqref{SW2} can be extended very neatly. Denote by $\Young^d\subset\Rep S_d$ the full subcategory of direct sums of Young modules, see \cite{Er}. Then over any field $\bk$ we do get an equivalence
\begin{equation}\label{SW3}\Fun_{\bk}(\Young^d,\Vecc)\;\stackrel{\sim}{\to}\; \Rep^d \GL_V,\quad F\mapsto F(V^{\otimes d}).\end{equation}
We can observe that $\bk$-linear functors out of $\Young^d$ are just modules over the Schur algebra $S(n,d)$, for $n\ge d$, which yields the more standard interpretations of \eqref{SW3}.

Since the functor $\Gamma^d\Vecc\to\Young^d, \;U\mapsto U^{\otimes d}$ corresponds to the inclusion of $\Gamma^d\Vecc$ into its idempotent completion, the left-hand sides of \eqref{Equiv0} and  \eqref{SW3} can be canonically identified.
Moreover, while we used Schur functors to make \eqref{SW2} precise, the assignment $V\mapsto F(V^{\otimes d})$ for some functor $F$ from $\Young^d$ to $\Vecc$, as appears in \eqref{SW3}, paves the way for more general definitions of polynomial functors in the sense of (Q2).
Finally, $\Young^d$ is clearly the category of $S_d$-representations that arise via the braid action from the tensor category $\Vecc$ as in (Q1). It thus follows that \eqref{SW3} gives an ansatz for connecting (Q1), (Q2) and (Q3).

To explain our results, we henceforth fix an algebraically closed field $\bk$ of prime characteristic~$p$ and only consider tensor categories $\cC$ over $\bk$.

\subsection*{Representations appearing in tensor categories (Q1)}
We introduce in \S\ref{sec:defis} the notion of an `inductive system', which is an assignment $\mA$ of a pseudo-abelian subcategory $\mA^d\subset\Rep S_d$ for each $d$, with strong compatibility conditions under $\Res^{S_d}_{S_{d-1}}$. For $\mathrm{char}(\bk)>0$, {\em semisimple} inductive systems were classified by Kleshchev in \cite{CS}.

We show that for a tensor category $\cC$ and $X\in\cC$, the representations $\omega(X^{\otimes d})$ from above define an inductive system $\mB_X[\cC]$, which is independent of $\omega$ and invariant under tensor functors. We show in Theorem~\ref{ThmVerp} that the $\mB_L[\Ver_p]$, for $L$ running over simple objects in $\Ver_p$, give precisely the semisimple inductive systems from~\cite{CS}.

For a tensor category $\cC$, we get an inductive system $\mB[\cC]=\sum_{X}\mB_X[\cC]$, which is closed under the induction product, ordinary tensor products, taking duals, and contains the trivial representations, see Theorem~\ref{ThmTCIS}. The minimal such `closed' system is the inductive system $$\Young\;=\;\mB[\Vecc]$$ of Young modules. Similarly, $\sVec$ yields the signed Young modules as studied in \cite{Donkin} and $\mB[\Ver_p]$ is a new closed inductive system. These results lead to interesting observations regarding modular representation theory of $S_d$. An explicit consequence is \Cref{CorAlg} stating that every `completely splittable' representation, see \cite{CS}, is `algebraic'. We also obtain in Proposition~\ref{Propp2}, for $p=2$, strict inclusions
$$\mB[\Vecc]\;\subset\;\mB[\Ver_4^+]\;\subset\;\mB[\Ver_4]\;\subset\;\mB[\Ver_8^+],$$
which lead to intriguing questions whether the chain continues to be strictly ascending, and whether the union is the inductive system of all representations.

\subsection*{Annihilator ideals of objects in tensor categories} The shadow in $K_0(\Rep S_n)$ of an inductive system yields an inductive system in a related sense due to Zalisskii \cite{Za}, see Remark~\ref{rem:Zal}. The latter notion was used by Baranov and Kleshchev in \cite{BK} to classify the maximal ideals in $\bk S_\infty$ when $p>2$. We reformulate their result in terms of tensor categories. To every object in a tensor category (more generally, to any inductive system) we can assign an ideal in $\bk S_\infty$. For $X\in\cC$ it is given by the kernel $\Ann(X)$ of the braid morphism `$\bk S_\infty \to \End_{\cC}(X^{\otimes \infty})$'. 

The maximal ideals in $\bk S_\infty$ for $p>2$ are then precisely $\Ann(L)$, for $L$ varying over the simple objects in $\Ver_p$. As a consequence of this realisation, we obtain a very simple description of these maximal ideals, which appears to be new; they are all generated by one symmetriser and/or one skew symmetriser. For $p=2$ the classification of maximal ideals is not yet complete, but we show that the two known maximal ideals are given precisely by $\Ann(L)$, for $L$ varying over the simple objects in $\Ver_4$ (but, we can no longer take $\Ver_2=\Vecc$). These results are proved in \Cref{PropNewGen} and \Cref{LemSpinIdeal}.

\subsection*{Strict polynomial functors (Q3)}
For every tensor category $\cC$, we propose in \S 7.1 a notion of a category $\SPol^d_{\circ}\cC$ of strict polynomial functors, generalising the known cases $\cC$ equal to $\Vecc$ or $\sVec$. Extending \eqref{Equiv0} we prove that, for an object $X\in\cC$, we have
$$\SPol^d_{\circ}\cC\;\simeq\; \Rep_{\cC}^d\GL_X$$
if and only if $\mB^d[\cC]=\mB^d_X[\cC]$. Here $\GL_X$ is the general linear affine group scheme internal to $\cC$ associated to $X$. In more detail, in \Cref{ThmMain} we actually prove equivalences
$$\cC\boxtimes\mo(\mB^d[\cC])\;\simeq\;\SPol_\circ^d\cC\qquad\mbox{and}\qquad \cC\boxtimes\mo(\mB^d_X[\cC])\;\simeq\;\Rep^d_{\cC}\GL_X,$$
where we write $\mo(-)$ for an abelian subcategory of $\Fun(-,\Vecc)$ of functors satisfying a minor finiteness property. In particular, the first equivalence demonstrates why (Q1) and (Q3) are simply different packagings of the same content.

Returning to Schur-Weyl duality, we prove in \Cref{PropInvTh} that for a tensor category~$\cC$ and $X\in\cC$, the braiding morphism
$$\bk S_d\;\to\;\End_{\GL_{X}}(X^{\otimes d})$$
is an isomorphism if and only if it is injective if and only if $\bk S_d\in\mB^d_X[\cC]$.

\subsection*{Universal polynomial functors (Q2)}
In \S 7.2 we define a universal functor to be the assignment of an endofunctor to each tensor category (over a fixed field) which `commutes with tensor functors'. The category $\Pol^d_{\bk}$ of universal polynomial functors of degree $d$ is then the topologising (= closed under taking subquotients and direct sums) subcategory of the category of universal functors generated by the universal functor `$X\mapsto X^{\otimes d}$'.
We prove an equivalence
$$\Pol^d_{\bk}\;\simeq\; \mo(\bT^d),$$
where $\bT^d$ is the sum of $\mB^d[\cC]$, for $\cC$ running over all tensor categories over $\bk$. As hoped, the two notions of polynomial functors are thus intimately related. For example
$$\cC\boxtimes\Pol^d_{\bk}\;\simeq\; \SPol^d_{\circ}\cC,$$
if $\mB^d[\cC]=\bT^d$, and generally $\SPol^d_{\circ}\cC$ is a Serre quotient of $\cC\boxtimes\Pol^d_{\bk}$. Note that \cite[Conjecture~1.4]{BEO} predicts that, at least if we focus solely on tensor categories of moderate growth,
$$\bT^d\;=\,\mB^d[\Ver_{p^\infty}], \qquad\mbox{so}\quad \Pol^d_{\bk}\;\simeq\; \mo(\mB^d[\Ver_p^\infty]).$$

Our results also prove some plausible statements that were lacking proof before, for instance that $\Fr_+X$ and $\bigwedge^2X$ are always simple $\GL_X$-representations, see Example~\ref{ExFrg2}(2).

\subsection*{Outlook and speculation}
Above we already mentioned several questions that arise from the current work and potential applications. More questions are presented throughout the text. Here we comment on potential applications to the structure theory of tensor categories.

For this purpose, it is convenient to consider a finite tensor category $\cC$. Then another interesting inductive system, besides $\mB[\cC]$, is $\mB_{\cP r}[\cC]$, which is similarly defined but only considering the projective objects in $\cC$. This system is closed except that it need not contain the trivial representation. In fact, we prove that $\cC$ is Frobenius-exact if and only if~$\mB^p_{\cP r}[\cC]$ contains trivial representations if and only if~$\mB_{\cP r}[\cC]$ is closed. By \cite{EOf} or \cite{CEO}, $\mB^p_{\cP r}[\cC]$ thus determines whether $\cC$ admits a tensor functor to $\Ver_p$. Based on recent progress in~\cite{CF}, it seems plausible that similarly $\mB^{p^2}_{\cP r}[\cC]$ could determine whether $\cC$ admits a tensor functor to $\Ver_{p^2}$ {\em etc.} On the other hand, \cite[Conjecture~1.4]{BEO} predicts that $\mB_{\cP r}[\cC]$ can only contain `very specific' representations. Should the conjecture be false this opens avenues for proving so.

\subsection*{Structure} In Section~\ref{SecPrel} we recall some necessary background. Sections~\ref{sec:indsys} to \ref{sec:ideals} are concerned with the notion of inductive systems and their appearance from tensor categories, while Sections~\ref{sec:back} to \ref{sec:coin} are devoted to both notions of polynomial functors and all above equivalences.

%%%%%%%%%%%%%%%%%%%%%%%%%%%%%%%%%%%%%%%%%%%%%%%%%%%%%%%%%%%%%%%%%%%%%%%%%%%%%%%%%%%%%%%%%%%%%%%%%%%%%%%%%%%%%%%%%%%%%%%%%%%%%%%%%%%%%%%%%%%%%%%%%%%%%%%%%%%%%%%%%%%%%%%%%%%%%%%%%%%%%%%%%%%%%%%%%%%%%%%%%%%%%%%%%%%%%%%%%%%%%%%%%%%%%%%%%%%%%%%%%%%%%%%%%%%%%%%%%%%%%%%%%%%%%%%%%%%%%%%%%%%%%%%%%%%%%%%%%%%%%%%%%%%%%%%%%%%%%%%%%%%%%%%%%%%%%%%%%%%%%%%%%%%%%%%%%%%%%%%%%%%%%%%%%%%%%%%%%%%%%%%%%%%%%%%%%%%%%%%%%%%%%%%%%%%%%%%%%%%%%%%%%%%%%%%%%%%%%%%%%%%%%%%%%%%%%%%%%%%%%%%%%%%%%%%%%%%%%%%%%%%%%%%%%%%%%%%%%%%%%%%%%%%%%%%%%%%%%%%%%%%%%%%%%%%%%%%%%%%%%%%%%%%%%%%%%%%%%%%%%%%%%%%%%%%%%%%%%%%%%%%%%%%%%%%%%%%%%%%%%%%%%%%%%%%%%%%%%%%%%%%%%%%%%%%%%%%%%%%%%%%%%%%%%%%%%%%%%%%%%%%%%%%%%%%%%%%%%%%%%%%%%%%%%%%%%%%%%%%%%%%%%%%%%%%%%%%%%%%%%%%%%%%%%%%%%%%%%%%%%%%%%%%%%%%%%%%%

\section{Preliminaries}\label{SecPrel}

\subsection{Conventions and notation}

Throughout, $\bk$ will denote an algebraically closed field. That $\bk$ be algebraically closed is not required for most statements, especially in Sections \ref{sec:indsys} to \ref{sec:ideals}. We set $\mN=\{0,1,2,\cdots\}$.

We will ignore certain set-theoretic subtleties, which could for instance be resolved by only considering tensor categories that are controlled in size by some cardinality, see \cite[Lemma~2.2.8]{CEO2}.

\subsubsection{}\label{sec:not} Sets of homomorphisms in a category $\cA$ are denoted by $\Hom_{\cA}(-,-)$, although for simplicity we sometimes simply write $\Hom(A,B)$ when clear that $A,B\in\cA$. We also abbreviate $\Hom_{\Vecc}$ to $\Hom_{\bk}$ and $\Hom_{\Rep G}$ to $\Hom_G$.

We identify the symmetric group $S_n$, for $n\in\mZ_{>0}$, with the permutation group of the set $\{1,2,\cdots,n\}$. This gives a canonical embedding $S_n\subset S_{n+1}$.
More generally, for a composition $\lambda\vDash n$, we consider the Young subgroup 
$S_\lambda<S_n$ of permutations that do not mix elements from different subsets in the partitioning $\{1,\ldots,\lambda_1\}\cup\{\lambda_1+1,\ldots,\lambda_1+\lambda_2\}\cup\cdots$. We follow the standard convention of denoting the simple $S_n$-modules in characteristic $p>0$ by $D^\lambda$, for $\lambda$ varying over $p$-regular partitions, see \cite[\S 11]{James-book}.

For an essentially small abelian category, we denote by $\Irr\cA$ and $\Indec\cA$ the sets of isomorphism classes of simple and indecomposable objects. Typically, for $i\in \Irr\cA$, we denote a representative by $L_i\in\cA$ and its projective cover by $P_i$.

\subsubsection{}\label{pacat}
A category $\cA$ is {\bf pseudo-abelian} if it is additive and idempotent-complete. A pseudo-abelian subcategory of $\cA$ is a full subcategory $\cB$ closed under taking direct sums and summands.

For two pseudo-abelian subcategories $\cB_1$ and $\cB_2$ of $\cA$, we denote by $\cB_1+\cB_2$ the pseudo-abelian subcategory of direct summands of direct sums of objects $M_1\oplus M_2$, with $M_i\in\cB_i$. We mostly deal with Krull-Schmidt categories, where we can leave out `direct summands of' in the previous sentence. Similarly, for a family $\{\cB_\alpha\subset\cA\}$ of pseudo-abelian subcategories, we write $\sum_\alpha\cB_\alpha$ for the pseudo-abelian subcategory of $\cA$ generated by (finite) direct sums of objects in the $\cA_\alpha$.

\subsection{Tensor categories}

\subsubsection{} An essentially small $\bk$-linear symmetric category
$(\cC,\otimes,\unit)$ is a {\bf tensor category over $\bk$} if
\begin{enumerate}
\item $\cC$ is abelian with objects of finite length;
\item $\bk\to\End_{\cC}(\unit)$ is an isomorphism;
\item $(\cC,\otimes,\unit)$ is rigid, meaning that every object $X$ has a monoidal dual $X^\ast$.
\end{enumerate}
Such categories are also sometimes called symmetric tensor categories or pretannakian categories. Note that $\unit$ is a simple object, see \cite[Theorem~4.3.8]{EGNO}. It then follows easily from the assumptions (1)-(3) that morphisms spaces in a tensor category are finite-dimensional. We denote the ind-completion of $\cC$ by $\Ind\cC$, and write $\Vecc^\infty=\Ind\Vecc$ for the category of all vector spaces over $\bk$.

The standard example of a tensor category is the category of finite dimensional (rational) representations $\Rep G=\Rep_{\bk}G$ of an abstract group $G$ or an affine group scheme $G$ over $\bk$.

A tensor functor between tensor categories is a $\bk$-linear symmetric monoidal exact functor. A tensor functor $F:\cC\to\cD$ is {\bf surjective} if every object in $\cD$ is a subquotient of an object in the essential image of $F$. A {\bf tensor subcategory} of a tensor category is a topologising rigid monoidal subcategory.

\subsubsection{} For a tensor category $\cC$ with objects $X_1,\cdots, X_n$, any $\sigma\in S_n$ gives an isomorphism (natural in $X_1,\cdots, X_n$)
$$\sigma_{X_1,\cdots, X_n}:\; X_1\otimes X_2\otimes \cdots \otimes X_n\;\xrightarrow{\sim}\;X_{\sigma^{-1}(1)}\otimes X_{\sigma^{-1}(2)}\otimes \cdots\otimes X_{\sigma^{-1}(n)}.$$
For $X\in\cC$, this yields algebra morphisms
\begin{equation}\label{BrMor}
\beta^n_X:\;\bk S_n\to \End_{\cC}(X^{\otimes n}),\quad\mbox{determined by}\quad \sigma\mapsto \sigma_{X,X,\cdots, X},\quad\mbox{for $\sigma\in S_n$}.
\end{equation}

\subsubsection{}\label{sec:proj}

We denote by $\Inj\cC$ the category of injective objects in $\Ind\cC$. Since $\Ind\cC$ is a Grothendieck category, every object has an injective hull. A tensor category~$\cC$ has a non-zero projective object if and only if it has enough projectives (since $X\otimes P$ is projective for $X\in\cC$ and projective $P\in\cC$) if and only if every object has a projective cover if and only if the indecomposable injective objects in $\Inj\cC$ actually belong to $\cC\subset\Ind\cC$ (since $P$ is projective if and only if $P^\ast$ is injective).  We then say that the tensor category `has projective objects' and write $\Proj\cC$ for the category of projective objects in $\cC$. In fact, an object in $\cC$ is projective if and only if it is injective, see \cite[Proposition~6.1.3]{EGNO}, although we will not use that.

For $i\in\Irr\cC$ denote the injective hull of $L_i$ in $\Ind\cC$ by $I_i$ and, should it exist, the projective cover of $L_i$ in $\cC$ by $P_i$.

For any associative algebra $A$ in $\cC$, we can consider its category $\Mod_{\cC}A$ of modules in~$\cC$, which are pairs of an object $Y\in\cC$ with an algebra morphism $A\to\uEnd(Y):=Y^\ast\otimes Y$, or equivalently an appropriate action map $A\otimes Y\to Y$.

\subsubsection{}\label{DefVer4} We refer to \cite{BE, BEO} for details on the incompressible tensor categories $\Ver_{p^n}^+$ and $\Ver_{p^n}$. We just mention here that, for $p=2$, as $\bk$-linear categories, we have
$$\Ver_4\;\simeq\; \bk[x]/x^2\mbox{-mod}\,\oplus\,\Vecc.$$
The regular $\bk[x]/x^2$-module $P$ is the projective cover of $\unit$ and $P\simeq V^{\otimes 2}$, where $V$ is the (projective) simple object in $\Ver_4$ corresponding to he copy of $\Vecc$ above.

\subsubsection{}
For groups $G,H$ and $V_1\in \Rep G$ and $V_2\in \Rep H$, in line with notation from Section~\ref{SecDTP} below, we write $V_1\boxtimes V_2$ for the $G\times H$-representation on $V_1\otimes V_2$. If $G=H$ and $V=V_1=V_2$, we can abbreviate this to $V^{\boxtimes 2}$, and similarly for higher powers.

%Perhaps immediately go to $\bk$-linear $\cA$ with faithful functor $F:\cA\to\Vecc$ and define topologising subcategory genreated by $F$ (automatic abelian). Remark that for $\cA=A\md$, all functors are subquotient of (possibly infinite) direct sum of forgetful.

\subsection{Deligne tensor product}\label{SecDTP}

\subsubsection{} Takeuchi's theorem, see \cite[Theorem~1.9.15]{EGNO}, states that every $\bk$-linear abelian category with finite dimensional morphism spaces and all objects of finite length is equivalent to the category of (finite dimensional) comodules over some coalgebra over $\bk$.

We denote by $\Tak_{\bk}$ the 2-category of $\bk$-linear categories as above, with 1-morphisms given by $\bk$-linear functors and 2-morphisms given by all natural transformations. We consider 2-subcategories
$$\Tak^{\ex}\subset \Tak^{\rex}\subset \Tak$$
 where we keep the same objects and 2-morphisms, but only consider (right) exact functors. 
 
 In \cite{Del90}, see also \cite{CF} for an overview, Deligne introduced the product $\cA\boxtimes \cB\in \Tak_{\bk}$ for $\cA,\cB$ in $\Tak_{\bk}$. This is a category equipped with a bilinear bifunctor 
\begin{equation}
\label{DefboxDel}-\boxtimes-:\;\cA\times\cB\to\cA\boxtimes\cB, \quad (X,Y)\mapsto X\boxtimes Y,    
\end{equation}
 satisfying the following universal property.
For every $\bk$-linear abelian category $\cC$, restriction along \eqref{DefboxDel} yields an equivalence between the category of right exact $\bk$-linear functors $\cA\boxtimes \cB\to\cC$ with the category of bilinear bifunctors $\cA\times\cB\to\cC$ that are right exact in each variable.

\subsubsection{}\label{Defpf} By construction, we obtain a pseudo-functor
$$-\boxtimes-:\; \Tak^{\rex}\times \Tak^{\rex}\,\to\, \Tak^{\rex}$$
which restricts, by \cite[Proposition~5.7]{Del90}, to a pseudo-functor
$$-\boxtimes-:\; \Tak^{\ex}\times \Tak^{\ex}\,\to\, \Tak^{\ex}.$$
A useful observation, stated formally in \cite[Proposition~3.2.6]{CF}, but applied earlier in \cite{CEO, EOf}, is that the latter can again be extended to a pseudo-functor
\begin{equation}\label{definitive}
-\boxtimes-:\; \Tak^{\ex}\times \Tak\,\to\, \Tak.
\end{equation}

It will be very convenient to choose a partially strict version of the pseudo-functor \eqref{definitive}. Concretely, we will assume that, for every $\cA\in\Tak$, the pseudo-functor
$$-\boxtimes \cA:\;\Tak^{\ex}\to\Tak$$
is actually a (strict) 2-functor and, moreover,
$-\boxtimes\Vecc$ is simply the inclusion $\Tak^{\ex}\subset\Tak$. We freely use the corresponding identities 
$$\cC\boxtimes\Vecc=\cC, \quad \Id_{\cC}\boxtimes\Id_{\cA}=\Id_{\cC\boxtimes\cA}\quad\mbox{and} \quad (G\circ F)\boxtimes \cA=(G\boxtimes\cA)\circ (F\boxtimes\cA).$$

To obtain this strict version, it suffices to choose, for every $\cA\in\Tak$, a faithful exact functor to $\Vecc$ (and choosing the identity functor for $\cA=\Vecc$) and correspondingly define $\cC\boxtimes\cA$ as a category of comodules in $\cC$ over the coalgebra over $\bk$ defined from $\cA\to\Vecc$. We refer to \cite{CF} for more details.
For convenience, we make the additional choice that in case $\cA=\Rep S_d$, the choice of $\cA\to\Vecc$ is simply the forgetful functor $\Frg$.

Our strictness assumption on \eqref{definitive} and the fact that it is a pseudo-functor between (strict) 2-categories allow us to simplify some of its coherence conditions:

\begin{lemma}\label{LemNonsense}
For $\Delta:\cA\to\cB$ in $\Tak$ and $F:\cC\to\cD$ in $\Tak^{\ex}$, the natural isomorphism
$$a(F,\Delta):\;(\cD\boxtimes \Delta)\circ (F\boxtimes \cA)\;\stackrel{\sim}{\Rightarrow}\; (F\boxtimes\cB)\circ (\cC\boxtimes \Delta)$$
obtained from the composition isomorphisms to and from $(F\boxtimes \Delta)$, satisfy
\begin{enumerate}
\item $a(\Id_{\cC},\Delta)=\Id_{\cC\boxtimes\Delta}$;
\item $a( G\circ F,\Delta)=(G\boxtimes\cA)(a(F,\Delta))\circ a(G,\Delta)_{F\boxtimes\cA}$, for every $G:\cD\to\cE$ in $\Tak^{\ex}$;
\item For every natural transformation $\zeta:F\Rightarrow G$, for $G:\cC\to\cD$ in $\Tak^{\ex}$, 
$$a(G,\Delta)\circ (\cD\boxtimes\Delta)(\zeta\boxtimes\cA)\;=\; (\zeta\boxtimes\cB)_{\cC\boxtimes\Delta}\circ a(F,\Delta).$$
\item For every natural transformation $\rho:\Delta\Rightarrow \Delta'$, for $\Delta':\cA\to\cB$ in $\Tak$,
$$(F\boxtimes\cB)(\cC\boxtimes\rho)\circ a(F,\Delta)\;=\; a(F,\Delta')\circ (\cD\boxtimes \rho)_{F\boxtimes \cA}.$$
\end{enumerate}
\end{lemma}

For $\cC$ and $\cD$ tensor categories, their product $\cC\boxtimes\cD$ is again a tensor product, see~\cite{Del90}. Furthermore, by restricting one of the arguments in $\cC\times\cD\to\cC\boxtimes\cD$ to $\Vecc$, we find embeddings of $\cC$ and $\cD$ into $\cC\boxtimes\cD$. By abuse of notation, we sometimes identify $\cC$ and $\cD$ with the equivalent tensor subcategories in $\cC\boxtimes\cD$.

\begin{example}\label{ExOm}Let $\cC$ be a tensor category.
\begin{enumerate}
\item By our choice of realisation of $-\boxtimes-$, the category $\cC\boxtimes\Rep S_d$ {\em is} the category of $S_d$-representations in $\cC$, that is of functors $B\,S_d\to \cC,$
where $B\,S_d$ is the one-object category corresponding to $S_d$. For example $\Vecc\boxtimes\Rep S_d=\Rep S_d$. 

For an $S_d$-representation $M$ in $\cC$ (an object in $\cC\boxtimes \Rep S_d$), we write $(M)^{S_d}$ resp. $(M)_{S_d}$, for the subobject, resp. quotient, of invariants resp coinvariants.

\item We can view $-^{\otimes n}$ as a (non-additive) symmetric monoidal functor
$$\RT_n=\RT_n^{\cC}:\;\cC\xrightarrow{-^{\otimes n}}\cC\boxtimes\Rep S_n,$$
by using the braid action \eqref{BrMor} of $S_n$ (`$R$' refers to representation and `$T$' to tensor).
\item For a fixed exact and faithful functor
$\Omega :\cC\to\Vecc,$
and $n\in\mZ_{>0}$ we consider the composite functor, 
$$\Omega^n:\; \cC\xrightarrow{\RT_n}\cC\boxtimes\Rep S_n\xrightarrow{\Omega\boxtimes\Rep S_n}\Rep S_n.$$
In particular, $\Omega^1=\Omega$.
\end{enumerate}

\end{example}

\subsection{Module categories and enriched categories}\label{ModEnr}
Let $\cC$ be a tensor category over~$\bk$.

\subsubsection{}
A (left) $\cC$-module category is a category $\cM$ in $\Tak_{\bk}$ equipped with a bilinear functor
$$-\circledast-:\; \cC\times\cM\;\to\; \cM,$$
exact in the first variable, and a natural isomorphism
$$(X\otimes Y)\circledast M\;\xrightarrow{\sim}\; X\circledast (Y\circledast M),$$
satisfying the conditions in \cite[\S 7.1]{EGNO}. In particular $\circledast$ is automatically exact in the second variable, so that we can view it as an exact functor
$\cC\boxtimes\cM\to \cM.$

Following \cite[\S 7.9]{EGNO}, for a $\cC$-module category $\cM$, we have a bifunctor 
$$\uHom(-,-):\;\cM^{\op}\times\cM\,\to\,\Ind\cC,$$
where $\uHom(M_1,M_2)$ represents the left exact functor
$$\Hom_{\cM}(-\circledast M_1,M_2):\;\cC^{\op}\to\Vecc.$$

\subsubsection{}\label{FinHom}
Let $\cM$ be a $\cC$-module category such that the internal hom takes values in $\cC\subset\Ind\cC$. Then we can associate to $\cM$ a $\cC$-enriched category $\uM$ with objects the same as $\cM$, but with morphism objects given by the internal homs. For example, the unit morphisms $\unit\to\uEnd(M)$ come from the identity under
$$\Hom_{\cC}(\unit,\uEnd(M))\;\simeq\;\Hom_{\cM}(\unit\circledast M,M)\simeq \End_{\cM}(M).$$

\begin{example}\label{SelfEnr}
The standard self-enrichment $\uC$ of $\cC$, with $\Ob\uC=\Ob\cC$ and
$$\Hom_{\uC}(X,Y)\;=\; \uHom(X,Y):=X^\ast\otimes Y$$
can be obtained in this way from the regular $\cC$-module.
\end{example}

Denote by $\Fun_{\cC}(-,-)$ the category of $\cC$-enriched functors between two $\cC$-enriched categories (so that $\Fun_{\Vecc}$ is just $\Fun_{\bk}$). We write $\Fun^{\cC}_{\bk}(-,-)$ for the category of $\cC$-module functors between two $\cC$-module categories, as defined in \cite[\S 7.2]{EGNO}.

\begin{example}
$\cC$-enriched categories with one object are the same thing as algebra objects in $\cC$, and we use the same notation for the interpretation of an algebra object as an enriched category. For an algebra object $R$ in $\cC$ and any $\cC$-module category $\cN$ the category
$$\Mod_{\cN}R\;:=\;\Fun_{\cC}(R,\uN)$$
is the category of modules $R\circledast N\to N$ of the monad $R\circledast-$ on $\cN$.
\end{example}

\begin{lemma}\label{LemModtoEnr}
For $\cC$-module categories $\cM,\cN$ as in \ref{FinHom}, there is a functor
$$\Fun^{\cC}_{\bk}(\cM,\cN)\;\to\; \Fun_{\cC}(\uM,\uN)$$
such that that for any $M\in\cM$, the triangle
$$\xymatrix{
\Fun^{\cC}_{\bk}(\cM,\cN)\ar[rr]\ar[rrd]&& \Fun_{\cC}(\uM,\uN)\ar[d]\\
&&\Mod_{\cN}\uEnd(M)
}$$
is commutative, where the vertical arrow is simply restriction onto the full subcategory of $\cM$ on the object $M$ and the diagonal arrow sends a $\cC$-module functor $F$ to the $\uEnd(M)$-module
$$\uEnd(M)\circledast F(M)\to F(M)$$
inherited from the $\uEnd(M)$-module structure of $M$.
\end{lemma}
\begin{proof}
Given a module functor, which is a functor $F:\cM\to\cN$ with natural transformation
$$s_{X,M}:F(X\circledast M)\to X\circledast F(M)$$
satisfying the conditions in \cite[\S 7.2]{EGNO}, we assign the enriched functor $F':\uM\to\uN$ as follows. For $M\in\cM$ we set $F'(X)=F(X)$, and the morphisms
$$\uHom(M_1,M_2)\to \uHom(F'(M_1),F'(M_2))$$ 
 in $\cC$ come from the following natural transformations:
$$\Hom_{\cM}(-\circledast M_1,M_2)\Rightarrow \Hom_{\cN}(F(-\circledast M_1),F(M_2))\Rightarrow \Hom_{\cN}(-\circledast F(M_1),F(M_2)),$$
where the first one comes functoriality of $F$ and the second one from $s$.

Given a morphism $\nu$ of $\cC$-module functors from $(F,s)$ to $(G,t)$, the morphisms 
$$\nu_M: F'(M)=F(M)\to G(M)=G'(M)$$
can be verified to form a natural transformation of $\cC$-enriched functors.

That these assignments produce a functor $\Fun^{\cC}(\cM,\cN)\to \Fun_{\cC}(\uM,\uN)$ with the commutative triangle is left as an exercise.
\end{proof}

\subsubsection{}\label{MoreNon}
As observed in \cite[Proposition~3.3.6]{CF}, when $\cC$ is a tensor category and $\Delta:\cA\to\cB$ a functor in $\Tak$, then $\cC\boxtimes\cA$ and $\cC\boxtimes\cB$ are canonically $\cC$-module categories, and  $\cC\boxtimes \Delta$ is a $\cC$-module functor. This simply follows by considering naturality and $\cC\boxtimes\cC\boxtimes \cA$. We thus obtain a functor
$$\cC\boxtimes-:\;\Fun_{\bk}(\cA,\cB)\to \Fun_{\bk}^{\cC}(\cC\boxtimes\cA,\cC\boxtimes\cB).$$

%%%%%%%%%%%%%%%%%%%%%%%%%%%%%%%%%%%%%%%%%%%%%%%%%%%%%%%%%%%%%%%%%%%%%%%%%%%%%%%%%%%%%%%%%%%%%%%%%%%%%%%%%%%%%%%%%%%%%%%%%%%%%%%%%%%%%%%%%%%%%%%%%%%%%%%%%%%%%%%%%%%%%%%%%%%%%%%%%%%%%%%%%%%%%%%%%%%%%%%%%%%%%%%%%%%%%%%%%%%%%%%%%%%%%%%%%%%%%%%%%%%%%%%%%%%%%%%%%%%%%%%%%%%%%%%%%%%%%%%%%%%%%%%%%%%%%%%%%%%%%%%%%%%%%%%%%%%%%%%%%%%%%%%%%%%%%%%%%%%%%%%%%%%%%%%%%%%%%%%%%%%%%%%%%%%%%%%%%%%%%%%%%%%%%%%%%%%%%%%%%%%%%%%%%%%%%%%%%%%%%%%%%%%%%%%%%%%%%%%%%%%%%%%%%%%%%%%%%%%%%%%%%%%%%%%%%%%%%%%%%%%%%%%%%%%%%%%%%%%%%%%%%%%%%%%%%%%%%%%%%%%%%%%%%%%%%%%%%%%%%%%%%%%%%%%%%%%%%%%%%%%%%%%%%%%%%%%%%%%%%%%%%%%%%%%%%%%%%%%%%%%%%%%%%%%%%%%%%%%%%%%%%%%%%%%%%%%%%%%%%%%%%%%%%%%%%%%%%%%%%%%%%%%%%%%%%%%%%%%%%%%%%%%%%%%%%%%%%%%%%%%%%%%%%%%%%%%%%%%%%%%%%%%%%%%%%%%%%%%%%%%%%%%%%%%%%%%%

\section{Inductive systems of symmetric group representations}\label{sec:indsys}
Recall that we work over a fixed (algebraically closed) field $\bk$.

\subsection{Inductive systems}

\label{sec:defis}

\begin{definition}\label{DefIS}
An {\bf inductive system} is the assignment of a pseudo-abelian subcategory $\mA^n\subset \Rep S_n$, for each $n\in\mZ_{>0}$, so that $\mA^{n-1}\subset \Rep S_{n-1}$ is the minimal pseudo-abelian subcategory in which the composite
$$\mA^n\;\hookrightarrow\;\Rep S_n\;\xrightarrow{\Res^{S_n}_{S_{n-1}}}\Rep S_{n-1}$$
takes values.
An inductive system $\mA$ is {\bf semisimple} if $\mA^n$ comprises only semisimple representations for all $n$.
\end{definition}

\begin{remark}\label{rem:Zal}
\begin{enumerate}
\item Clearly an inductive system could be defined in terms of (isomorphism classes of) indecomposable modules, rather than subcategories.
\item Using the point of view in (1), our notion of a semisimple inductive system becomes equivalent with Kleshchev's \cite[Definition~0.2]{CS}.
\item There is a different notion of `inductive system' due to Zalesskii \cite[Definition~1.1]{Za} (of which Kleshchev's definition is again a special case), which is a collection of simple representation $\Phi^n$ for each $\bk S_n$ such that the representations in $\Phi^{n-1}$ are precisely those that are simple constituents of the restrictions of those in $\Phi^n$.
Hence, taking the collection of simple constituents of the modules in our definition yields an inductive system in the sense of \cite{Za}.
\end{enumerate}
\end{remark}

\subsubsection{}
We write $\mA\subset\mB$ for two inductive systems if $\mA^n$ is a subcategory of $\mB^n$ for all $n\in\mZ_{>0}$.
For a family of inductive systems $\{\mA_\alpha\mid\alpha\}$, we denote by $\sum_\alpha\mA_\alpha$ the inductive system with $(\sum_\alpha\mA_\alpha)^n=\sum_\alpha\mA_\alpha^n$, see \ref{pacat}.

For a given inductive system $\mA$, we can define the inductive system $\mI(\mA)$, where for each $n\in\mN$ the category~$\mI(\mA)^n $ comprises the direct summands of $S_n$-representations of the form
$$\Ind^{S_n}_{S_\lambda}(V_1\boxtimes V_2\boxtimes \cdots \boxtimes V_l),$$
where $\lambda\vDash n$ is a composition of length $l$ and~$V_i\in\mA^{\lambda_i}$. That $\mI(\mA)$ is again an inductive system follows from Mackey's Theorem. It is immediate that this procedure is idempotent, {\it i.e.} $\mI(\mI(\mA))=\mI(\mA)$.

\begin{example}\label{ExYoung}
The most obvious non-zero inductive system $\mA$ is defined by letting $\mA^n$ be the category of trivial $S_n$-representations. 

In this case we set $\Young:=\mI(\mA)$. Hence $\Young^n$ is the category of direct sums of Young modules, that is direct summands of direct sums of permutation modules $M^{\lambda}=\Ind^{S_n}_{S_\lambda}\unit$ for compositions $\lambda\vDash n$.
\end{example}

\begin{problem}
In \cite{BK}, the minimal inductive systems in Zalesskii's sense were classified for $\mathrm{char}(\bk)>2$. They are precisely the semisimple inductive systems $\bC\mS(s)$ in \Cref{ThmKl} below. These are also minimal inductive systems in our sense, but no longer the only ones. For example the inductive system of projective objects does not contain any of the $\bC\mS(s)$. It would be interesting to classify minimal inductive systems in the sense of Definition~\ref{DefIS}.
\end{problem}

\subsection{Semisimple inductive systems} If $\mathrm{char}(\bk)=0$ then all inductive systems are semi\-simple, and they are in bijection with the set of ideals in the poset of partitions (for the inclusion order) that are `minimal' in the sense that removing any one partition leaves a subset which is no longer an ideal.
Semisimple inductive systems in positive characteristic were classified by Kleshchev in \cite{CS}.

\subsubsection{}Following \cite[Definition~0.1]{CS}, a simple $S_n$-representation is {\bf completely splittable} (CS) if its restriction to any Young subgroup $S_\lambda <S_n$ is semisimple. CS representations were classified in \cite{CS}, but we only require the following results.

\begin{theorem}[Kleshchev]Let $\bk$ be a field of characteristic $p>0$.\label{ThmKl}
\begin{enumerate}
\item A simple $S_n$-representation is completely splittable if and only if it belongs to a semisimple inductive system.
\item There are $p-1$ minimal semisimple inductive systems, labelled $\bC\mS(s)$, $s\in\{1,\cdots, p-1\}$. All other semisimple inductive systems are sums of the minimal ones.
\item $\bC\mS^n(1)$ comprises precisely the trivial $S_n$-representations.
\item If $p>2$, then $\bC\mS^n(p-s)$ comprises precisely the tensor product of the sign representation with the representations in $\bC\mS^n(s)$.
\item If $s>1$, then $\bC\mS^n(s)$ contains the trivial representation if and only if $n+s\le p$. 
\item For $n\ge p$, the intersection
$\bC\mS^n(s)\cap\bC\mS^n(t),$ with $s\not=t$,
contains only the zero representation.
\end{enumerate}
\end{theorem}
\begin{proof}
One direction in (1) is immediate, the other direction follows from \cite[Theorems~2.1 and 2.8(i)]{CS}. Part (2) follows from \cite[Theorem~2.8]{CS}. Parts (3), (5) and (6) follow immediately from the definition on p590 of \cite{CS}. Finally, taking the tensor product with the sign representation clearly produces an involution on the set of semisimple inductive systems. Tracing that this involution is as in (4) is easy using the explicit descriptions in \cite{CS}.
\end{proof}

\subsection{Closed inductive systems}

\subsubsection{}\label{properties}
We list some potential properties of an inductive system~$\bA$:
\begin{itemize}
\item[(i)] Each $\bA^n\subset\Rep S_n$ is closed under (internal) tensor products;
\item[(ii)] Each $\bA^n\subset\Rep S_n$ is closed under duality $V\mapsto V^\ast$;
\item[(iii)] Each $\bA^n\subset\Rep S_n$ contains the trivial representation $\unit$;
\item[(iv)]  $\mI(\mA)=\mA$ or, equivalently, for all $m,n\in\mZ_{>0}$, the induction product
$$\mA^m \times\mA^n\xrightarrow{\Ind^{S_{m+n}}_{S_m\times S_n}(-\boxtimes -)}\Rep S_{m+n}$$
takes values in $\mA^{m+n}$.

\end{itemize}

\begin{definition}\label{DefIS2}
An {inductive system} $\mA$ is {\bf closed} if it satisfies \ref{properties}(i)-(iv). In other words $\mA$ is closed if $\bI(\mA)=\mA$ and each $\bA^n\subset\Rep S_n$ is a rigid monoidal subcategory.
\end{definition}

\begin{example}\label{ExIndSys2}
\begin{enumerate}
\item The unique {\em maximal} (closed) inductive system corresponds to $\mA^n=\Rep S_n$ for all $n\in\mZ_{>0}$.
\item The unique {\em minimal} closed inductive system is $\Young$ from Example~\ref{ExYoung}.
\item By the combination of (1) and (2), if $\mathrm{char}(\bk)=0$ then the only closed inductive system is $\Young$. Similarly, if $\mathrm{char}(\bk)=p>0$, then
$$\Young^n\;=\;\mA^n\;=\;\Rep S_n,\quad\mbox{for }\,n<p,$$
for every closed inductive system $\mA$.
\item The {\bf signed Young modules} are the indecomposable direct summands of the $\bk S_n$-modules which are induced from one-dimensional representations of Young subgroups (which are exterior tensor products of trivial and sign representations). We have the corresponding closed inductive system $\SYoung$, with $\SYoung=\Young$ if $\mathrm{char}(\bk)=2$. For $p>2$, the signed Young modules were classified by Donkin in \cite{Donkin}.

\item Assume $\mathrm{char}(\bk)=p>0$. Another closed inductive system is $p\Perm$, where $p\Perm^n$ consists of the `$p$-permutation modules' of $S_n$, see \cite{Br}. These can be characterised equivalently as:
\begin{enumerate}
\item the modules that have a $P$-invariant basis for every $p$-subgroup $P<S_n$;
\item the direct sums of direct summands of representations $\Ind^{S_n}_H\unit$, for subgroups $H<S_n$;
\item the direct summands of permutation modules.
\end{enumerate}
\item Assume $\mathrm{char}(\bk)=p>0$. We call a $\bk S_n$-module $M$ {\bf unentangled} if, for every Young subgroup $S_\lambda<S_n$, the module $\Res^{S_n}_{S_\lambda}M$ is a direct summand of an exterior tensor product
$$M_1\boxtimes M_2\boxtimes\cdots \boxtimes M_{l},\qquad\mbox{for certain $M_i\in \Rep S_{\lambda_i}$.}$$
 Completely splittable modules are obvious examples of unentangled modules.
By Mackey's formula, unentangled modules constitute a closed inductive system, which we denote by $\UnEn$. It follows that $\UnEn^{n}=\Rep S_n$ if and only if $n<2p$.
\end{enumerate}
\end{example}

%\begin{remark}\label{RemCIS}
%
%
% If $\mathrm{char}(\bk)=p>0$, $m\in\mN$ is not $p$-divisible and $\mA$ is an inductive system satisfying \ref{properties}(iv), {\em e.g.} $\mA$ is closed, then $\mA^{m-1}$ determines $\mA^m$. Indeed, every summand of $\Ind^{S_m}_{S_{m-1}}M$ with $M\in\mA^{m-1}$ is in $\mA^m$. But also conversely, every $N\in\mA^m$ is a summand of $\Ind^{S_m}_{S_{m-1}}\Res^{S_m}_{S_{m-1}}N$ and $\Res^{S_m}_{S_{m-1}}N$ is in $\mA^{m-1}$ by Definition~\ref{DefIS}.
%
%
%\end{remark}

\subsection{Some inclusions between inductive systems}
Assume $\mathrm{char}(\bk)=p>0$.

\subsubsection{} Let $\bC\mS$ be the semisimple inductive system $\sum_{1\le s<p}\bC\mS(s)$ of all CS representations and let $\bC\mS_+$ be the subsystem of $\bC\mS$ corresponding to odd $1\le s<p$ only. Then we set
$$\ICS:=\mI(\bC\mS)\quad\mbox{and}\quad\ICS_+:=\mI(\bC\mS_+).$$

\begin{prop}\label{PropInclusions}
We have inclusions
$$\xymatrix{
\Young\;\ar@{^{(}->}[r]\ar@{^{(}->}[d] & \SYoung\; \ar@{^{(}->}[r]\ar@{^{(}->}[d] & p\Perm\\
\ICS_+\;\ar@{^{(}->}[r]&\ICS\;\ar@{^{(}->}[r]&\UnEn.
}$$
No other inclusions exist between the displayed systems, except that the square consists of equalities if $p=2$, and the downward arrows are equalities if $p=3$.
\end{prop}

\begin{proof}
All inclusions except $\SYoung\subset p\Perm$ are by definition. To prove the latter we can observe that the inductive system $\bA$ comprising trivial and sign modules is included in $p\Perm$ and use that $\bI(\bA)=\SYoung$ while $\bI(p\Perm)=p\Perm$.

Now we prove that there are no other inclusions than the ones stated in the lemma.

Consider first $p=2$. In this case we only need to prove that $\Young\subset 2\Perm$ is strict. Let $P<S_4$ be a 2-Sylow subgroup. It follows easily that
$\Ind^{S_4}_P\unit$
is a direct sum of the two simple $S_4$-modules, while the only simple module in $\Young^4$ is the trivial one, see~\cite{EW}.

Now assume $p\ge 3$. The inclusions
$\Young^p\subset \SYoung^p$ and $\ICS_+^p\subset \ICS^p$
are strict, as $\ICS_+^p$ does not contain the sign module.

For $p=3$, it only remains to prove that $3\Perm$ is not included in $\SYoung$. For this we can consider the $3$-permutation module $\Ind^{S_4}_{C_3}\unit$, which is 8-dimensional and has every simple $S_4$-representation in its socle. Thus by dimension count, $\Ind^{S_4}_{C_3}\unit$ is a direct sum of the four irreducible representations. However, not all irreducible representation are Young modules.

For the rest of the proof we can thus assume that $p>3$. 
By \cite[\S1]{Donkin}, the number of isomorphism classes of indecomposable modules in $\SYoung^p$ is one plus the number of partitions of $p$. Hence, these must be the trivial module, the sign module and the indecomposable projectives, whereas the modules in $\ICS^p$ are all the simple modules and the indecomposable projectives. This shows that the inclusion $\SYoung\subset\ICS$ is strict. For $\Young\subset\ICS_+$ we can argue similarly.

We can observe that $\ICS_+^p$ is not included in $p\Perm^p$. Indeed, such an inclusion would imply that all simple modules in $\ICS_+^p$ are permutation modules over $C_p$. As the latter group does not detect tensoring with the sign module, this would mean that all simple $S_p$-modules are permutation modules over $C_p$. Take the simple $S_p$-module corresponding to the partition $(p-1,1)$. It has dimension $p-2<p$, so it can only be a $p$-permutation module if it is trivial over $C_p<S_p$, which is not true.

Next we can observe that $p\Perm^p$ is not included in $\ICS^p$. Indeed, the $p$-permutation module $\Ind^{S_p}_{C_p}\unit$ contains every simple $S_p$-representation in its socle. However, since this modules does not contain projective summands it would have to be semisimple in order to be in $\ICS^p$. But, as explained in the previous paragraph, not every simple $S_p$-representation is a $p$-permutation module, so $\Ind^{S_p}_{C_p}\unit$ is not semisimple and thus not in $\ICS^p$.

Finally, we show that there are no inclusions between $p\Perm$ and $\UnEn$. Indeed, $p\Perm^p$ is not equal to $\Rep S_p=\UnEn^p$. On the other hand, consider subgroups $S_p<S_p\times S_p<S_{2p}$ where the first inclusion is diagonal and the second is a Young subgroup. Then the permutation module spanned by coset $S_{2p}/S_p$ is not unentangled, since its restriction to $S_p\times S_p$ contains a direct summand $\Ind^{S_p\times S_p}_{S_p}\unit$ (which is projective for both copies of $S_p$ in $S_p^{\times 2}$ but not projective for $S_p^{\times 2}$).
\end{proof}

\begin{remark}\label{rem:min}
While the intersection of two inductive systems need not be an inductive system, the intersection of a family of closed inductive systems remains a closed inductive system (as condition (iv) supersedes the problematic property of inductive systems). We can thus speak of the minimal closed inductive system containing a class of representations. For example, the minimal closed inductive system containing the sign representation of $S_p$ is a closed inductive system lying strictly between $\Young$ and $\SYoung$, as it does not contain the sign representation of $S_{p^2}$.
\end{remark}

%%%%%%%%%%%%%%%%%%%%%%%%%%%%%%%%%%%%
%%%%%%%%%%%%%%%%%%%%%%%%%%%%%%%%%%%%
%%%%%%%%%%%%%%%%%%%%%%%%%%%%%%%%%%%%
%%%%%%%%%%%%%%%%%%%%%%%%%%%%%%%%%%%%
%%%%%%%%%%%%%%%%%%%%%%%%%%%%%%%%%%%%
%%%%%%%%%%%%%%%%%%%%%%%%%%%%%%%%%%%%
%%%%%%%%%%%%%%%%%%%%%%%%%%%%%%%%%%%%
%%%%%%%%%%%%%%%%%%%%%%%%%%%%%%%%%%%%
%%%%%%%%%%%%%%%%%%%%%%%%%%%%%%%%%%%%
%%%%%%%%%%%%%%%%%%%%%%%%%%%%%%%%%%%%
%%%%%%%%%%%%%%%%%%%%%%%%%%%%%%%%%%%%
%%%%%%%%%%%%%%%%%%%%%%%%%%%%%%%%%%%%
%%%%%%%%%%%%%%%%%%%%%%%%%%%%%%%%%%%%
%%%%%%%%%%%%%%%%%%%%%%%%%%%%%%%%%%%%
%%%%%%%%%%%%%%%%%%%%%%%%%%%%%%%%%%%%
%%%%%%%%%%%%%%%%%%%%%%%%%%%%%%%%%%%%

\section{Inductive systems from tensor categories}

Let $\cC$ be a tensor category over $\bk$. 

\subsection{Representations from the braid action}

\subsubsection{} For $X\in \cC$ and $n\in\mN$, we have an anti-algebra morphism, coming from the anti-autoequivalence $-^\ast$, fitting into a commutative diagram
\begin{equation}\label{commsq}
\xymatrix{
\End(X^{\otimes n})\ar[rr]^\sim && \End((X^\ast)^{\otimes n})\\
kS_n\ar[u]^{\beta_X^n}\ar[rr]^\sim_{S_n\ni g\mapsto g^{-1}}&&kS_n\ar[u]_{\beta^n_{X^\ast}}.
}
\end{equation}

For $I\in\Inj\cC$, the space $\Hom((X^\ast)^{\otimes n},I)$ is a right $\End((X^\ast)^{\otimes n})$-module and hence a (left) $S_n$-representation via either path in \eqref{commsq}.
If $\cC$ has projective objects, see \ref{sec:proj}, these $S_n$-representations can simply be defined as
$$\Hom(P,X^{\otimes n}),\quad P\in\Proj\cC,$$
as that space is naturally isomorphic to $\Hom((X^\ast)^{\otimes n},P^\ast)$.
To simplify formulas (without changing the content substantially) we sometimes make the non-essential assumption that $\cC$ has projective objects.

\begin{definition}\label{DefPhiX}
For $X\in \cC$ and $n\in\mZ_{>0}$, denote by 
$$\mB^n_X=\mB^n_X[\cC]\;\subset\;\Rep S_n$$ the pseudo-abelian subcategory of direct summands of the $S_n$-representations 
$$\Hom((X^\ast)^{\otimes n}, I),\quad I\in \Inj\cC.$$
Equivalently, $\mB^n_X[\cC]\subset\Rep S_n$ is the pseudo-abelian subcategory generated by $\Omega^n(X)$, with notation and assumptions as in Example~\ref{ExOm}.

\end{definition}

\begin{theorem}\label{ThmTCIS}
\begin{enumerate}
\item For each $X\in\cC$, the categories $\{\mB^n_X\mid n\in\mZ_{>0}\}$ form an inductive system~$\mB_X=\mB_X[\cC]$.

\item For a full subcategory $\cA\subset\cC$, take the inductive system $\mB_{\cA}[\cC]:=\sum_{X\in\cA}\mB_X[\cC]$.
\begin{enumerate}
\item If $\cA$ is closed under tensor products, then $\mB_{\cA}[\cC]$ satisfies \ref{properties}(i).
\item If $\cA$ is closed under $X\mapsto X^\ast$, then $\mB_{\cA}[\cC]$ satisfies \ref{properties}(ii).
\item If $\unit\in\cA$, then $\mB_{\cA}[\cC]$ satisfies \ref{properties}(iii).
\item If $\cA$ is additive, then $\mB_{\cA}[\cC]$ satisfies \ref{properties}(iv).
\end{enumerate}

\item The inductive system $\mB[\cC]:=\mB_{\cC}[\cC]=\sum_{X\in \cC}\mB_X[\cC]$ is closed.
\item If $\cC$ is semisimple, then $$\mB[\cC]\;=\;\mathbf{I}\left(\sum_{L\in \Irr\cC}\mB_L\right).$$
\item We have $\bB[\cC]\subset\UnEn$.
\end{enumerate}
\end{theorem}
This theorem will be proved in the next section.

\begin{example}\label{ExProj}
If $\cC$ has projective objects, we set
$$\mB_{\cP r}[\cC]:=\mB_{\Proj\cC}[\cC].$$
By Theorem~\ref{ThmTCIS}(2) (and \cite[Proposition~6.1.3]{EGNO}), this inductive system satisfies \ref{properties}(i), (ii) and (iv). Moreover, we will see in Theorem~\ref{CorFrEx} that $\mB_{\cP r}[\cC]$ is closed if and only if~$\cC$ is Frobenius-exact.
\end{example}

An important property of the above definitions, is invariance under tensor functors:

\begin{prop}\label{PropF}
Consider a tensor functor $F:\cC\to\cD$.
\begin{enumerate}
\item For $X\in\cC$, we have
$\mB_{F(X)}[\cD]\;=\;\mB_X[\cC].$
\item We have
$\mB[\cC]\subset \mB[\cD].$
\item If $\cC$ is a finite tensor category and $F$ is surjective, then
$\mB_{\cP r}[\cC]=\mB_{\cP r}[\cD]$.
\end{enumerate}
\end{prop}
\begin{proof}
Let $F_\ast:\Ind\cD\to\Ind\cC$ be the right adjoint of $F$. Since $F$ is exact, $F_\ast$ sends injective objects to injective objects. It also follows from faithfulness of $F$ that every indecomposable injective object in $\Ind\cC$ is a direct summand of $F_\ast(I)$ for some $I\in\Inj\cD$. Part (1) therefore follows from 
$$\Hom((F(X)^\ast)^{\otimes n},I)\simeq \Hom(F((X^\ast)^{\otimes n}),I)\;\simeq\; \Hom((X^\ast)^{\otimes n}, F_\ast(I)).$$

Part (2) is an immediate consequence of part (1). Part (3) follows from part (1) and \cite[Theorem~6.1.16]{EGNO}.
\end{proof}

\begin{remark}\label{RemF}\label{RemTCIS}
\begin{enumerate}
\item 
It will follow from Corollary~\ref{CorFrEx4} that Theorem~\ref{ThmTCIS}(4) remains valid in Frobenius-exact tensor categories. However, it does not extend beyond Frobenius-exact categories, see Remark~\ref{RemVer4}.
\item It is tempting to expect an equality $\mB[\cC]= \mB[\cD]$ in Proposition~\ref{PropF}(2) for surjective tensor functors $F$ (and we currently know of no counterexamples). When $\cD$ is semisimple, this is indeed the case. More generally, in Corollary~\ref{CorFrEx3}(2), we will show that when $\cD$ is Frobenius-exact and $F$ is surjective, then $\mB[\cC]= \mB[\cD]$.
\item If $X\in\cC$ is sent to a non-simple object by some tensor functor, then it follows as an application of \Cref{PropF}(1) that $\mB^{p^i}_X$ must contains a projective module for every $i\in\mN$.
\end{enumerate}

\end{remark}

\begin{example}\label{Ex1}
\begin{enumerate}
\item By Proposition~\ref{PropF}(2), we have $\mB[\Rep G]=\Young$ for any affine group scheme $G$ over $\bk$.
\item The Delannoy category $\cD$ from \cite{HSS} satisfies $\mB[\cD]=\Young$. For the interested reader familiar with \cite{HSS}, we sketch the proof. As $\cD$ is semisimple, we can reduce this claim via Corollary~\ref{cor:gen} to the observation
$$\mB_{\mathfrak{C}(\mR)}^n\;=\;\Young^n,$$
with $\mathfrak{C}(\mR)$ the generating \'etale algebra. The latter follows from the fact that powers $\mR^{\times n}$, when considered as sets with an action of the product of the oligomorphic group $\mG=\Aut(\mR,<)$ and $S_n$ factor  as in the example
$$\mR^{\times 4}\;=\; \mR^{(4)}\times S_4\;\sqcup\;\mR^{(3)}\times S_4/S_2\;\sqcup\;  \mR^{(2)}\times (S_4/S_3\sqcup S_4/S_{2,2})\; \sqcup \;\mR,$$
showing that $\mathfrak{C}(\mR)^{\otimes 4}$ in $\cD\boxtimes\Rep S_4$ is isomorphic to
$$\mathfrak{C}(\mR^{(4)})\boxtimes M^{(1,1,1,1)}\;\oplus\;\mathfrak{C}(\mR^{(3)})\boxtimes M^{(2,1,1)}\;\oplus\;  \mathfrak{C}(\mR^{(2)})\boxtimes (M^{(3,1)}\oplus M^{(2,2)})\; \oplus \;\mathfrak{\mR}\boxtimes\unit.$$
\end{enumerate}
\end{example}

\begin{question}\label{QBk}
\begin{enumerate}
\item We define the closed inductive system $\bT$ with $\bT^n=\sum_{\cC}\mB^n[\cC]$, where the sum ranges over all tensor categories~$\cC$ over $\bk$. If $\mathrm{char}(\bk)=0$, then clearly $\bT^n=\Rep S_n$, for all $n$. In positive characteristic, describing $\bT$ is an open problem. By Theorem~\ref{ThmTCIS}, we have $\bT\subset\UnEn$. Is it true that
$\bT=\UnEn$?
\item Using terminology from \cite{HS} and generalising Example~\ref{Ex1}(2), is it true that
$\mB[\underline{\Rep}(G;\mu)]=\Young,$
for every pro-oligomorphic group $G$ with quasi-regular measure $\mu$ satisfying condition~(P)?
\end{enumerate}

\end{question}

\subsection{The proof of Theorem~\ref{ThmTCIS}}

\begin{lemma}\label{LemDual}
Let $L\in\cC$ be simple and denote the injective hulls of $L$ and $L^\ast$ in $\Ind\cC$ by $I$ and $I'$. Then, for each $n\in\mN$ and $X\in\cC$, we have an isomorphism of $S_n$-representations
$$\Hom((X^\ast)^{\otimes n}, I)^\ast\;\simeq\;\Hom(X^{\otimes n}, I').$$
In particular $\mB^n_{X^\ast}$ comprises the duals of the modules in $\mB^n_X\subset\Rep S_n$.
\end{lemma}
\begin{proof}
Consider the exact functor
$$\Hom(-^\ast,I)^\ast:\;\cC^{\op}\to\Vecc.$$
This must be representable by an injective ind-object, which is easily identified as $I'$. In particular, we find an isomorphism
$$\Hom((X^\ast)^{\otimes n},I)^\ast \;\simeq\; \Hom(X^{\otimes n}, I')$$
of right $\End(X^{\otimes n})$-modules. The conclusion follows from diagram~\eqref{commsq}.
\end{proof}

\subsubsection{Construction}\label{Construction} We construct natural isomorphisms which will be crucial for the rest of the paper. To keep notation light we work with tensor categories with projective objects. The extension to the general case via injective objects is immediate.

Fix $t\in\mZ_{>0}$ and $P\in\Proj\cC$. For every $t$-tuple $(i_1,\cdots, i_t)$ in $\Irr\cC$ we choose a section $s_{i_1,\cdots, i_t}$ of the linear surjection (using the conventions from \S\ref{sec:not})
$$\Hom(P, P_{i_1}\otimes\cdots\otimes P_{i_t})\;\tto\; \Hom(P, L_{i_1}\otimes \cdots\otimes L_{i_t}).$$
We have a natural transformation
$$\bigoplus_{i_1,\cdots,i_t\in \Irr\cC}\Hom(P,P_{i_1}\otimes\cdots\otimes P_{i_t})\otimes_{\bk} \Hom(P_{i_1},-)\otimes_{\bk}\cdots\otimes_{\bk}\Hom(P_{i_t},-)\;{\Rightarrow}\;\Hom(P,-\otimes  \cdots\otimes-) $$
of functors $\cC^{\times t}\to\Vecc$, which is simply given by
$$g\otimes f_1\otimes \cdots \otimes f_t\;\mapsto \; (f_1\otimes \cdots\otimes f_t)\circ g,$$
where it is understood that $f_1\otimes \cdots \otimes f_t$ refers both to the tensor product over $\bk$ of morphism spaces and the tensor product in $\cC$ of morphisms.

Composing with the chosen sections then produces 
a natural transformation
\begin{equation}\label{EqNat}
\bigoplus_{i_1,\cdots,i_t}\Hom(P,L_{i_1}\otimes\cdots\otimes L_{i_t})\otimes_{\bk} \Hom(P_{i_1},-)\otimes_{\bk}\cdots\otimes_{\bk}\Hom(P_{i_t},-)\;\stackrel{\sim}{\Rightarrow}\;\Hom(P,-\otimes \cdots\otimes-), 
\end{equation}
of functors $\cC^{\times t}\to \Vecc$, which we claim to be an isomorphism. Indeed, since the involved functors are exact in each variable, it suffices to verify that the natural transformation produces an isomorphism on each $t$-tuple of simple objects, which is by definition.

\begin{lemma}\label{LemProd}
For $X,Y\in\cC$, the $S_n$-representations in $\mB_{X\otimes Y}^n$ are precisely the direct summands of representations $M\otimes N$, with $M\in\mB^n_X$ and $N\in\mB^n_Y$.
\end{lemma}
\begin{proof}
For simplicity of notation, we assume that $\cC$ has projective objects.
We apply \eqref{EqNat} for $t=2$ to produce an isomorphism, for $P\in\Proj\cC$,
$$\bigoplus_{i,j\in \Irr\cC}V_{ij}\otimes_{\bk}\Hom(P_i,X^{\otimes n})\otimes_{\bk}\Hom(P_j,Y^{\otimes n})\;\xrightarrow{\sim}\;\Hom(P,(X\otimes Y)^{\otimes n}),$$
of $\End(X^{\otimes n})\otimes_{\bk}\End(Y^{\otimes n})$-modules (for vector spaces $V_{ij}$), from which the claim follows quickly.
\end{proof}

\begin{lemma}\label{RemKappa}
 For every $P\in\Proj\cC$, $n\in\mZ_{>0}$ and $\kappa\vDash n$ of length $l$, take $\{Z_i\in\cC\mid 1\le i\le l\}$.
Then the $S_\kappa$-representation 
$$\Hom(P,Z_1^{\otimes \kappa_1}\otimes \cdots\otimes  Z_l^{\otimes\kappa_l})$$ is a direct sum of 
modules 
$$M_1\boxtimes M_2\boxtimes \cdots \boxtimes M_l\quad \mbox{with} \quad M_i\in \mB^{\kappa_i}_{Z_i}.$$
\end{lemma}
\begin{proof}
This is a direct application of \eqref{EqNat}.
\end{proof}

\begin{proof}[Proof of Theorem~\ref{ThmTCIS}]
Part (5) follows immediately from Lemma~\ref{RemKappa}. For part (1), we assume that $\cC$ has projective objects. For $X\in\cC$ and $P\in\Proj\cC$, the natural transformation \eqref{EqNat} for $t=2$ and evaluated at $(X^{\otimes n-1},X)$ produces an isomorphism 
$$\bigoplus_{i\in \Irr\cC} V_i\otimes_{\bk} \Hom(P_i,X^{\otimes n-1})\;\xrightarrow{\sim}\; \Hom(P, X^{\otimes n})$$
of $S_{n-1}$-representations, for some vector spaces $V_i$.

Part (2)(a) follows from \Cref{LemProd}. Part (2)(b) follows from \Cref{LemDual}, Part (2)(c) is obvious. Part (2)(d) follows from the observation that we have an isomorphism in $\Rep S_n$
\begin{equation}\label{eq:oplus}
\Hom(P, (X\oplus Y)^{\otimes n})\;\simeq\; \bigoplus_{l=0}^n \Ind^{S_n}_{S_l\times S_{n-l}}\Hom(P,X^{\otimes l}\otimes Y^{\otimes n-l}),
\end{equation}
and Lemma~\ref{RemKappa}.
This also proves part (4). 

Finally, part (3) is a special case of part (2).
\end{proof}

\begin{corollary}\label{cor:gen}
Let $X_\alpha$ be a collection of objects so that each indecomposable object in $\cC$ is a direct summand of some tensor product of the $X_\alpha$. Then $\bB[\cC]$ is the minimal closed inductive system (see Remark~\ref{rem:min}) containing every $\bB_{X_\alpha}$.
\end{corollary}
\begin{proof}
We only need to show that the minimal inductive system $\bB$ containing every $\bB_{X_\alpha}$ contains $\bB_Y$ for an arbitrary $Y\in\cC$. By equation~\eqref{eq:oplus} and Lemma~\ref{RemKappa} it suffices to consider indecomposable $Y$. Then $Y$ is a direct summand of a tensor product $X'$ of the $X_\alpha$, and thus $\bB_Y\subset \bB_{X'}$, again by~\eqref{eq:oplus}, and finally by Lemma~\ref{LemProd} we have $\bB_{X'}\subset\bB$.
\end{proof}

We conclude this section with an application of Construction~\ref{Construction}.

\begin{lemma}\label{YonLem}
For $Z_1,Z_2\in\cC$, the canonical `evaluation at $M$' morphism in $\cC$
$$\int^{M\in\mB^n[\cC]}(M^\ast\otimes_{\bk} Z_1^{\otimes n})^{S_n}\otimes (M\otimes_{\bk} Z_2^{\otimes n})^{S_n}\;\to\; \Gamma^n(Z_1\otimes Z_2)$$
is an isomorphism.
\end{lemma}
\begin{proof}
For notational convenience, we assume that $\cC$ has projective objects. It then suffices to show that the action of $\Hom_{\cC}(P,-)$ produces an isomorphism, for all $P\in\Proj\cC$. By application of~\eqref{EqNat} for $l=2$, we need to show that, for every $i,j\in\Irr\cC$
$$\int^M\Hom_{\cC}(P_i,M^\ast\otimes_{\bk} Z_1^{\otimes n})^{S_n}\otimes_{\bk}\Hom_{\cC}(P_j,M\otimes_{\bk} Z_2^{\otimes n})^{S_n}$$
$$\to\;\left(\Hom_{\cC}(P_i,Z_1^{\otimes n})\otimes_{\bk}\Hom(P_j, Z_2^{\otimes n})\right)^{S_n}$$
is an isomorphism, where in the target we consider invariants with respect to the diagonal action of $S_n$. By application of Lemma~\ref{LemDual}, we can rewrite this as
$$\int^{M\in\mB^n[\cC]}\Hom_{S_n}\left(M,A\right)\otimes_{\bk}\Hom_{S_n}\left(B,M\right)\;\to\;\Hom_{S_n}\left(B,A\right),\quad\mbox{with}$$
$$A:=\Hom_{\cC}(P_i, Z_1^{\otimes n})\in\mB^n[\cC]\quad\mbox{and}\quad B:=\Hom_{\cC}(P_j', (Z_2^\ast)^{\otimes n})\in\mB^n[\cC].$$
This is now indeed an isomorphism, by the (co-)Yoneda lemma.
\end{proof}

%%%%%%%%%%%%%%%%%%%%%%%%%%%%%%%%%%%%%%%%%%%%%%%%%%%%%%%%%%%%%%%%%%%%%%%%%%%%%%%%%%%%%%%%%%%%%%%%%%%%%%%%%%%%%%%%%%%%%%%%%%%%%%%%%%%%%%%%%%%%%%%%%%%%%%%%%%%%%%%%%%%%%%%%%%%%%%%%%%%%%%%%%%%%%%%%%%%%%%%%%%%%%%%%%%%%%%%%%%%%%%%%%%%%%%%%%%%%%%%%%%%%%%%%%%%%%%%%%%%%%%%%%%%%%%%%%%%%%%%%%%%%%%%%%%%%%%%%%%%%%%%%%%%%%%%%%%%%%%%%%%%%%%%%%%%%%%%%%%%%%%%%%%%%%%%%%%%%%%%%%%%%%%%%%%%%%%%%%%%%%%%%%%%%%%%%%%%%%%%%%%%%%%%%%%%%%%%%%%%%%%%%%%%%%%%%%%%%%%%%%%%%%%%%%%%%%%%%%%%%%%%%%%%%%%%%%%%%%%%%%%%%%%%%%%%%%%%%%%%%%%%%%%%%%%%%%%%%%%%%%%%%%%%%%%%%%%%%%%%%%%%%%%%%%%%%%%%%%%%%%%%%%%%%%%%%%%%%%%%%%%%%%%%%%%%%%%%%%%%%%%%%%%%%%%%%%%%%%%%%%%%%%%%%%%%%%%%%%%%%%%%%%%%%%%%%%%%%%%%%%%%%%%%%%%%%%%%%%%%%%%%%%%%%%%%%%%%%%%%%%%%%%%%%%%%%%%%%%%%%%%%%%%%%%%%%%%%%%%%%%%%%%%%%%%%%%%%%

\section{The Verlinde category}
Let $\bk$ be an algebraically closed field of characteristic $p>0$.
We determine $\mB^n[\Ver_p]$ for the symmetric fusion category $\Ver_p$.
We follow the labelling of simples in $\Ver_p$ from \cite{Os}, so $\Irr\Ver_p=\{1,2,\cdots, p-1\}$ with $L_1=\unit$. By \cite{Os} the tensor subcategories of $\Ver_p$ are given by $\Vecc$, $\sVec$ (with simple objects $L_1,L_{p-1}$) and $\Ver_p^+$ (with simple objects $L_i$ for $i$ odd).

\subsection{Main result}
We will prove the following theorem in Section~\ref{PrVerp}.
\begin{theorem}\label{ThmVerp}
\begin{enumerate}
\item For $1\le i<p$, the inductive system $\mB_{L_i}$ is the semisimple inductive system $\bC\mS(i)$ from Theorem~\ref{ThmKl}.
\item The inductive systems $\ICS$ and $\ICS_+$ are closed.
\item We have
$$\mB[\Vecc]=\Young,\quad \mB[\sVec]=\SYoung,\quad \mB[\Ver_p^+]=\ICS_+,\quad \mB[\Ver_p]=\ICS.$$

\end{enumerate}
\end{theorem}

From definition, it is clear that $\ICS$ and $\ICS_+$ satisfy all properties of a closed inductive system, except for the following fact:
\begin{corollary}\label{CorICSMon}
The subcategory $\ICS_n\subset\Rep S_n$ is monoidal.
\end{corollary}

The following special case can presumably alternatively be obtained from the analysis of $\overline{\Rep S_p}$ in \cite[\S 4.4]{EOs}.
\begin{corollary}
The tensor product of two simple $S_p$-modules is a direct sum of a semisimple and a projective module.
\end{corollary}
%\begin{proof}
%Since $\ICS_p$ consists of all direct sums of simple and projective modules, this follows from Corollary~\ref{CorICSMon} for $n=p$.
%\end{proof}

Recall that a module of a finite group $G$ is {\bf algebraic} if only finitely many isomorphism classes of indecomposable modules appear in its tensor powers; or equivalently if its class in the Green ring $R(G)$ satisfies a polynomial identity. We refer to \cite{Cr} for more context.

\begin{corollary}\label{CorAlg}
Every completely splittable module is algebraic.
\end{corollary}
\begin{proof}
By construction, $\ICS^n$ contains only finitely many isomorphism classes of indecomposable modules. Hence the conclusion follows from Corollary~\ref{CorICSMon}.
\end{proof}

\begin{corollary}\label{4Cor}
The following are equivalent for a Frobenius-exact tensor category $\cC$:
\begin{enumerate}
\item $\cC$ has a surjective tensor functor to $\Vecc$, $\sVec$, $\Ver_p^+$, $\Ver_p$ respectively;
\item $\mB^p[\cC]$ equals $\Young^p, \SYoung^p,\ICS_+^p,\ICS^p$  respectively;
\item $\mB[\cC]$ equals $\Young, \SYoung,\ICS_+,\ICS$  respectively.
\end{enumerate}
Precisely one of the four options applies to $\cC$.
\end{corollary}
\begin{proof}
That (1) implies (3) follows from Remark~\ref{RemF}(1). That (3) implies (2) is straightforward. That (2) implies (1) follows from \cite[Theorem~1.1]{CEO}, which states that, regardless of any assumption on $\mB[\cC]$, we must be in one of the four cases of (1). Indeed, we can then use that (1) implies (3) and the fact that degree $p$ separates the inductive systems (in those cases where they are distinct).
\end{proof}

\begin{problem}
A natural open problem is to extend Donkin's classification of signed Young modules \cite{Donkin} to a classification of indecomposables in $\ICS$. By the theory developed further in the paper, the problem is equivalent to classifying simple polynomial representations of degree $n$ in $\Rep_{\Ver_p} (\GL_X,\phi)$, for $X=(\oplus_i L_i)^n$. Note that, even though the simple objects in  $\Rep_{\Ver_p} (\GL_X,\phi)$ are classified in \cite{Ve}, the classification does not reveal which representations are polynomial.
\end{problem}

\subsection{Semisimple inductive systems and the proof}\label{PrVerp}
\begin{lemma}\label{LemSurj}
Let $\cC$ be a semisimple tensor category, $X\in\cC$ and $n\in\mN$. The following conditions are equivalent:
\begin{enumerate}
\item The braid action $\beta^n_X$ in \eqref{BrMor} is surjective.
\item In $\cC\boxtimes\Rep S_n$ we have an isomorphism 
$$X^{\otimes n}\;\simeq\; \bigoplus_{i=1}^l L_i\boxtimes V_i,$$
for non-isomorphic simple $L_i\in\cC$ and non-isomorphic simple $V_i\in \Rep S_n$. 
\end{enumerate}
\end{lemma}
\begin{proof}
Since $\cC$ is semisimple, $X^{\otimes n}$ is always isomorphic to $\oplus_{i=1}^l L_i\boxtimes M_i$
in $\cC\boxtimes\Rep S_n$, for non-isomorphic simple objects $L_i\in\cC$ and certain $M_i\in \Rep S_n$. The morphism~$\beta_X^n$ then corresponds to the action morphism
\begin{equation}\label{kMi}\bk S_n\;\to\; \bigoplus_{i=1}^l \End_{\bk}(M_i).\end{equation}
By Jacobson's density theorem, $\bk S_n\to \End_{\bk}(M_i)$ is surjective if and only if $M_i$ is simple. That \eqref{kMi} is surjective if and only if all $M_i$ are simple and non-isomorphic then follows easily.
\end{proof}

\begin{corollary}\label{CorSS}
Consider a semisimple tensor category $\cC_1$ with an object $Y\in\cC_1$ such that $\beta^n_Y$ is surjective for all $n\in\mN$. For a tensor functor $F:\cC_1\to \cC$ with $X:=F(Y)$ it follows that the inductive system $\mB_X$ is semisimple.
\end{corollary}
\begin{proof}
By Proposition~\ref{PropF}(1), it suffices to prove that $\mB_Y$ is semsimple. The latter follows immediately from Lemma~\ref{LemSurj}.
\end{proof}

\begin{proof}[Proof of Theorem~\ref{ThmVerp}]
Part (1) implies part (3) by Theorem~\ref{ThmTCIS}(4), and part (3) implies part (2) by Theorem~\ref{ThmTCIS}(3). It thus suffices to prove part (1). The cases $i=1$ and $i=p-1$ are straightforward. We thus focus on $1<i<p-1$ and consider the tensor category
$$\Ver_p(\SL_i)\;:=\; \overline{\Tilt \SL_i},$$
which is the semisimplification of the category of $\SL_i$-tilting modules, see~\cite{EOs, Ve}. The simple $\overline{V}\in \Ver_p(\SL_i)$ is the image of the natural $i$-dimensional $\SL_i$-representation $V$. By Schur-Weyl duality, $\beta^n_V$, and hence also
$$\beta^n_{\overline{V}}:\;\bk S_n\;\stackrel{\beta^n_V}{\tto}\;\End_{\SL_i}(V)\;\tto\; \End(\overline{V}),$$
is surjective. By \cite{Os}, $\Ver_p(\SL_i)$ admits a tensor functor to $\Ver_p$. More concretely, by \cite[Corollary~4.10]{Ve}, we have a tensor functor
$$\Ver_p(\SL_i)\to\Ver_p,\quad \overline{V}\mapsto L_i.$$
It thus follows from Corollary~\ref{CorSS} that $\mB_{L_i}$ is a semisimple inductive system.

Finally, to identify which semisimple inductive system $\mB_{L_i}$ is, we can use Kleshchev's classification in Theorem~\ref{ThmKl}. Since $\Sym^{p-i}L_i\not=0$, but $\Sym^{p-i+1}L_i=0$ ({\it i.e.} $\bB^{p-i}_{L_i}$ contains~$\unit$ but $\bB^{p-i+1}_{L_i}$ does not), it follows from Theorem~\ref{ThmKl}(2) and (5) that
$$\bC\mS(i)\;\subset\;\mB_{L_i}\;\subset\;\sum_{s\ge i}\mathbf{C}\mathbf{S}(s).$$ Moreover, since $\bigwedge^{i+1}L_i=0$, we find, by the same reasoning and Theorem~\ref{ThmKl}(4), that in fact $\mB_{L_i}=\bC\mS(i)$. 
\end{proof}

%%%%%%%%%%%%%%%%%%%%%%%%%%%%%%%%%%%%%%%%%%%%%%%%%%%%%%%%%%%%%%%%%%%%%%%%%%%%%%%%%%%%%%%%%%%%%%%%%%%%%%%%%%%%%%%%%%%%%%%%%%%%%%%%%%%%%%%%%%%%%%%%%%%%%%%%%%%%%%%%%%%%%%%%%%%%%%%%%%%%%%%%%%%%%%%%%%%%%%%%%%%%%%%%%%%%%%%%%%%%%%%%%%%%%%%%%%%%%%%%%%%%%%%%%%%%%%%%%%%%%%%%%%%%%%%%%%%%%%%%%%%%%%%%%%%%%%%%%%%%%%%%%%%%%%%%%%%%%%%%%%%%%%%%%%%%%%%%%%%%%%%%%%%%%%%%%%%%%%%%%%%%%%%%%%%%%%%%%%%%%%%%%%%%%%%%%%%%%%%%%%%%%%%%%%%%%%%%%%%%%%%%%%%%%%%%%%%%%%%%%%%%%%%%%%%%%%%%%%%%%%%%%%%%%%%%%%%%%%%%%%%%%%%%%%%%%%%%%%%%%%%%%%%%%%%%%%%%%%%%%%%%%%%%%%%%%%%%%%%%%%%%%%%%%%%%%%%%%%%%%%%%%%%%%%%%%%%%%%%%%%%%%%%%%%%%%%%%%%%%%%%%%%%%%%%%%%%%%%%%%%%%%%%%%%%%%%%%%%%%%%%%%%%%%%%%%%%%%%%%%%%%%%%%%%%%%%%%%%%%%%%%%%%%%%%%%%%%%%%%%%%%%%%%%%%%%%%%%%%%%%%%%%%%%%%%%%%%%%%%%%%%%%%%%%%%%%%%

\section{Frobenius-exact tensor categories and beyond}

Let $\bk$ be a field (algebraically closed) of characteristic $p>0$. For Frobenius-exact tensor categories that are not of moderate growth, contrary to \Cref{4Cor}, we cannot control the corresponding inductive systems via a fibre functor to $\Ver_p$. Therefore, in this section we develop some methods to study the inductive systems.

\subsection{Alternative characterisation}

\begin{theorem}\label{ThmFrEx}
For a tensor category $\cC$, the following conditions are equivalent:
\begin{enumerate}
\item $\cC$ is Frobenius-exact;
\item For the injective hull $J$ in $\Ind\cC$ of $\unit$ and every monomorphism $\alpha:\unit\hookrightarrow X$ in $\cC$, the surjection
$$\Hom(X^{\otimes p},J)\;\stackrel{\Hom(\alpha^{\otimes p},J)}{\tto}\; \Hom(\unit, J)\simeq\bk$$
has an $S_p$-equivariant section.
\item For every $I\in\Inj\cC$, every composition $\lambda$ of length $l$ and every $l$ monomorphisms $\{ Z_i\hookrightarrow Y_i\mid 1\le i\le l\}$ in $\cC$,
 the surjection
$$\Hom(Y_{1}^{\otimes \lambda_1}\otimes Y_{2}^{\otimes \lambda_2}\otimes \cdots\otimes Y_{l}^{\otimes \lambda_l},I)\;\tto\; \Hom( Z_{1}^{\otimes \lambda_1}\otimes Z_{2}^{\otimes \lambda_2}\otimes \cdots\otimes Z_{l}^{\otimes \lambda_l},I)$$
has an $S_\lambda$-equivariant section.
\item For every $I\in\Inj\cC$, $W\in\cC$, $n\in\mN$ and every monomorphism $Z\hookrightarrow Y$ in $\cC$, the surjection
$$\Hom( W\otimes Y^{\otimes n},I)\;\tto\; \Hom(W\otimes Z^{\otimes n},I)$$
has an $\End(W)\otimes kS_n$-equivariant section.
\item For every $n\in\mZ_{>0}$, the functor $\Omega^n:\cC\to\Rep S_n$ from Example~\ref{ExOm} sends epimorphisms to {\rm split} epimorphisms.
\end{enumerate}
\end{theorem}
\begin{proof}
We show that (2) implies (1). An $S_p$-equivariant section as in (2) corresponds to a $S_p$-equivariant morphism $X^{\otimes p}\to J$, meaning a morphism in $\Ind\cC$ that factors through $\Sym^p X$, with non-zero composite
$$\unit\xrightarrow{\alpha^{\otimes p}}X^{\otimes p}\tto\Sym^pX\to J.$$
Hence $\unit\to \Sym^p X$ is not zero and $\cC$ is Frobenius-exact by  \cite[Theorem~C(iii)]{Tann}.

Obviously, (3) implies (2).

Now we show that (4) implies (3). For clarity of notation, consider the case $l=2$. Then the surjection can be decomposed as
$$\Hom(Y_1^{\otimes \lambda_1}\otimes Y_2^{\otimes \lambda_2},I)\tto \Hom(Y_1^{\otimes \lambda_1}\otimes Z_2^{\otimes \lambda_2},I)\tto \Hom( Z_1^{\otimes \lambda_1}\otimes Z_2^{\otimes \lambda_2},I).$$
Under condition (4), we can consider an $\End(Y_1^{\otimes \lambda_1})\otimes kS_{\lambda_2}$-equivariant section of the first map and a $kS_{\lambda_1}\otimes \End(Z_2^{\otimes \lambda_2})$-equivariant section of the second, yielding a $kS_{\lambda_1}\otimes kS_{\lambda_2}$-equivariant overall section.

Now we show that (1) implies (4). Under assumption (1), by \cite[Theorem~3.2.2]{Tann}, there exists a non-zero commutative algebra $A$ in $\Ind\cC$ such that for every short exact sequence $X_1\hookrightarrow X\tto X_2$ in $\cC$, the short exact sequence of $A$-modules
$$0\to A\otimes X_1\to A\otimes X\to A\otimes X_2\to 0$$
is split. It follows, see \cite[\S 8.2]{EOf}, that the algebra $A$ must be injective in $\Ind\cC$. 

As an instance of the splitting property, the obvious epimorphism of $A$-modules
$$A\otimes Y^\ast\;\tto\;A\otimes Z^\ast$$
has an $A$-equivariant section, which induces an $\End(W^\ast)\otimes kS_n$-equivariant splitting of 
$$A\otimes W^\ast \otimes (Y^\ast)^{\otimes n}\;\tto\;A\otimes W^\ast \otimes (Z^\ast)^{\otimes n},$$
where we used $A\otimes X^{\otimes n}\simeq (A\otimes X)^{\otimes_An}$. By applying $\Hom(\unit,-)$ and using adjunction, this yields an $\End(W)\otimes kS_n$-equivariant splitting of 
$$\Hom(W\otimes Y^{\otimes n},A)\;\tto\; \Hom(W\otimes Z^{\otimes n},A).$$
It follows easily that this conclusion remains valid for any indecomposable injective summand of $A$. The injective hull $J$ of $\unit$ is such a summand, so the conclusion in (4) is valid for $I=J$. That is then valid for all $I$ follows from the observation that any indecomposable injective object is a summand of $V^\ast\otimes J$ for some $V\in\cC$.

Finally we observe that (5) is equivalent to the previous four (equivalent) properties. Indeed, since we can realise
$\Omega$ as $\Hom(-^\ast, I),$
for some $I\in\Inj\cC$, it follows easily that (4) implies (5), and that (5) implies (2).
\end{proof}

\begin{remark}
Theorem~\ref{ThmFrEx} extends a list of equivalent conditions for Frobenius-exactness from \cite[Theorem~C]{Tann} that are tautologically satisfied for fields $\bk$ of characteristic zero.

\end{remark}

\begin{corollary}\label{CorFrEx4}
If $\cC$ is Frobenius-exact and $X$ is a filtered object in $\cC$, then
$\mB_{X}^d=\mB_{\mathrm{gr}X}^d.$
\end{corollary}
\begin{proof}
For $A$ as in the proof of Theorem~\ref{ThmFrEx}, it follows that, as $A$-modules,
$$A\otimes W^\ast\otimes (X^\ast)^d\;\simeq\;A\otimes W^\ast\otimes (\gr X^\ast)^d,$$
for all $W\in\cC$, equivariantly over $S_d$, which leads to isomorphisms of $S_d$-representations
$$\Hom((X^\ast)^d, W\otimes A)\;\simeq\;\Hom((\mathrm{gr}X^\ast)^d, W\otimes A).$$
This concludes the proof.
\end{proof}

\begin{corollary}\label{CorFrEx2}\label{CorFrEx3}
\begin{enumerate}
\item Let $\cC$ be a Frobenius-exact tensor category and consider an object $X\in\cC$ with subquotient $Y\in\cC$, then $\mB^d_Y$ is included in $\mB^d_X$.
\item Consider a surjective tensor functor $F:\cC\to\cD$ where $\cD$ is Frobenius-exact. Then $\mB[\cC]=\mB[\cD]$.
\end{enumerate}
\end{corollary}
\begin{proof}
Part (1) follows immediately from \Cref{CorFrEx4}. Part (2) follows from part (1) and Proposition~\ref{PropF}.
\end{proof}

For tensor categories with projective objects we can augment Theorem~\ref{ThmFrEx} with additional characterisations in terms of the language introduced in the current paper.

\begin{theorem}\label{CorFrEx}
On a tensor category $\cC$ with projective objects, the following conditions are equivalent:
\begin{enumerate}
\item $\cC$ is Frobenius-exact;
\item The inductive system $\mB_{\cP r}[\cC]$ is closed;
\item We have $\unit\in \mB^p_{\cP r}[\cC]$;
\item $\mB^p_{\cP r}[\cC]$ is not contained in $\Proj\Rep S_p$;
\item For the projective cover $q:Q\tto\unit$ of the tensor unit, the surjection
$$\Hom(Q,Q^{\otimes p})\;\stackrel{\Hom(Q,q^{\otimes p})}{\tto}\; \Hom(Q,\unit)\simeq \bk$$
has an $S_p$-equivariant section;
\item For every $P\in\Proj\cC$, every composition $\lambda$ of length $l$ and every $l$ epimorphisms $\{Y_i\tto Z_i\mid 1\le i\le l\}$ in $\cC$,
 the surjection
$$\Hom(P, Y_{1}^{\otimes \lambda_1}\otimes Y_{2}^{\otimes \lambda_2}\otimes \cdots\otimes Y_{l}^{\otimes \lambda_l})\;\tto\; \Hom(P, Z_{1}^{\otimes \lambda_1}\otimes Z_{2}^{\otimes \lambda_2}\otimes \cdots\otimes Z_{l}^{\otimes \lambda_l})$$
has an $S_\lambda$-equivariant section.
\end{enumerate}
\end{theorem}
\begin{proof}
That (1), (5) and (6) are equivalent follows from Theorem~\ref{ThmFrEx}. 
The special case
$$\Hom(Q,Q^{\otimes n})\tto\Hom(Q,\unit)\simeq\bk$$
of (6) shows that (6) implies that $\unit\in \mB^{n}_{\cP r}[\cC]$ for all $n$, hence $\mB_{\cP r}[\cC]$ is closed by \Cref{ExProj}. So we find that (6) implies (2). 
Clearly, (2) implies (3), and (3) implies (4). Finally, that (4) implies (1) follows from \cite[Example~5.2.5]{CF}.
\end{proof}

\subsection{Wreath products}

\subsubsection{Notation}\label{not}
For $m,n\in\mZ_{>0}$ we consider the wreath product
$$S_m\wr S_n\;=\; S_m^{\times n} \rtimes S_n\;<\; S_{mn}.$$

For any $V\in \Rep S_m$, we let $V^{\boxtimes_\wr n}$ be the $ S_m\wr S_n$-representation, which is $V^{\boxtimes n}$ as a $S_m^{\times n}$-representation and where $S_n$ acts by permuting the tensor factors.

Finally, any $S_n$-representation $M$ can be interpreted as an $S_m\wr S_n$-representation by inflation, with trivial action of $S_m^{\times n}$, which we will denote by $M$ again.

\begin{prop}\label{PropWR}
Let $\cC$ be Frobenius-exact and $X\in\cC$.
The restriction to $S_m\wr S_n$ of any representation in $\mB^{mn}_X$ is a direct summand of a direct sum of representations of the form
$$\Ind^{S_m\wr S_n}_{S_m\wr S_\kappa}\left((M_1^{\boxtimes_\wr \kappa_1}\otimes M_1')\boxtimes \cdots\boxtimes (M_l^{\boxtimes_\wr \kappa_l}\otimes M_l')\right)$$
for some composition $\kappa\vDash n$ of length $l$, with $M_i\in \mB^m_X$ and $M_i'\in \mB^{\kappa_i}[\cC]$ (inflated to a $S_m \wr S_{\kappa_i}$-representation), using the notation in \ref{not} and the identification
$S_m\wr S_\kappa= \prod_{i=1}^l S_m\wr S_{\kappa_i}.$
\end{prop}

We start the proof with the following lemma.

\begin{lemma}\label{LemYn}
If $\cC$ is Frobenius-exact and has projective objects, we can choose, for all $P\in\Proj\cC$ and $Y\in\cC$, an isomorphism of $\End(Y)^{\otimes n}$-modules
$$\bigoplus_{i_1,\cdots, i_n}\Hom(P,L_{i_1}\otimes \cdots\otimes L_{i_n})\otimes_{\bk}\Hom(P_{i_1},Y)\otimes_{\bk}\cdots\otimes_{\bk}\Hom(P_{i_n},Y)\;\xrightarrow{\sim}\; \Hom(P, Y^{\otimes n}), $$
such that the canonical action of $\sigma\in S_n$ on the right-hand side is exchanged with the action
$$f\otimes h_1\otimes \cdots\otimes h_n\;\mapsto\;(\sigma_{L_{i_1},\cdots, L_{i_n}}\circ f)\otimes (h_{\sigma^{-1}(1)}\otimes\cdots \otimes h_{\sigma^{-1}(n)})
$$
on the left-hand side.
\end{lemma}
\begin{proof}
We choose sections $s_{i_1,\cdots, i_n}$ of
$$\Hom(P,P_{i_1}\otimes \cdots\otimes P_{i_n})\tto \Hom(P,L_{i_1}\otimes \cdots\otimes L_{i_n})$$
in the following way. We fix a total order $\le$ on $\Irr\cC$.
For any
$i_1\le i_2\le \cdots\le i_n,$
denote by $\lambda\vDash n$ the composition governing equalities among the $i_l$, so
$$i_1=i_2=\cdots=i_{\lambda_1}<i_{\lambda_1+1}=\cdots=i_{\lambda_2}<\cdots.$$
We then take an $S_\lambda$-equivariant section $s_{i_1,\cdots, i_n}$, which exists by \Cref{CorFrEx}.
For all shufflings of $(i_1,\cdots, i_n)$, naturally labelled by the coset $ S_n/S_\lambda$, we define the section by demanding it forms a commutative diagram with $s_{i_1,\cdots, i_n}$ and the braid action of the corresponding shortest (or due to the set-up, any) representative in $S_n$ of the element in $ S_n/S_\lambda$.

If we choose such sections for the construction \ref{Construction}, the inherited $S_n$-action on the left-hand side can now be computed directly. The $\End(Y)^{\otimes n}$-equivariance follows from the naturality of \eqref{EqNat}.
\end{proof}

\begin{proof}[Proof of Proposition~\ref{PropWR}]
By Lemma~\ref{LemYn}, we have an isomorphism of $\bk S_m^{\times n}$-modules, 
$$\bigoplus_{i_1,\cdots, i_n}\Hom(P,L_{i_1}\otimes \cdots\otimes L_{i_n})\otimes_{\bk}\Hom(P_{i_1},X^{\otimes m})\otimes_{\bk}\cdots\otimes_{\bk}\Hom(P_{i_n},X^{\otimes m})\;\xrightarrow{\sim}\; \Hom(P, X^{\otimes mn}), $$
together with a recipe for the inherited $S_n$-action on the left-hand side.

The proposed description then follows by choosing a decomposition into indecomposable summands of each $S_m$-representation $\Hom(P_{i},X^{\otimes m})$, and applying \Cref{RemKappa}.
\end{proof}

\begin{remark}
As the proof shows, we can improve the formulation in Proposition~\ref{PropWR} to specify that $M_i'\in\mB^{\kappa_i}_{L}$ for some simple $L\in\cC$.
\end{remark}

\subsection{Applications}

\begin{theorem}\label{Thmp2}
Let $\cC$ be a Frobenius-exact tensor category with
$\mB^p[\cC]\subset p\Perm^p,$
then
$$\mB[\cC]\;\subset\;p\Perm.$$
%Moreover,
%$$\mB^i[\cC]=\Young^i,\quad\mbox{for $i\le 4$}.$$
\end{theorem}
\begin{proof}
As restrictions of $p$-permutation representations to subgroups remain $p$-permutation representations, it is sufficient to prove
$$\mB^{p^n}[\cC]\;\subset\; p\Perm^{p^n},\quad\mbox{for $n\in\mZ_{>0}$}.$$
We prove this inclusion by induction on $n$, using the base case $n=1$.

By definition of $p$-permutation modules and the fact that $S_{p^{n-1}}\wr S_p<S_{p^n}$ contains a Sylow $p$-subgroup it suffices to consider the restriction to $S_{p^{n-1}}\wr S_p$ of representations in $\mB^{p^n}[\cC]$.
Assuming the inclusion is satisfied for $n-1$, it then follows from an application of Proposition~\ref{PropWR}, using
$$M^{\boxtimes_{\wr}i}\otimes\Ind^{S_i}_H\unit\;\simeq\;\Ind^{S_m\wr S_i}_{S_m\wr H}M^{\boxtimes_{\wr}i}\quad\mbox{and }\quad (\Ind^{S_m}_K\unit)^{\boxtimes_{\wr}i}\;\simeq\; \Ind^{S_m\wr H}_{K\wr H}\unit$$
in $\Rep S_m\wr S_i$ and $\Rep S_m\wr H$, for $H<S_i$ and $K<S_m$, that the same is true for $n$.
\end{proof}

\begin{remark}
\begin{enumerate}
\item If $\mathrm{char}(\bk)=2$, then the assumption $\mB^2[\cC]\subset2\Perm^2$ is automatic.
\item If $p=3$, then the condition $\mB^3[\cC]\subset3\Perm^3=\SYoung^3$ is equivalent with $\cC$ being of Frobenius type $\Vecc$, see \cite[Question~7.3]{CEO}.
\end{enumerate}
\end{remark}

Recall the notation $\Indec$ from \S\ref{sec:not}.

\begin{prop}
Let $\cC$ be a Frobenius-exact tensor category, then
$\Indec(\mB^n[\cC])$
is a finite set, for each $n\in\mN$.
\end{prop}
\begin{proof}
The proof is similar to that of Theorem~\ref{Thmp2}. Again it is sufficient to prove the claim for $\mB^{p^n}[\cC]$. This can be done by induction on $n$, via Proposition~\ref{PropWR}. The base case $n=1$ follows from representation finiteness of $\bk S_p$.
\end{proof}

\begin{conjecture}
Let $\cC$ be a Frobenius-exact tensor category with
$\mB^p[\cC]\subset \SYoung^p,$
then
$$\mB[\cC]\;\subset\;\SYoung.$$
\end{conjecture}

\subsubsection{}\label{SecTriv} We can apply Proposition~\ref{PropWR} to give a less technical proof of the `Key Lemma' \cite[Lemma~5.1]{CEO}. This can be reformulated as Lemma~\ref{LemKey} below, see also \Cref{ExFrg}(4). For this, for a finite group $G$, we denote by
$$\Triv_G:\Rep G\to\Vecc$$
the $\bk$-linear functor that assignts to a representation its maximal trivial direct summand, see~\cite[\S 2]{Tann}. If $N\lhd G$ is a normal subgroup then $\Triv_N(V)$ is naturally a $G/N$-module.

\begin{lemma}\label{LemKey}
Let $\cC$ be a Frobenius-exact tensor category, $n\in\mZ_{>1}$ and $M\in \mB^{p^n}[\cC]$. Then
$$\dim_{\bk}\Triv_{S_{p^{n-1}}}\left(\Triv_{S_{p}^{\times p^{n-1}}}(M)\right)\;=\;\dim_{\bk}\left(\Triv_{S_{p}\wr S_{p^{n-1}}}M\right).$$
\end{lemma}
\begin{proof}
This is a consequence of Proposition~\ref{PropWR}.
\end{proof}

\subsection{Beyond Frobenius-exact categories} Let $p=2$.

\subsubsection{The spin inductive system} Following \cite{GK}, for $l\in\mN$, we call $D^{(l+2,l)}$ the {\bf spin module} of $\bk S_{2l+2}$. It restricts to $D^{(l+1,l)}$ over $S_{2l+1}$, see \cite[1.11 and 1.12]{GK}. The restriction of $D^{(l+1,l)}$ to $S_{2l}$ is a non-split self-extension of $D^{(l+1,l-1)}$, which we denote by $E^{(l+1,l-1)}$, see \cite[11.2.7 and 11.2.10]{BookKl}.
For example, $E^{(2,0)}\simeq\bk S_2$ and $D^{(2,1)}$ is the non-trivial simple $S_3$-representation.

Even though it is not yet guaranteed by the above description, defining $\Spin^{2l+2}\subset\Rep S_{2l+2}$ as the category of direct sums of $E^{(l+2,l)}$, and $\Spin^{2l+1}\subset\Rep S_{2l+2}$ as the category of direct sums of $D^{(l+1,l)}$, yields an inductive system $\Spin$, which will follow from its realisation in Theorem~\ref{ThmSpin}. This inductive system is `almost semisimple' in the sense that, for all $n\in\mN$, there exists $d>n$ with $\Spin^d$ semisimple.

\begin{theorem}\label{ThmSpin}
For $V$ the non-trivial simple object in $\Ver_4$, we have
$\mB_{V}[\Ver_4]=\Spin.$
\end{theorem}
The proof will be a consequence of the following technical lemma.

\begin{lemma}\label{LemSpin}
Assume that there exists an inductive system $\bA$ where $\bA^{2l+1}$, for $l>0$, comprises precisely direct sums of one non-trivial simple module, then $n\mapsto\Spin^n$ is an inductive system and $\bA=\Spin$.
\end{lemma}
\begin{proof}
By \cite[Theorem~C]{James}, for a 2-regular $\lambda\vdash n$, the multiplicity of $D^{\lambda^-}$ in the restriction of $D^\lambda$ to $S_{n-1}$ is $m$, where, when counting from right to left, the $m$-th removable box in $\lambda$ is the first for which removing the box yields a 2-regular partition, denoted by $\lambda^-$.

It follows that the only simple $S_{2l+2}$-representation for which its restriction to $S_{2l+1}$ only has simple constituents $D^{(l+1,l)}$ is $D^{(l+2,l)}$. Similarly, the only options for simple $S_{2l+1}$-representations for which their restriction to $S_{2l}$ only have simple constituents $D^{(l+1,l-1)}$ are $D^{(l+2,l-1)}$ and $D^{(l+1,l)}$. The former case would imply that the restriction of $D^{(l+2,l-1)}$ equals $D^{(l+1,l-1)}$, which is contradicted by \cite[Theorem~1.2]{CS}.

Now we prove that $\mA^{2l+1}$ comprises the direct sums of $D^{(l+1,l)}$, by induction on $l>0$. The base case $l=1$ is by assumption. So assume the claim is true for a given $l>0$. By the previous paragraph, the only allowed modules in $\mA^{2l+2}$ have only $D^{(l+2,l)}$ as simple constituent and subsequently, the only allowed modules in $\mA^{2l+3}$ have only $D^{(l+2,l+1)}$ as simple constituent. By our assumption, it follows that $\mA^{2l+3}$ comprises the precisely the direct sums of $D^{(l+2,l+1)}$, concluding the induction step.

It thus follows that $\bA^n=\Spin^n$ for $n$ odd, and by definition of $\Spin$ and the assumption that $\bA$ is an inductive system, that $\bA^n=\Spin^n$ for $n$ even.
\end{proof}

\begin{proof}[Proof of Theorem~\ref{ThmSpin}]
It follows from~\ref{DefVer4} that $V^{\otimes 2l+1}$ is a direct sum of $2^l$ copies of $V$. Since we can realise $\Ver_4$ as a monoidal quotient of $\Tilt \SL_2$, with $V$ the image of the vector representation (which we denote by $V$ as well), see \cite[Remark~3.10]{BE}, it also follows that
$$\beta^{2l+1}_V:\bk S_{2l+1}\tto\End_{\SL_2}(V)\tto\End_{\Ver_4}(V^{\otimes 2l+1})$$
is surjective. It follows that, in $\Ver_4\boxtimes\Rep S_{2l+1}$, we have 
$V^{\otimes 2l+1}\simeq V\boxtimes L_{2l+1}$
for some simple $L_{2l+1}$ in $\Rep S_{2l+1}$ of dimension $2^l$, and $\mB_V^{2l+1}[\Ver_4]$ comprises direct sums of $L_{2l+1}$. The claim now follows from Lemma~\ref{LemSpin}.
\end{proof}

\subsubsection{} We define some $\bk S_4$-modules. Recall that $\bk S_4$ has two simple modules, $\unit$ and $L:=D^{(3,1)}$. All extension spaces between two simple modules all one-dimensional. More generally, a complete description of the Young modules is given in \cite[\S 5.4]{EW}.

By $R$ we denote the unique (indecomposable) $\bk S_4$-module with structure
$$\resizebox{0.25\hsize}{!}{$\xymatrix{
\unit\ar@{-}[rd] &&\\
&\unit\ar@{-}[rd]&& L\ar@{-}[ld]\ar@{-}[rd]\\
&&L&&\unit\ar@{-}[rd]\\
&&&&& \unit .
}$}$$

We denote by $V$ the generator of $\Ver_{2^n}$, see \cite{BE, BEO}, which is the image of the tautological $\SL_2$-representation under the defining symmetric monoidal functor $\Tilt \SL_2\to \Ver_{2^n}$.

\begin{prop}\label{Propp2}
\begin{enumerate}
\item $\bB^i[\cC]\;=\;\Rep S_i$, for every tensor category $\cC$ over $\bk$ and $i\le 3$.
 \item We have strict inclusions
$$\Young^4=\mB^4[\Vecc]\;\subset\;\mB^4[\Ver_4^+]\;\subset\; \mB^4[\Ver_4]\;\subset\;\mB^4[\Ver_8^+].$$
\item Moreover,
$$\Indec\mB^4[\Ver_4^+]\;=\;\Indec\Young^4\sqcup\{R\}\quad\mbox{and}\quad \Indec\mB^4[\Ver_4]\;=\;\Indec\mB^4[\Ver_4^+]\sqcup\{E^{(3,1)}\}.$$
\end{enumerate}
\end{prop}
\begin{proof}
Part (1) is immediate since $\Young^i$ equals $\Rep S_i$ for $i\le 3$.

The inclusions in part (2) follow from inclusions of tensor categories and Proposition~\ref{PropF}(2). We only need to show that every inclusion is strict. We start with part (3) which takes care of the first two inclusions.

By part (1), to understand $\mB^4[\cC]$, it suffices to consider indecomposable objects in $\cC$.
Let $P$ be the projective cover of $\unit$ in $\Ver_4$, or equivalently in $\Ver_4^+$, and $V$ the non-trivial simple object in $\Ver_4$. We already proved that $\Hom(P,V^{\otimes 4})$ is $E^{(3,1)}$. Since $V^{\otimes 2}=P$ it follows from \eqref{EqNat} (see also proof of Lemma~\ref{LemProd}) that
$$\Hom(P,P^{\otimes 4})\;\simeq\; \Hom(P,V^{\otimes 4})^{\otimes 2}\;=\; (E^{(3,1)})^{\otimes 2}$$
as $S_4$-representations. We can also quickly calculate this representation directly using the explicit realisation of $\Ver_4^+$ as in \cite[5.2.1]{BE}. Both methods allow us to conclude that the projective cover $Y^{(1,1,1,1)}$ of $\unit$ is a direct summand of $(E^{(3,1)})^{\otimes 2}$. We denote a complementary (self-dual) summand of $(E^{(3,1)})^{\otimes 2}$ by $M$. A direct computation shows that there exists a projective presentation
$$\left(Y^{(2,1,1)}\right)^3\to Y^{(1,1,1,1)}\oplus Y^{(2,1,1)}\to M\to 0,$$
with $Y^{(2,1,1)}$ the projective cover of $L$. Since $[M:\unit]=4$ and $[M:L]=2$ it follows that $M$ is an extension of a cyclic module with socle filtration $L:\unit:\unit$ and one with $\unit:\unit:L$. To conclude that $M=R$ it is sufficient to show that $M$ is indecomposable (a non-split extension).
By \cite[Lemma~7.1.5]{CF} over the Klein 4-subgroup $C_2^2<S_4$, $M$ is a direct sum of two $\mathrm{Gal}(\bk:\mF_2)$-conjugate modules, which shows that $M$ must be indecomposable. This proves part (3).

To conclude the proof of part (2), we prove strictness of the inclusion $\mB^4[\Ver_4]\subset\mB^4[\Ver_8^+]$. We have seen that all modules in $\mB^4[\Ver_4]$ are self-dual. In contrast, it follows from \cite[Example~10.2.5]{CEO2} that there exists an object $X$ (denoted `$E_1$' {\it loc. cit.}) in $\Ver_8^+$, for which
$$\dim_{\bk}\Hom(P, X^{\otimes 4})_{S_4}=1\quad\mbox{and}\quad \dim_{\bk}\Hom(P, (X^\ast)^{\otimes 4})_{S_4}=2,$$
for $P$ the projective cover of $\unit$ (and $-_{S_4}$ represents taking coinvariants). By \Cref{LemDual} this implies that the module $\Hom(P, X^{\otimes 4})$ in $\mB^4[\Ver_8^+]$ is not self-dual.
\end{proof}

\subsubsection{} From \Cref{Propp2} it also follows with minor additional work that
$$\Indec\mB^4_{\cP r}[\Ver_4^+]=\{Y^{(2,1,1)},Y^{(1,1,1,1)}, R\}\quad\mbox{and}\quad \Indec\mB^4_{\cP r}[\Ver_4]=\{Y^{(2,1,1)},Y^{(1,1,1,1)}, E^{(3,1)},R\}.$$

In \cite[Corollary~7.2.3]{CF} it was proved that, for $n>2$, the $S_4$-modules in 
$$\mB^4_{\cP r}[\Ver_{2^n}]\;\supset\;\mB^4_{\cP r}[\Ver_{2^n}^+]$$
must restrict to permutation modules over the Klein subgroup $C_2^{\times 2}<S_4$. A more refined potential result is suggested in the following question.

\begin{question}
For $n>2$, do we have
$$\Indec\mB_{\cP r}^4[\Ver_{2^n}]\;\subset\;\{Y^{(3,1)},Y^{(2,2)},Y^{(2,1,1)},Y^{(1,1,1,1)}\}?$$
\end{question}

\begin{remark}\label{RemVer4}
For $\cC=\Ver_4^+$ we can easily observe from \Cref{Propp2} that $\mB[\cC]$ is bigger than $\mI(\sum_{L\in\cC}\mB_L)$. Indeed, in this case the only simple object is $\unit$, so $\mI(\sum_{L\in\cC}\mB_L)$ is just $\Young$. With minor additional work, we can come to the same conclusion for $\cC=\Ver_4$.
\end{remark}

We conclude this section with another example of symmetric group representations appearing in $\Ver_{2^n}$. As the result is not needed in the rest of the paper we omit the proof, which is mainly an application of tilting module theory for $\SL_2$.

\begin{lemma}\label{Lem2n}For $n\in\mZ_{>1}$, let $P$ be the projective cover of $\unit$ in $\Ver_{2^n}$ and $V$ the generator. Then the $S_{2^n}$-representation $\Hom(P,V^{\otimes 2^n})$
is the subquotient of the Young module 
$$Y^{(2^{n-1},2^{n-1})}=M^{(2^{n-1},2^{n-1})}=\Ind^{S_{2^n}}_{S_{2^{n-1}}\times S_{2^{n-1}}}\unit$$ obtained by removing $\unit$ from top and socle.
\end{lemma}

\section{Ideals in the finitary symmetric group algebra}\label{sec:ideals}
We let $\bk$ again be of arbitrary characteristic (and algebraically closed). One of the main motivations in \cite{Za} of Zalesskii's notion of inductive systems was the connection with ideals in $\bk S_\infty$, for $S_\infty:=\cup_n S_n$. Here we show how Definitions~\ref{DefIS} and~\ref{DefPhiX} relate to such ideals.

\subsection{Definitions}

\subsubsection{}
There is a canonical bijection between the set of two-sided ideals $I<\bk S_\infty$ and the set of collections of ideals $\{I_n<\bk S_n\}$ such that $I_{n-1}=I_n\cap \bk S_{n-1}$. It follows that for an inductive system $\bA$, we obtain an ideal $\cJ(\bA)<\bk S_\infty$ with
$$\cJ(\bA)\cap\bk S_n\:=\:\bigcap_{M\in\bA^n} \Ann_{\bk S_n}M.$$

\begin{example}\label{ExBK}
In \cite{BK}, Baranov and Kleshchev proved that, if $\mathrm{char}(\bk)=p>2$, the {\em maximal} ideals in $\bk S_\infty$ are given precisely by
$$\cJ^{(i)}\,:=\, \cJ(\mathbf{C}\mathbf{S}(i)),\quad 1\le i< p.$$
\end{example}

For the inductive system of an object $X$ in a tensor category $\cC$, we abbreviate 
$$\Ann(X)\;:=\;\cJ(\mB_X[\cC])\;<\;\bk S_\infty.$$
We can define this ideal more directly:

\begin{lemma}
The ideal $\Ann(X)$ equals the kernel of the algebra morphism
$$\varinjlim\beta^n_X:\;\bk S_\infty=\varinjlim \bk S_n\;\to\; \End_{\cC}(X^{\otimes \infty}):=\varinjlim \End_{\cC}(X^{\otimes n}).$$
\end{lemma}
\begin{proof}
Since the morphisms $f\mapsto f\otimes X$ from $\End_{\cC}(X^{\otimes n})$ to $\End_{\cC}(X^{\otimes n+1})$ are injective, it suffices to prove that, for an arbitrary $n\in\mZ_{>0}$, 
$\ker(\beta_X^n)$ equals the intersection of all $\Ann_{\bk S_n}M$, for $M\in\bB_X^n.$
For simplicity of notation we assume that $\cC$ has projective objects. Then $\ker(\beta_X^n)$ is equal to the kernel of the associated morphism
$$\bk S_n\;\to\;\prod_{i\in \Irr \cC} \End_{\bk}(\Hom_{\cC}(P_i,X^{\otimes n})),$$
by faithfulness of $\oplus_i\Hom_{\cC}(P_i,-)$, from which the conclusion follows.
\end{proof}

\begin{remark}
\begin{enumerate}
\item The ideal $\Ann(X)$ contains the same information as the kernel of the symmetric monoidal functor $\mathcal{O}\mathcal{B}(\dim X)\to \cC$ from the oriented Brauer category (the universal $\bk$-linear symmetric rigid monoidal category on an object of categorical dimension $\dim X$ see, \cite{Del07, Selecta}), but $\Ann(X)$ allows one to compare (for instance compare which object has the largest annihilator) objects of different dimension.
\item In \cite[Corollary~4.11]{CEO} it was proved that a tensor category $\cC$ is of moderate growth if and only if $\Ann(X)\not=0$ for each $X\in\cC$.

\end{enumerate}

\end{remark}

Now assume $\mathrm{char}(\bk)=p>0$. We denote the symmetriser in $\bk S_i\subset \bk S_\infty$ by $\mathfrak{s}_i$ and the skew symmetriser in $\bk S_i$ by $\mathfrak{a}_i$. We denote the idempotent in $\bk S_i$ corresponding to the trivial module by $\mathfrak{e}^s_i$, for example $\mathfrak{e}^s_i=\mathfrak{s}_i/i!$ when $i<p$. The maximal ideals $\cJ^{(1)}$ and $\cJ^{(p-1)}$ are generated by $\mathfrak{a}_2$ and $\mathfrak{s}_2$ respectively.

\begin{prop}\label{PropGenId}

\begin{enumerate}
\item For $\cC=\Ver_p$ and $1< i<p-1$, the ideal $\Ann(L_i)<\bk S_\infty$ is generated by $\mathfrak{a}_{i+1}$ and $\mathfrak{s}_{p+1-i}$.
\item For $n\in\mZ_{>0}$ and $V\in \Ver_{p^n}$ the generator, the ideal $\Ann(V)<\bk S_\infty$ is generated by $\mathfrak{a}_{3}$ and $\mathfrak{e}^s_{p^n-1}$.
\end{enumerate}
\end{prop}
\begin{proof}
We start with case (2), so we need to describe the kernel of
$$\bk S_\infty \;\tto\; \End_{\SL_2}(V^{\otimes \infty})\;\tto\; \End_{\Ver_{p^n}}(V^{\otimes \infty}).$$
By classical invariant theory, the kernel of the first surjection is generated by $\mathfrak{a}_3$. We can interpret the middle algebra as an infinite rank Temperley-Lieb algebra, and by construction in \cite{BEO, AbEnv}, see also \cite[\S 5]{Selecta}, the kernel of the second surjection is generated by the Jones-Wenzl idempotent of rank $p^n-1$. This can be lifted to $\mathfrak{e}^s_{p^n-1}\in\bk S_{p^n-1}$, for instance because this Jones-Wenzl idempotent cuts out the Steinberg direct summand $\Sym^{p^n-1}V=\Gamma^{p^n-1}V$ in the $\SL_2$-representation $V^{\otimes p^n-1}$.

Part (1) is proved similarly. Indeed, with $U$ the tautological $\SL_i$-representation and using the proof of Theorem~\ref{ThmVerp}(1), we need to describe the kernel of
$$\bk S_\infty \;\tto\; \End_{\SL_i}(U^{\otimes \infty})\;\tto\; \End_{\Ver_p(\SL_i)}(U^{\otimes \infty}).$$
It is well-known that the maximal tensor ideal in $\Tilt \SL_i$, to be quotiented out to create $\Ver_p(\SL_i)$, is generated by the symmetriser in $\bk S_{p+1-i}$, see for instance the proof of \cite[Lemma~4.1.1]{CEN}. Thus an $\SL_i$-endomorphism of $U^{\otimes m}$ that vanishes in the Verlinde category can be written as a composition
\begin{equation}\label{eqm}U^{\otimes m}\to U^{\otimes p+1-i}\otimes U^{\otimes l}\xrightarrow{s\otimes U^{\otimes l}}  U^{\otimes p+1-i}\otimes U^{\otimes l}\to U^{\otimes m}\end{equation}
for $s=s_{p+1-i}$ and $l\in\mN$ where the unlabelled arrows are arbitrary $\SL_i$-morphisms. By Schur-Weyl duality we know that arbitrary morphisms between tensor powers of $U$ of the same degree~$d$ come from $\bk S_d$. Furthermore, by considering the centre $\mu_i<\SL_i$, we know that there can only be morphisms between tensor powers of degree differing by a multiple of $i$, and any such morphism can be interpreted (say if the degree of the target is the largest)  as 
$$U^{\otimes d}=U^{\otimes d}\otimes (\Lambda^iU)^{\otimes r}\hookrightarrow U^{\otimes d+ir}\to U^{\otimes d+ir},$$
where the inclusion is a direct summand and the final arrow must thus come from $\bk S_{d+ir}$. This principle and the fact (as follows from highest weight consideration and simplicity of the tilting modules in the fundamental alcove) show that $m$ must be at least $p+1-i$ in order to have non-zero endomorphisms that vanish in the Verlinde category, one can then show that the composition \eqref{eqm} must be in the algebra ideal generated by $s_{p+1-i}\otimes U^{\otimes m+i-1-p}$.
\end{proof}

\begin{problem}
A natural open problem is the description of the ideals $\Ann(L)$ for $L$ varying over all simple objects in $\Ver_{p^\infty}$, as classified in \cite{BE, BEO}.
\end{problem}

\subsection{Maximal ideals}
Let $\bk$ be an algebraically closed field of characteristic $p>0$.

\subsubsection{}
First we consider the case $p>2$, in which case we have the classification of Example~\ref{ExBK}.
The maximal ideals $\cJ^{(1)}$ and $\cJ^{(p-1)}$ in $\bk S_\infty$ are generated by the skew symmetriser and symmetriser of $\bk S_2$ respectively.

\begin{prop}\label{PropNewGen}
The maximal ideal $\cJ^{(i)}<\bk S_\infty$, for $1<i<p-1$, is generated by the skew symmetriser in $\bk S_{i+1}$ and the symmetriser in $\bk S_{p+1-i}$.
\end{prop}
\begin{proof}
This follows from Proposition~\ref{PropGenId} and \Cref{ThmVerp}(1).
\end{proof}

\begin{remark}
To each object $X$ in a tensor category $\cC$, we can now associate a non-empty subset of $\{1,\ldots ,p-1\}$ of those $i$ for which $\Ann(X)\subset \Ann(L_i)=\cJ^{(i)}$. For example, for $X=\unit^n$, this set is $\{1,\ldots, \min(p-1,n)\}$.
\end{remark}

\subsubsection{} If $p=2$, it is no longer true that the only maximal ideal in $\bk S_\infty$ correspond to semisimple inductive systems. In other words, there is at least one maximal ideal other than the augmentation ideal $\cJ^{(1)}=\Ann(\unit)$, see \cite[Remark(5) on p597]{BK}.  However, all known maximal ideals connect again to tensor categories:

\begin{lemma}\label{LemSpinIdeal}
With $V\in \Ver_4$, we have the maximal ideal 
$$\Ann(V)\;=\;\cJ(\Spin)\;<\;\bk S_\infty.$$
It is generated by $\mathfrak{e}^s_3$.
\end{lemma}
\begin{proof}
It follows from the proof of \Cref{ThmSpin} that
$$\bk S_\infty/\cJ(\Spin)\;\simeq\; \End_{\Ver_4}(V^{\otimes \infty})\;\simeq\;\varinjlim \End_{\Ver_4}(V^{\otimes 2l+1})\;\simeq\;\varinjlim \End_{\Ver_4}(V^{2^l}).$$
Hence the quotient is a simple algebra, as a direct limit of simple (matrix) algebras. That the ideal is generated by $\mathfrak{e}^s_3$ follow by \Cref{PropGenId}(2) and the fact that $\mathfrak{s}_3=\mathfrak{a}_3$ is in the ideal of $\bk S_3$ generated by $\mathfrak{e}^s_3$.
\end{proof}

%%%%%%%%%%%%%%%%%%%%%%%%%%%%%%%%%%%%%%%%%%%%%%%%%%%%%%%%%%%%%%%%%%%%%%%%%%%%%%%%%%%%%%%%%%%%%%%%%%%%%%%%%%%%%%%%%%%%%%%%%%%%%%%%%%%%%%%%%%%%%%%%%%%%%%%%%%%%%%%%%%%%%%%%%%%%%%%%%%%%%%%%%%%%%%%%%%%%%%%%%%%%%%%%%%%%%%%%%%%%%%%%%%%%%%%%%%%%%%%%%%%%%%%%%%%%%%%%%%%%%%%%%%%%%%%%%%%%%%%%%%%%%%%%%%%%%%%%%%%%%%%%%%%%%%%%%%%%%%%%%%%%%%%%%%%%%%%%%%%%%%%%%%%%%%%%%%%%%%%%%%%%%%%%%%%%%%%%%%%%%%%%%%%%%%%%%%%%%%%%%%%%%%%%%%%%%%%%%%%%%%%%%%%%%%%%%%%%%%%%%%%%%%%%%%%%%%%%%%%%%%%%%%%%%%%%%%%%%%%%%%%%%%%%%%%%%%%%%%%%%%%%%%%%%%%%%%%%%%%%%%%%%%%%%%%%%%%%%%%%%%%%%%%%%%%%%%%%%%%%%%%%%%%%%%%%%%%%%%%%%%%%%%%%%%%%%%%%%%%%%%%%%%%%%%%%%%%%%%%%%%%%%%%%%%%%%%%%%%%%%%%%%%%%%%%%%%%%%%%%%%%%%%%%%%%%%%%%%%%%%%%%%%%%%%%%%%%%%%%%%%%%%%%%%%%%%%%%%%%%%%%%%%%%%%%%%%%%%%%%%%%%%%%%%%%%%%%%%%%%%%%%%%%%%%%%%%%%%%%%%%%%%%%%%%%%%%%%%%%%%%%%%%%%%%%%%%%%%%%%%%%%%%%%%%%%%%%%%%%%%%%%%%%%%%%%%%%%%%%%%%%%%%%%%%%%%%%%%%%%%%%%%%%%%%%%%%%%%%%%%%%%%%%%%%%%%%%%%%%%%%%%%%%%%%%%%%%%%%%%%%%%%%%%%%%%%%%%%%%%%%%%%%%%%%%%%%%%%%%%%%%%%%%%%%%%%%%%%%%%%%%%%%%%%%%%%%%%%%%%%%%%%%%%%%%%%%%%%%%%%%%%%%%%%%%%%%%%%%%%%%%%%%%%%%%%%%%%%%%%%%%%%%%%%%%%%%%%%%%%%%%%%%%%%%%%%%%%%%%%%%%%%%%%%%%%%%%%%%%%%%%%%%%%%%%%%%%%%%%%%%%%%%%%%%%%%%%%%%%%%%%%%%%%%%%%%%%%%%%%%%%%%%%%%%%%%%%%%%%%%%%%%%%%%%%%%%%%%%%%%%%%%%%%%%%%%%%%%%%%%%%%%%%%%%%%%%%%%%%%%%%%%%%%%%%%%%%%%%%%%%%%%%%%%%%%%%%%%%%%%%%%%%%%%%%%%%%%%%%%%%%%%%%%%%%%%%%%%%%%%%%%%%%%%%%%%%%%%%%%%%%%%%%%%%%%%%%%%%%%%%%%%%%%%%%%%%%%%%%%%%%%%%%%%%%%%%%%%%%%%%%%%%%%%%%%%%%%%%%%%%%%%%%%%%%%%%%%%%%%%%%%%%%%%%%%%%%%%%%%%%%%%%%%%%%%

\section{Background on polynomial representations and functor categories}
\label{sec:back}

\subsection{Polynomial representations of general linear groups}
Fix a tensor category $\cC$ over $\bk$ and an object $X\in\cC$. We refer to \cite[\S 7]{ComAlg} for a summary of the basic theory of affine group schemes internal to tensor categories and their representations. We also fix $d\in\mN$.

\subsubsection{}\label{DefSchur} We define the Schur algebra $\cS(X,d)$, which is an algebra in $\cC$, as the algebra of $S_d$-invariants in the internal endomorphism algebra of $X^{\otimes d}$
$$\cS(X,d):=\underline{\End}(X^{\otimes d})^{S_d}\;\simeq\; \Gamma^d(X^\ast\otimes X).$$

\subsubsection{}  The affine group scheme~$\GL_X$ sends any ind-algebra $A$ in $\cC$ to the automorphisms of the $A$-module $A\otimes X$. One verifies that $\GL_X$ is represented by the quotient $\cO(\GL_X)$ of the algebra
\begin{equation}\label{eqSymSym}\Sym(X^\ast\otimes X\,\oplus\, X^\ast\otimes X)\;\simeq\; \Sym(X^\ast\otimes X)\otimes \Sym(X^\ast\otimes X)\end{equation}
by the ideal generated by the images of
$$X^\ast\otimes X\;\xrightarrow{(\ev_X, X^\ast\otimes \co_X\otimes X)}\;\unit\;\oplus X^\ast\otimes X\otimes X^\ast\otimes X$$
and a similar morphism
$$X\otimes X^{\ast}\;\to\;\unit\;\oplus X^\ast\otimes X\otimes X^\ast\otimes X.$$

\subsubsection{} For the obvious bi-grading of \eqref{eqSymSym}, the ideal defining $\cO(\GL_X)$ is thus generated by two subobjects of the direct sum of the components of degree $0,0$ and $1,1$. Using this grading, it follows easily that the defining morphism from the left factor in (the right-hand side of) equation~\eqref{eqSymSym} yields a monomorphism
\begin{equation}\label{eqSymO}
\Sym(X^\ast\otimes X)\;\hookrightarrow \; \cO(\GL_X).
\end{equation}
This induces, for any commutative algebra $A$ in $\Ind\cC$
$$\GL_X(A)=\Aut_A(A\otimes X)\;\hookrightarrow\;\End_A(A\otimes X).$$

\subsubsection{} Denote by $\Rep_{\cC}\GL_X$ the representation category of $\GL_X$, that is the category of $\cO(\GL_X)$-comodules in $\cC$. It contains
the $d$-th tensor power $X^{\otimes d}$ of the defining representation. We consider the category of polynomial representations of degree $d$ 
$$\Rep^d_{\cC}\GL_X\;\subset\; \Rep_{\cC}\GL_X$$
which is the topologising subcategory of $\Rep_{\cC}\GL_X$ generated by the $\GL_X$-representations $X^{\otimes d}\otimes Z$, with $Z\in\cC$.

\begin{lemma}\label{lem:SchurPoly}
Restricting \eqref{eqSymO} to $\Sym^d(X^\ast\otimes X)$ yields a sub-coalgebra of $\cO(\GL_X)$ which has as dual algebra $\cS(X,d)$ from \S\ref{DefSchur}. The resulting functor
 from $\Mod_{\cC}\cS(X,d)$ to $ \Rep_{\cC}\GL_X$
yields an equivalence 
$$\Mod_{\cC}\cS(X,d)\;\xrightarrow{\sim}\;\Rep^d_{\cC}\GL_X. $$

\end{lemma}
\begin{proof}
Both claims follow precisely as in the classical case $\cC=\Vecc$. For example, if a $\GL_X$-representation $Y$ is such that the coaction takes values in $\Sym^d(X^\ast\otimes X)\subset \cO(\GL_X)$, then the coaction
$$Y\;\hookrightarrow\; \Sym^d(X^\ast\otimes X)\otimes Y$$
realises $Y$ as a subquotient of $X^{\otimes d}\otimes \omega((X^\ast)^{\otimes d}\otimes Y)$, where we use $\omega$ for the forgetful functor from $\Rep_{\cC}\GL_X$ to $\cC$.
\end{proof}

\subsubsection{} 
Recall from \cite[\S 8]{Del90} the fundamental group $\pi(\cC)$ of $\cC$. It is an affine group scheme internal to $\cC$, that sends a commutative algebra $A$ in $\Ind\cC$ to the automorphism group of the monoidal functor $A\otimes-$ from $\cC$ to the module category of $A$ in $\Ind\cC$. 
Then we have the standard homomorphism of affine group schemes in $\cC$
\begin{equation}
\label{EqPhi} \phi:\; \pi(\cC)\;\to\; \GL_X,
\end{equation}
For example, on $\bk$-points it is given by evaluation at $X$
$$\Aut^{\otimes}(\Id_{\cC})\;\to\; \Aut_{\cC}(X).$$
Following \cite[\S 8]{Del90} we consider the topologising subcategory $\Rep_{\cC}(\GL_X,\phi)$ of $\Rep_{\cC}\GL_X$ comprising $\GL_X$-representations $Y$ on which the canonical action of $\pi(\cC)$ on $Y$ coincides with the restriction via \eqref{EqPhi}.

Denote by $\Rep^d_{\cC}(\GL_X,\phi)$ the topologising subcategory of $\Rep_{\cC}(\GL_X,\phi)$ of representations that are in $\Rep^d_{\cC}\GL_X$ as well as in $\Rep_{\cC}(\GL_X,\phi)$.

%We have a diagram of inclusions of full subcategories
%$$\xymatrix{
%\Rep^d_{\cC}\GL_X\ar@{^{(}->}[rr]&&  \Rep_{\cC}\GL_X\\
%\Rep^d_{\cC}(\GL_X,\phi)\ar@{^{(}->}[rr]\ar@{^{(}->}[u]&&  \Rep_{\cC}(\GL_X,\phi).%\ar@{^{(}->}[u]
%}$$

\begin{lemma}
\begin{enumerate}
\item The equivalence
\begin{equation}\label{EqDTRep}
\Rep_{\cC}(\GL_X,\phi)\boxtimes \cC\;\xrightarrow{\sim}\; \Rep_{\cC}\GL_X
\end{equation}
from \cite[Lemma~4.2.3]{CEO2} restricts to an equivalence
$$\Rep^d_{\cC}(\GL_X,\phi)\boxtimes \cC\;\xrightarrow{\sim}\; \Rep^d_{\cC}\GL_X.$$
\item $\Rep^d_{\cC}(\GL_X,\phi)$ is the topologising subcategory of $\Rep_{\cC}\GL_X$ generated by $X^{\otimes d}$.
\end{enumerate}
\end{lemma}
\begin{proof}
Denote by 
$\cA\subset\Rep^d_{\cC}(\GL_X,\phi)$ the topologising subcategory of $\Rep_{\cC}\GL_X$ generated by $X^{\otimes d}$. Restriction of the equivalence \eqref{EqDTRep} yields a functor
$$\cA\boxtimes\cC\;\to\; \Rep_{\cC}^d\GL_X,$$
by construction fully faithful and essentially surjective, and hence an equivalence.
This implies that
$$\Rep^d(\GL_X,d)\boxtimes\cC\;\to\; \Rep_{\cC}\GL_X,$$
is also essentially surjective and hence an equivalence, proving part (1). Part (2) then follows from \cite[Corollary~3.2.8]{CF}.
\end{proof}

\subsection{Modules of categories}We refer to \cite[\S 4.8 and \S 4.10]{Be} for a textbook treatise of the main concepts of this section.

\subsubsection{}\label{defAF} We fix the following data. Let $\cA$ be a $\bk$-linear pseudo-abelian category with a fixed {\em faithful} $\bk$-linear functor
$$\omega:\cA\to\Vecc,$$
where, as before, $\Vecc$ is the category of finite dimensional vector spaces. It follows that $\cA$ is a Krull-Schmidt category, and by finite dimensionality, for each $M\in\Indec\cA$, there is a unique $\bk$-algebra morphism
\begin{equation}\label{alphaM}
\alpha_M:\; \End_{\cA}(M)\,\to\, \bk.
\end{equation}

We define $\mo^\omega(\cA)$, or simply $\mo(\cA)$ when $\omega$ is clear from context, or when $\mo^\omega(\cA)$ is independent of $\omega$ as in \eqref{eq:modFun} below, as the topologising subcategory, generated by $\omega,$ of the abelian category
$\Fun_{\bk}(\cA,\Vecc)$
of $\bk$-linear functors $\cA\to\Vecc$. Note that the condition for a functor $\cA\to\Vecc$ to be a subquotient of some $\omega^n$, $n\in\mN$, is only a `finiteness' condition. Indeed, by faithfulness of $\omega$, every functor $\cA\to\Vecc$ is (inside the category of functors $\cA\to\Vecc^\infty$) a subquotient of $\omega^\kappa$ for some cardinality $\kappa$. Moreover, if $\Indec\cA$ is finite, then
\begin{equation}\label{eq:modFun}
\mo(\cA)\;=\;\mo^\omega(\cA)\;=\;\Fun_{\bk}(\cA,\Vecc).
\end{equation}

%Even though objects in $\Fun_{\bk}(\cA,\Vecc)$ can have infinite length, Jordan-H\"older multiplicities are finite, and every simple object has a p

\begin{example}
We are interested in cases where $\cA$ is a pseudo-abelian subcategory of $\Rep G$, for a finite group $G$, and $\omega=\Frg$ is the forgetful functor. Note that $\mo^\omega(\Rep G)$ contains all finitely generated objects, ({\it i.e.} quotients of representable functors), but in general more. Similarly, $\mo^\omega(\Rep G)$ is usually a proper subcategory of $\Fun(\Rep G,\Vecc)$.
\end{example}

\begin{lemma}\label{LemFiltFrg}Keep notation and assumptions from \S\ref{defAF}.
\begin{enumerate}
\item There is a bijection 
$$\Indec \cA\;\stackrel{1:1}{\leftrightarrow}\;\Irr\,\mo^\omega(\cA),\;\quad M\mapsto \Delta_M,$$
so that $\Delta_M(N)=0$ for all $M\not=N\in\Indec\cA$.
\item Inside $\mo^\omega(\cA)$, we have decomposition multiplicities
$$[F:\Delta_M]\;=\;\dim_{\bk}F(M),\quad\mbox{for all $M\in\Indec\cA$ and $F\in\mo^{\omega}(\cA)$}.$$
\end{enumerate}
\end{lemma}
\begin{proof}
The bijection between simple objects in $\Fun_{\bk}(\cA,\Vecc)$ and $\Indec\cA$ is well-known. Concretely, for $M\in\Indec\cA$, we can define
$\Delta_M:\;\cA\to\Vecc,$
uniquely determined by $\Delta_M(N)=0$ for all indecomposables $N\not\simeq M$ and $\Delta_M(M)=\bk$, while the action of $\Delta_M$ on $\End_{\cA}(M)$ is given by
$\alpha_M$ in \eqref{alphaM}.

For part (1) it thus suffices to show that $\Delta_M$ is included in $\mo^\omega(\cA)$.
For this we can observe that $\Delta_M$ is a quotient of its projective cover in $\Fun_{\bk}(\cA,\Vecc)$
$$\Hom_{\cA}(M,-):\;\cA\to\Vecc,$$
which is in turn, by faithfulness of $\omega$, a subobject of $\omega^{\dim_{\bk}\omega(M)}$.

The previous paragraph shows that $\Hom_{\cA}(M,-)$ is the projective cover of $\Delta_M$ in $\mod(\cA)$ so that part (2) follows from the Yoneda lemma.
\end{proof}

\subsubsection{} Let $\cB\subset\cA$ be a pseudo-abelian subcategory. We denote the restriction of $\omega$ to $\cB$ again by $\omega$. Then we have an obvious restriction functor
\begin{equation}\label{FAB}
\mo^\omega(\cA)\to\mo^\omega(\cB),\quad \Delta\mapsto \Delta|_{\cB}
\end{equation}
which identifies the right-hand side with the Serre quotient of $\mo^\omega(\cA)$ with respect to the Serre subcategory of $F\in \mo^{\omega}(\cA)$ that satisfy $F(N)=0$ for all $N\in\cB$. The following lemma is now standard.

\begin{lemma}\label{LemSer}
\begin{enumerate}
\item Let $\cC$ be an abelian category.
Composition with~\eqref{FAB} yields an equivalence of categories between the category of exact functors $\mo(\cB)\to\cC$ and the category of exact functors $\mo(\cA)\to\cC$  which send $F:\cA\to\Vecc$ to zero when $F|_{\cB}=0$.
\item Consider a family of pseudo-abelian subcategories $\{\cB_\alpha\subset\cA\mid\alpha\in\Lambda\}$ such that $\sum_\alpha\cB_\alpha=\cA$. Then the functor
$$\mo(\cA)\;\to\; \prod_\alpha \mo(\cB_\alpha)$$
is faithful.
Assume furthermore that there is a function $\gamma:\Lambda^{(2)}\to\Lambda$, mapping every unordered pair $\{\alpha,\beta\}$ to $\gamma(\alpha,\beta)$ such that
$\cB_\alpha\subset\cB_{\gamma(\alpha,\beta)}\supset\cB_\beta$.
Then for any $\Delta, \nabla\in\mo(\cA)$ and a collection of morphisms $$\{f_\alpha:\Delta|_{\cB_\alpha}\Rightarrow \nabla|_{\cB_\alpha}\mid\alpha\in\Lambda\},$$
there is $f:\Delta\Rightarrow\nabla$ with $f|_{\cB_\alpha}=f_\alpha$ if and only if $f_{\gamma(\alpha,\beta)}|_{\cB_\alpha}=f_\alpha$ for all $\{\alpha,\beta\}\in \Lambda^{(2)}$.

\end{enumerate}

\end{lemma}

\begin{example}
We have pseudo-abelian subcategories $\mB_X^d[\cC]\subset\bT^d$, where $(\cC,X\in\cC)$ runs over all tensor categories and their objects. For $(\cD,Y\in\cD)$, Proposition~\ref{PropF}(1) implies
$$\mB^d_X[\cC]=\mB^d_{X}[\cC\boxtimes\cD]\subset \mB^d_{X\oplus Y}[\cC\boxtimes\cD]\supset \mB^d_{Y}[\cC\boxtimes\cD]=\mB^d_Y[\cD].$$ Hence
we are in the situation of Lemma~\ref{LemSer}(2).
\end{example}

\subsubsection{}\label{FunAC}
We fix a pseudo-abelian subcategory $\cA\subset\Rep S_d$ with $\Indec\cA$ finite.
Then
$$\mo(\cA)\boxtimes \cC=\Fun_{\bk}(\cA,\Vecc)\boxtimes \cC\;\simeq\; \Fun_{\bk}(\cA,\cC),$$
see \cite[Example~3.2.4(3)]{CF}.
By the Yoneda lemma 
\begin{eqnarray*}Q&\simeq &\int^{M\in \cA} Q(M)\otimes_{\bk} \Hom_{S_d}(M,-)\\
&\simeq &
 \int_{M\in \cA} Q(M)\otimes_{\bk} \Hom_{S_d}(-,M)^\ast,\qquad \mbox{for all $Q\in \Fun_{\bk}(\cA,\cC)$}.
  \end{eqnarray*}

For all $U,V\in\cC$, we define
\begin{equation}\label{eqRUV}R_{U;V}\in\,\Fun_{\bk}(\cA,\cC),\quad M\mapsto \left((M\otimes_{\bk}(U^{\ast})^{\otimes d})\right)^{S_d}\otimes V.\end{equation}

%%%%%%%%%%%%%%%%%%%%%%
%%%%%%%%%%%%%%%%%%%%%%%%%%%%%%%%%
%%%%%%%%%%%%%%%%%%%%%%%%%%%%%%%%%
%%%%%%%%%%%%%%%%%%%%%%%%%%%%%%%%%
%%%%%%%%%%%%%%%%%%%%%%%%%%%%%%%%%
%%%%%%%%%%%
%%%%%%%%%%%
%%%%%%%%%%%
%%%%%%%%%%%
%%%%%%%%%%%
%%%%%%%%%%%
%%%%%%%%%%%
%%%%%%%%%%%
%%%%%%%%%%%

\section{Two approaches to polynomial functors}\label{SecPoly}

Fix $d\in\mN$. 

\subsection{Strict polynomial functors}
Fix a tensor category $\cC$ over $\bk$.

\subsubsection{}
Recall the self-enrichement $\uC$  of $\cC$ from Example~\ref{SelfEnr}.
We can also consider $\uC$ as enriched in the category of affine schemes in $\Ind\cC$, and define strict polynomial functors in this fashion, see \cite{Ax} for the case $\cC=\sVec$. Instead, we follow the equivalent path to polynomial functors from \cite{Kr} and references therein. 

Let $\Gamma^d\uC$ be the $\cC$-enriched category with same objects as $\uC$, but 
$$\Gamma^d\uC(X,Y)\;:=\;\Gamma^d( \uC(X,Y))\;\simeq\; \underline{\Hom}(X^{\otimes d}, Y^{\otimes d})^{S_d}.$$
The composition arrows are obtained from the observation that evaluation at $Y^{\otimes d}$,
$$\uHom(Y^{\otimes d}, Z^{\otimes d})\otimes \uHom(X^{\otimes d}, Y^{\otimes d})\;\to\; \uHom(X^{\otimes d}, Z^{\otimes d}),$$
restricts to a map between the $\Gamma^d\uC$-morphism subobjects.

\begin{remark}\label{RemGamEnr}
 Let $\Gamma^d\cC$ be the underlying category of $\Gamma^d\uC$, with morphism spaces
$$\Hom_{\cC}(X^{\otimes d}, Y^{\otimes d})^{S_d}.$$
Then $\Gamma^d\cC$ is the full subcategory of $\cC\boxtimes\Rep S_d$ corresponding to the essential image of $(-)^{\otimes d}$, see Example~\ref{ExOm}(1).
As in Section~\ref{ModEnr} we consider the $\cC$-enrichment of  $\cC\boxtimes \Rep S_d$ as a $\cC$-module category. Then $\Gamma^d\uC$ is the full $\cC$-enriched subcategory of $\underline{\cC\boxtimes \Rep S_d}$ corresponding to the objects contained in $\Gamma^d\cC\subset \cC\boxtimes \Rep S_d$.

\end{remark}

\begin{definition}\label{DefSPol}
The category $\SPol^d\cC$ of {\bf strict polynomial $\cC$-functors of degree $d$} is the $\bk$-linear category of $\cC$-enriched functors $\Gamma^d\uC\to\uC$ and $\cC$-enriched natural transformations. 
\end{definition}

Concretely, a strict polynomial $\cC$-functor $\mT$ is an assignment
$$\Ob \cC\to\Ob\cC,\quad X\mapsto \mT(X),$$
and, for each pair of objects $X,Y$ in $\cC$, a morphism in $\cC$
\begin{equation}\label{EqMorPF}
\underline{\Hom}(X^{\otimes d},Y^{\otimes d})^{S_d}\to \uHom(\mT(X),\mT(Y))
\end{equation}
satisfying the usual associativity and identity relations.

A morphism from $\mT_1$ to $\mT_2$ comprises a morphism 
$\mT_1(X)\to \mT_2(X)$ in $\cC$ for every object $X\in \cC$ so that the two canonical morphisms
$$\uHom(X^{\otimes d},Y^{\otimes d})^{S_d}\otimes \mT_1(X)\;\rightrightarrows\; \mT_2(Y)$$
are equal for all $X,Y\in\cC$.

Using exactness of the bifunctor $\uHom$ it follows easily that $\SPol^d\cC$ is an abelian category. In particular, a chain $\mT_1\to \mT_2\to \mT_3$ in $\SPol^d\cC$ is a short exact sequence if and only if
$$0\to \mT_1(X)\to \mT_2(X)\to \mT_3(X)\to 0$$
is exact in $\cC$ for each $X\in\cC$.

\begin{example}\label{TYZ}
For $U,V\in\cC$, we define $\mT_{U;V}\in\SPol^d\cC$, by
$$\mT_{U;V}(X)\;:=\; U\otimes \uHom(V^{\otimes d},X^{\otimes d})^{S_d}\,=\,U\otimes \Gamma^d(V^\ast\otimes X).$$
The defining morphisms \eqref{EqMorPF}, for all $X,Y\in\cC$, can be obtained via adjunction from
$$U\otimes \Gamma^d(V^\ast\otimes X)\otimes \Gamma^d(X^\ast\otimes Y)\;\to\; U\otimes \Gamma^d(V^\ast\otimes Y).$$
\end{example}

\begin{example}\label{ExFullTensor}
For a fixed $Z\in\cC$, we define the assignment
$X\mapsto Z\otimes X^{\otimes d}$, with morphisms
$$\uHom(X^{\otimes d}, Y^{\otimes d})^{S_d}\;\hookrightarrow\;\uHom(X^{\otimes d}, Y^{\otimes d})\;\hookrightarrow\; \uHom(Z\otimes X^{\otimes d}, Z\otimes  Y^{\otimes d})$$
given by inclusion of invariants followed by whiskering of $\co_Z$. We denote the corresponding strict polynomial functor by $Z\otimes \mT^d$.
\end{example}

\begin{prop}\label{PropTUV}
\begin{enumerate}
\item For every $\mT\in \SPol^d\cC$ we have natural isomorphisms
$$\Hom_{\SPol^d\cC}(\mT_{U;V}, \mT)\;\simeq\; \Hom_{\cC}(U,\mT(V)).$$
\item Every object in $\Ind(\SPol^d\cC)$ is a quotient of a direct sum of objects of the form $\mT_{U;V}$.
\item Every object in $\Ind(\SPol^d\cC)$ is a subquotient of a direct sum of strict polynomial functors $ Z\otimes \mT^{ d}$ as in Example~\ref{ExFullTensor}.
\end{enumerate}
\end{prop}
\begin{proof}
Part (1) is just a version of the Yoneda lemma. Part (1) implies that 
$$\prod_{U,V}\Hom_{\SPol^d\cC}(\mT_{U;V},-)\,:\; \SPol^d\cC\to\Vecc^\infty$$
is a faithful functor, from which part (2) follows.

Since $\mT_{U;V}$ is a subobject of $ U\otimes (V^{\ast})^{\otimes d}\otimes\mT^d$, part (3) follows from part (2).
\end{proof}

Motivated by Proposition~\ref{PropTUV} we define the following `finite' version of $\SPol^d\cC$.

\begin{definition}\label{DefSPolcirc}
The full (topololigising) subcategory $\SPol^d_{\circ}\cC$ of $\SPol^d\cC$ is the category of subquotients of the objects $Z\otimes \mT^d$, for $Z\in\cC$.
\end{definition}

\begin{example}
We have $\SPol^1\cC\simeq\cC\simeq \SPol^1_{\circ}\cC$, see Theorem~\ref{ThmDisc}.
\end{example}

\subsubsection{}Fix an object $X\in\cC$. 
We have an obvious $\bk$-linear exact functor
\begin{equation}\label{Eval2}
\SPol^d\cC\;\to\; \Mod_{\cC}\cS(X,d)\simeq \Rep^d \GL_X,\quad \mT\mapsto \mT(X).
\end{equation}
 Indeed, if we interpret the $\cC$-algebra $\cS(X,d)$ as a $\cC$-enriched category with one object, it is by construction a full subcategory of $\Gamma^d\uC$ and the above functor corresponds to restriction.

\subsection{Universal polynomial functors}

\subsubsection{} A {\bf universal functor} is the data of a (not necessarily additive) endofunctor $\Phi^\cC$ of each tensor category $\cC$ over $\bk$ together with a natural isomorphism
$$\eta^F\;:\; \Phi^{\cD}\circ F\;\stackrel{\sim}{\Rightarrow}\;F\circ \Phi^{\cC} $$
for each tensor functor $F:\cC\to \cD$ such that:
\begin{enumerate}
\item $\eta^{\Id_{\cC}}=\Id_{\Phi^{\cC}}$ and $\eta^{G\circ F}\;=\;G(\eta^F)\circ (\eta^G)_F,$ for all compositions $G\circ F$ of tensor functors.
\item For every natural transformation $f:F\Rightarrow G$ of tensor functors $F,G:\cC\to\cD$, we have 
$$f_{\Phi^{\cC}}\circ\eta^F= \eta^G\circ\Phi^{\cD}( f).$$
\end{enumerate}

A morphism between universal functors $\{\Phi^{\cC},\eta^F\}$ and $\{\Psi^{\cC},\xi^F\}$ is the assignment of a natural transformation $\alpha^{\cC}:\Phi^{\cC}\Rightarrow \Psi^{\cC}$ for each $\cC$, so that for each tensor functor $F:\cC\to\cD$ 
$$ F(\alpha^{\cC})\circ \eta^F\;=\; \xi^F\circ (\alpha^{\cD})_F.$$

The category $\UFun=\UFun_{\bk}$ of universal functors is $\bk$-linear and abelian. Indeed, finite limits and colimits can be computed at each object in each tensor category.

\begin{example}\label{ExTd}
The functors 
$$T_d^{\cC}:\,\cC\to\cC,\; X\mapsto X^{\otimes d}$$
define a universal functor $T_d$. Indeed, the monoidal structure of tensor functors provides the required natural isomorphisms 
$$\gamma^F_X:\;F(X)^{\otimes d}\;\xrightarrow{\sim}\; F(X^{\otimes d}).$$ All coherence conditions follow from the definition of monoidal functors.
\end{example}

\begin{definition}\label{DefLPol}
The category $\Pol^d_{\bk}$ of {\bf universal polynomial functors of degree $d$} is the topologising subcategory of $\UFun_{\bk}$ generated by $T_d$.
\end{definition}

\begin{example}
\begin{enumerate}
\item For $\mathrm{char}(\bk)=0$ or $d<\mathrm{char}(\bk)=p$ the Schur functors $\mS^\lambda$, see \cite{Del02}, for $\lambda\vdash d$ are universal polynomial functors of degree $d$.
\item For arbitrary characteristics, $\Sym^d$, $\Gamma^d$, $\bigwedge^d$, $\Lambda^d$ and $\mathrm{A}^d$, see \cite[\S 2.3]{CEO}, are universal polynomial functors.
\item For $\mathrm{char}(\bk)=p>0$, the enhanced Frobenius functor $\Fr^{en}:\cC\to\cC\boxtimes\overline{\Rep S_p}$ from \cite[\S 3.2]{CEO} decomposes into $(p-1)^2$ universal polynomial functors of degree $p$. In fact, with notation as in \S \ref{sec:D} below, and $M_i$ running over non-negligible indecomposable $S_p$-representations:
$$\Fr^{en}\;=\;\bigoplus_i \,D_{M_i}(-)\boxtimes \overline{M_i}.$$
\end{enumerate}

\end{example}

\subsubsection{}For any tensor category $\cD$ over $\bk$ with object $Y\in\cD$, there is an obvious $\bk$-linear and exact evaluation functor
$\Pol^d_{\bk}\to \cD,$ sending $\Phi$ to $\Phi^{\cD}(Y).$
Of particular interest is $\cD=\Rep_{\cC}(\GL_X,\phi)$ and $Y$ the $\GL_X$-representation $X$, for some $X\in\cC$. This yields a $\bk$-linear and exact evaluation functor
\begin{equation}\label{Eval3}
\Pol^d_{\bk}\;\to\; \Rep^d_{\cC}(\GL_X,\phi),\quad \Phi\mapsto \Phi^{\cC}(X).
\end{equation}

\begin{remark}\label{RemUniqueAct}
For ease of notation, we will often abbreviate $\Phi^{\Rep_{\cC}\GL_X}(X)$ to $\Phi^{\cC}(X)$ as in \eqref{Eval3}. Indeed, we can identify both under the forgetful functor $\Rep_{\cC}\GL_X\to\cC$ (by definition of universal functors). Moreover, the $\cO(\GL_X)$-coaction on $\Phi^{\cC}(X)$ can be recovered unambiguously; it is the unique coaction that  $\Phi^{\cC}(X)$ inherits as a subquotient of the $\GL_X$-representation $(X^{\otimes d})^n$, for $\Phi^{\cC}$ viewed as a subquotient of $(T_d)^n$.
\end{remark}

%%%%%%%%%%%%%%%%%%%%%%%%%%%%%%%%%%%%%%%%%%%%%%%%%%%%%%%%%%%%%%%%%%%%%%%%%%%%%%%%%%%%%%%%%%%%%%%%%%%%%%%%%%%%%%%%%%%%%%%%%%%%%%%%%%%%%%%%%%%%%%%%%%%%%%%%%%%%%%%%%%%%%%%%%%%%%%%%%%%%%%%%%%%%%%%%%%%%%%%%%%%%%%%%%%%%%%%%%%%%%%%%%%%%%%%%%%%%%%%%%%%%%%%%%%%%%%%%%%%%%%%%%%%%%%%%%%%%%%%%%%%%%%%%%%%%%%%%%%%%%%%%%%%%%%%%%%%%%%%%%%%%%%%%%%%%%%%%%%%%%%%%%%%%%%%%%%%%%%%%%%%%%%%%%%%%%%%%%%%%%%%%%%%%%%%%%%%%%%%%%%%%%%%%%%%%%%%%%%%%%%%%%%%%%%%%%%%%%%%%%%%%%%%%%%%%%%%%%%%%%%%%%%%%%%%%%%%%%%%%%%%%%%%%%%%%%%%%%%%%%%%%%%%%%%%%%%%%%%%%%%%%%%%%%%%%%%%%%%%%%%%%%%%%%%%%%%%%%%%%%%%%%%%%%%%%%%%%%%%%%%%%%%%%%%%%%%%%%%%%%%%%%%%%%%%%%%%%%%%%%%%%%%%%%%%%%%%%%%%%%%%%%%%%%%%%%%%%%%%%%%%%%%%%%%%%%%%%%%%%%%%%%%%%%%%%%%%%%%%%%%%%%%%%%%%%%%%%%%%%%%%%%%%%%%%%%%%%%%%%%%%%%%%%%%%%%%%%%%%%%%%%%%%%%%%%%%%%%%%%%%%%%%%%%%%%%%%%%%%%%%%%%%%%%%%%%%%%%%%%%%%%%%%%%%%%%%%%%%%%%%%%%%%%%%%%%%%%%%%%%%%%%%%%%%%%%%%%%%%%%%%%%%%%%%%%%%%%%%%%%%%%%%%%%%%%%%%%%%%%%%%%%%%%%%%%%%%%%%%%%%%%%%%%%%%%%%%%%%%%%%%%%%%%%%%%%%%%%%%%%%%%%%%%%%%%%%%%%%%%%%%%%%%%%%%%%%%%%%%%%%%%%%%%%%%%%%%%%%%%%%%%%%%%%%%%%%%%%%%%%%%%%%%%%%%%%%%%%%%%%%%%%%%%%%%%%%%%%%%%%%%%%%%%%%%%%%%%%%%%%%%%%%%%%%%%%%%%%%%%%%%%%%%%%%%%%%%%%%%%%%%%%%%%%%%%%%%%%%%%%%%%%%%%%%%%%%%%%%%%%%%%%%%%%%%%%%%%%%%%%%%%%%%%%%%%%%%%%%%%%%%%%%%%%%%%%%%%%%%%%%%%%%%%%%%%%%%%%%%%%%%%%%%%%%%%%%%%%%%%%%%%%%%%%%%%%%%%%%%%%%%%%%%%%%%%%%%%%%%%%%%%%%%%%%%%

\section{Three sides of the same coin}\label{sec:coin}

In this section we show how the study of symmetric group representations from tensor categories from the first half of the paper and both approaches to polynomial functors in Section~\ref{SecPoly} are actually different ways of looking at the same information.
Throughout we fix $d\in\mZ_{>0}$. Unless further specified, $\cC$ is an arbitrary tensor category over $\bk$.

\subsection{Main results}

Recall $\bT^d\subset\Rep S_d$ from Question~\ref{QBk}.
The results in the following theorem will be proved in Theorems~\ref{LemDVan}, \ref{ThmPol}, \ref{ThmSPol} and \ref{ThmDisc}, while (3) follows from (4).
\begin{theorem}\label{ThmMain}
Let $\cC$ be a tensor category over $\bk$ and $X\in\cC$.
\begin{enumerate}
\item Evaluation \eqref{Eval2} yields an equivalence
$$\SPol^d_{\circ}\cC\;\xrightarrow{\sim}\; \Rep^d_{\cC}\GL_X$$
 if and only if $\mB^d_X=\mB^d[\cC]$.
\item Evaluation \eqref{Eval3} yields an equivalence
$$\Pol^d_{\bk}\;\xrightarrow{\sim}\; \Rep^d_{\cC}(\GL_X,\phi)$$
if and only if $\mB^d_X=\bT^d$.
\item If $\mB^d[\cC]=\bT^d$, we have an equivalence
$$\SPol^d_{\circ}\cC\;\simeq\; \cC\boxtimes \Pol^d_{\bk}.$$
\item There are equivalences
$$\Pol^d_{\bk}\simeq \mo^{\Frg}(\bT^d),\;\, \SPol^d_\circ\cC\simeq\cC\boxtimes \mo^{\Frg}(\mB^d[\cC])\;\mbox{and}\;\, \Rep^d_{\cC}(\GL_X,\phi)\simeq \Fun_{\bk}(\mB_X^d,\Vecc).$$
\item Inside $\Rep_{\cC} \GL_X$, the representation $X^{\otimes d}$ has a Jordan-H\"older filtration with simples labelled by $\Indec\mB^d_X$, where the simple corresponding to $M\in\mB^d_X$ appears $\dim_{\bk}M$ times.
\end{enumerate}
\end{theorem}

Theorem~\ref{ThmMain} motivates the following terminology.
\begin{definition}
\begin{enumerate}
\item An object $X\in\cC$ is {\bf relatively $d$-discerning} if $\mB_X^d=\mB^d[\cC]$.
\item An object $X\in\cC$ is {\bf $d$-discerning} if $\mB_X^d=\bT^d$.
\item A tensor category $\cC$ is {\bf $d$-discerning} if $\mB^d[\cC]=\bT^d$.
\end{enumerate}
\end{definition}

\begin{lemma}\label{LemDiscFrEx}
Let $\cC$ be a finite Frobenius-exact tensor category with projective generator~$P\in\cC$. Then $P^n$ is relatively $d$-discerning for all $n\ge d$.
\end{lemma}
\begin{proof}
Corollary~\ref{CorFrEx2}(1) implies that
$\mB^d[\cC]=\cup_n \mB^d_{P^n}.$
That, for any $X$, $\mB^d_{X^n}=\mB^d_{X^d}$ for all $n\ge d$ follows easily by induction on $d$.
\end{proof}

\begin{remark}Lemma~\ref{LemDiscFrEx} does not remain true without the condition of Frobenius-exactness. Indeed, it is not true for $\Ver_4$ or $\Ver_4^+$ by Example~\ref{Ex22}(2) below.
\end{remark}

\begin{example}\label{Ex22} We consider the case $d=2$ and $p=\mathrm{char}(\bk)=2$.
\begin{enumerate}
\item Every tensor category over $\bk$ is 2-discerning, since $\Young^2=\Rep S_2.$
\item An object $X$ is 2-discerning if and only if $\bigwedge^2 X\not=0$ ({\it i.e.} $X$ is not invertible) and $\Fr(X)=X^{(1)}\not=0$. For instance, in $\Vecc$, $\Ver_4^+$ and $\Ver_4$ every object is 2-discerning so long as it contains a direct summand $\unit$ and is not indecomposable.
\item For every tensor category $\cC$ and $n\in\mZ_{>1}$, we have
$$\Pol^2\;\simeq\; \SPol^2\cC\;\simeq\;\Rep^2 \GL_n\;\simeq\; \Fun_{\bk}(\Rep S_2,\Vecc).$$
%which corresponds to the category of finite dimensional modules over the path algebra of the quiver
%$$\xymatrix{
%e\ar@/^/[r]^a& s\ar@/^/[l]^b
%}$$with relation $b\circ a=0$.
\end{enumerate}
\end{example}

\begin{example}
If $d<p=\mathrm{char}(\bk)$ or $\mathrm{char}(\bk)=0$, then every tensor category $\cC$ over $\bk$ is $d$-discerning and
$$\SPol^d\cC\;\simeq\;\cC\boxtimes\Pol^d\;\simeq\;\cC\boxtimes \Rep^d \GL_d\;\simeq\;\cC\boxtimes \Rep S_d.$$
\end{example}

We conclude this section by observing that in tensor categories of moderate growth, no object can be (relatively) $d$-discerning for all $d\in\mN$.
\begin{lemma}\label{Lemneq}
If $\cC$ is of moderate growth, then $\mB_X\not=\mB[\cC]$ for all $X\in\cC$.
\end{lemma}
\begin{proof}
By \cite[Proposition~4.7(3)]{CEO}, there is some $n\in\mN$ for which $$\beta_X^n:\;kS_n\;\to\; \End(X^{\otimes n})$$
is not injective. By construction, the kernel of $\beta^n_X$ acts trivially on all representations in $\mB_X^n$. Conversely, $\mB^n[\cC]$ contains the faithful regular representation $\bk S_n$, see Example~\ref{ExIndSys2}(2).
\end{proof}

\begin{remark}
For tensor categories not of moderate growth, Lemma~\ref{Lemneq} is not true. 
See Example~\ref{Ex1}(2), or $(\Rep GL)_t$, say for $\bk=\mC$ and $t\not\in\mZ$, see \cite[\S 10]{Del07}.
\end{remark}

\begin{remark}Since the categories $\bT^d$ and $\bB^d[\cC]$ are monoidal,  Theorem~\ref{ThmMain}(4) allows for the monoidal structure to be extended, via Day convolution as in \cite{Kr}, to categories of polynomial functors.
\end{remark}

\subsection{Representations of the symmetric group and universal functors}

\subsubsection{}\label{DeftoPol} We define a $\bk$-linear functor
\begin{equation}\label{Funmo0}
\Fun_{\bk}(\Rep S_d,\Vecc)\;\to\; \UFun_{\bk},\quad \Delta\mapsto T_\Delta,
\end{equation}
as follows. For any $\bk$-linear functor
$\Delta:\Rep S_d\to \Vecc,$
and tensor category $\cC$, we define the composite functor
$$T_{\Delta}^{\cC}:\;\cC\xrightarrow{-^{\otimes d}}\cC\boxtimes \Rep S_d\xrightarrow{\cC\boxtimes\Delta}\cC.$$
For every tensor functor $F:\cC\to\cD$, we have natural isomorphisms corresponding to the two commutative squares in 
$$
\xymatrix{
\cC\ar[rr]^-{\RT_d^{\cC}=-^{\otimes d}}\ar[d]^F&&\cC\boxtimes \Rep S_d\ar[rr]^-{\cC\boxtimes\Delta}\ar[d]^{F\boxtimes \Rep S_d}&&\cC\ar[d]^F\\
\cD\ar[rr]^-{\RT_d^{\cD}=-^{\otimes d}}&&\cD\boxtimes \Rep S_d\ar[rr]^-{\cD\boxtimes\Delta}&&\cD.
}
$$
Indeed, the first square is commutative via $\gamma$ in Example~\ref{ExTd}.
For the second square we can take the natural isomorphism 
$$a(F,\Delta):\; (\cD\boxtimes \Delta)\circ (F\boxtimes \Rep S_d)\;\stackrel{\sim}{\Rightarrow}\;F\circ (\cC\boxtimes\Delta)$$ from Lemma~\ref{LemNonsense}, which allows us to give $T_\Delta:=\{T^{\cC}_{\Delta}\}$ the structure of a universal functor, as follows from Lemma~\ref{LemNonsense}(1)-(3).

For a natural transformation $\beta:\Delta_1\Rightarrow\Delta_2$, we have the natural transformation
$$(\cC\boxtimes \beta)_{\RT_d^{\cC}}:\; T^{\cC}_{\Delta_1}\Rightarrow T^{\cC}_{\Delta_2}.$$
As an immediate application of Lemma~\ref{LemNonsense}(4) this produces a morphism of universal functors.

\begin{lemma}\label{LemSq0}
Fix $\Delta\in\Fun_{\bk}(\Rep S_d,\Vecc)$.
\begin{enumerate}
\item The functor \eqref{Funmo0} is exact.
\item For a tensor category $\cC$ and a faithful exact $\bk$-linear functor $\Omega:\cC\to\Vecc$, the following diagram of $\bk$-linear functors is commutative (with $\Omega^d$ from Example~\ref{ExOm}):
$$\xymatrix{
\cC\ar[rr]^{T_{\Delta}^{\cC}}\ar[d]_{\Omega^d}&&\cC\ar[d]^{\Omega}\\
\Rep S_d\ar[rr]^{\Delta}&&\Vecc.
}$$

\item We have $T^{\cC}_{\Delta}=0$ if and only if $\Delta|_{\mB^d[\cC]}=0$. For a given $X\in\cC$, we have $T^{\cC}_{\Delta}(X)=0$ if and only if $\Delta|_{\mB^d_X}=0$.
\end{enumerate}

\end{lemma}
\begin{proof}
Part (2) is a direct application of the commutative diagram 
$$\xymatrix{
\cC\ar[rr]^-{-^{\otimes d}}&&\cC\boxtimes\Rep S_d\ar[rr]^{\cC\boxtimes\Delta}\ar[d]_{\Omega\boxtimes \Rep S_d}&&\cC\ar[d]^{\Omega}\\
&&\Rep S_d\ar[rr]^{\Delta}&&\Vecc,
}$$
which is an example of Lemma~\ref{LemNonsense}.

For part (1), observe that exactness of \eqref{Funmo0} is equivalent to exactness of the composite 
$$\Fun(\Rep S_d,\Vecc)\xrightarrow{\Delta\mapsto T_{\Delta}}\UFun\xrightarrow{\Phi\mapsto\Phi^{\cC}(X) } \cC\xrightarrow{\Omega}\Vecc$$
for every tensor category $\cC$ and $X\in\cC$, and some choice of $\Omega$ as in (2). By part (2), this composite functor is simply $\Delta\mapsto \Delta(\Omega^d(X))$, which is indeed exact, proving part (1).

Part (3) follows from part (2) and faithfulness of $\Omega$.
\end{proof}

\begin{example}\label{ExFrg}
The universal functor $T_{\Frg}$ is isomorphic to $\{T_d^{\cC}\}$ from Example~\ref{ExTd}.
\end{example}

\subsection{Representations of the symmetric group and universal polynomial functors}
\label{sec:D}

 By Example~\ref{ExFrg} and Lemma~\ref{LemSq0}(1), if $\Delta$ is a subquotient of $\Frg^n$ for some $n\in\mN$, then $T_\Delta$ is a universal polynomial functor. 
Hence \eqref{Funmo0} restricts to a $\bk$-linear exact functor
\begin{equation}\label{Funmo1}
T:\,\mo^{\Frg}(\Rep S_d)\;\to\; \Pol^d_{\bk},\quad \Delta\mapsto T_\Delta.
\end{equation}

\begin{example}\label{ExFrg2}
\begin{enumerate}
\item By Lemma~\ref{LemFiltFrg}, the simple objects $\Delta_M$ in $\mo(\Rep S_d)$ are in bijection with the indecomposable $S_d$-representations $M\in\Indec\Rep S_d$. We abbreviate
$$\{D_M^{\cC}\mid \cC\}\;=\;D_M\;:=\; T_{\Delta_M},\quad\mbox{for}\quad M\in\Indec\Rep S_d.$$
For a tensor category $\cC$ with object $X\in\cC$, we will typically further abbreviate $D_M^{\cC}(X)$ to $D_M(X)$.
\item Set $d=2$ and assume $p=2$. Then $\bigwedge^2X$ is defined as the image of $1+\sigma:X^{\otimes 2}\to X^{\otimes 2}$ (equivalently, the quotient $X^{\otimes 2}/\Gamma^2 X$), and the quotient $X^{(1)}$ of $\Gamma^2X$ by $\bigwedge^2 X$ is isomorphic to $\Fr(X)$, see \cite{BE}. We thus have a length three filtration of $X^{\otimes 2}$, with successive quotients given by $\bigwedge^2X$, $X^{(1)}$ and $\bigwedge^2 X$. By comparing the above with the projective presentations of $\Delta_M$, for the two indecomposable $M\in\Rep S_2$, or alternatively just by observing that $X^{(1)}$ and $\bigwedge^2X$ are evaluations of universal polynomial functors and applying Theorem~\ref{LemDVan}, we find
$$D_{\unit}(X)\simeq\Fr(X)= X^{(1)}\qquad \mbox{and}\qquad D_{\bk C_2}(X)\simeq\bigwedge^2(X).$$
By Theorem~\ref{LemDVan} below, $X^{(1)}$ and $\bigwedge^2X$ are simple (or zero) $\GL_X$-representations. Furthermore, by Theorem~\ref{ThmPol}, the only indecomposable universal polynomial functors of degree $2$ are $\Fr$, $\bigwedge^2$, $\Sym^2$, $\Gamma^2$ and $T_2$.
\item For $d<\mathrm{char}(\bk)$ or $\mathrm{char}(\bk)=0$, the decomposition in terms of Schur functors
$$X^{\otimes d}\;\simeq\; \bigoplus_{\lambda\vdash d}\mS_\lambda(X)$$
from \cite[\S 1]{Del02} satisfies $\mS_\lambda= D_{S^\lambda}$, for $S^\lambda$ the Specht module corresponding to $\lambda$. By Theorem~\ref{LemDVan} below $\mS_\lambda X$ is a simple (or zero) $\GL_X$-representation.
\item For $p=\mathrm{char}(\bk)>0$ and $d=p^j$, consider the functor $\Delta_{\unit}=\Triv_{S_{p^j}}$ from \S\ref{SecTriv}. Then
$$D_{\unit}=T_{\Delta_{\unit}}\;\simeq\;\Fr_+^{(j)},$$ as in \cite[\S 4.1]{Tann}. As proved in \cite[Lemma~5.1]{CEO}, with a new interpretation in Lemma~\ref{LemKey}, we have $\Fr_+^{(j)}\simeq(\Fr_+)^j$ on Frobenius-exact categories.
\end{enumerate}

\end{example}

\subsubsection{}
By Lemma~\ref{LemSq0} and Lemma~\ref{LemSer}(1) applied to the pseudo-abelian subcategories $\mB^d_X\subset\bT^d\subset\Rep S_d$, the functors~\eqref{Eval3} and~\eqref{Funmo1} yield a commutative square, defining an exact faithful functor $H_X$:
\begin{equation}\label{PolComDia}
\xymatrix{\mo^{\Frg}(\Rep S_d)\ar[r]&\mo^{\Frg}(\bT^d)\ar[rr]\ar[d]&& \Pol_{\bk}^d\ar[d]^{\eqref{Eval3}}\\
&\mo(\mB_{X}^d)\ar[rr]^-{H_X}&& \Rep_{\cC}^d(\GL_X,\phi).
}
\end{equation}

\begin{theorem}\label{LemDVan}
Consider a tensor category $\cC$ with $X\in\cC$ and $M\in\Indec\Rep S_d$.
\begin{enumerate}
\item  The functor $H_X:\mo(\mB_{X}^d)\to \Rep^d_{\cC}(\GL_X,\phi)$ is an equivalence.
\item The $\GL_X$-representation $D_M(X)$ is simple if $M\in\mB^d_X$, and zero if $M\not\in\mB^d_X$.
\item Multiplicities of the simple constituents of $X^{\otimes d}$ in $\Rep_{\cC}\GL_X$ are given by
$$[X^{\otimes d}: D_M(X)]\;=\; \dim_{\bk}M.$$
\end{enumerate}
\end{theorem}
\begin{proof}
We will prove, as an application of the straightforward Lemma~\ref{Lem4Equiv} below, that the composite functor
$$F:\;\Fun_{\bk}(\mB^d_X,\cC)\;\simeq\;\cC\boxtimes\mo(\mB^d_X)\;\xrightarrow{\cC\boxtimes H_X}\; \cC\boxtimes \Rep^d_{\cC}(\GL_X,\phi)\simeq \Rep^d_{\cC}\GL_X\;\simeq \; \Mod_{\cC}\cS(X,d)$$
is an equivalence. This implies that also $H_X$ is an equivalence, see \cite[Corollary~3.2.8]{CF}.
Firstly, $F$ is exact and faithful since $H_X$ is so, see \ref{Defpf}.

From the expressions in \ref{FunAC}, we find for $Q:\mB^d_X\to\cC$,
   \begin{eqnarray}\label{FQ1}F(Q)&\simeq &\int^{M\in \mB^d_X} Q(M)\otimes \left(M^\ast\otimes_{\bk}X^{\otimes d}\right)^{S_d}\\
  &\simeq &\int_{M\in \mB^d_X} Q(M)\otimes (M^\ast\otimes_{\bk} X^{\otimes d})_{S_d}.\label{FQ2}
  \end{eqnarray}
  The $\GL_X$-representation structures of the latter two objects are the unique ones coming from quotients or subobjects of copies of $X^{\otimes d}$, see Remark~\ref{RemUniqueAct}.
  
 As collection of objects $E_{\cB}$ in $\cB:=\Mod_{\cC}\cS(X,d)$, we take the free modules
 $$E_{\cB}\;:=\; \{\cS(X,d)\otimes Y\mid Y\in\cC\}.$$
Now, for $Y\in\cC$, consider the functor $R_{X;Y}$ from \eqref{eqRUV}, for $\cA=\mB^d_X$.
It follows from \eqref{FQ1} and Lemma~\ref{YonLem}
that 
\begin{equation}\label{FRY}
F(R_{X;Y})\;\simeq\; \cS(X,d)\otimes Y.
\end{equation}
Moreover, for any $Q:\mB^d_X\to\cC$, we have
\begin{eqnarray*}
\Nat(R_{X;Y},Q)&=&\int_{M\in\mB^d_X}\Hom_{\cC}(R_{X;Y}(M),Q(M))\\
&\simeq &\Hom_{\cC}\left(Y,\int_{M\in\mB_X^d}(M^\ast\otimes_{\bk} X^{\otimes d})_{S_d}\otimes Q(M) \right)\\
&\simeq &\Hom_{\Mod_{\cC}\cS(X,d)}(\cS(X,d)\otimes Y, F(Q))\\
&\simeq &\Hom_{\Mod_{\cC}\cS(X,d)}(F(R_{X;Y}), F(Q)).
\end{eqnarray*}
 The first step is by definition, the second is adjunction, the third follows from~\eqref{FQ2} and the fourth from~\eqref{FRY}. Since $F$ is faithful and the morphism spaces in $\Rep_{\cC}\GL_X$ are finite dimensional over~$\bk$, it follows that all conditions in Lemma~\ref{Lem4Equiv} are indeed satisfied.

Part (2) follows immediately from part (1), as the equivalence must send simple objects to simple objects. An alternative more direct proof goes as follows. By \cite[Example~7.10.2]{EGNO}, any object $Y\in\cC$ is simple as an $\uEnd(Y)$-module. So if $Y$ has the structure of an $A$-module, for a $\cC$-algebra $A$, and the defining algebra morphism
$A\to\uEnd(Y)$
 is an epimorphism in~$\cC$, then $Y$ is a simple $A$-module. We need to show that $D_M(X)$ is simple as an $\cS(X,d)$-module, so it is sufficient to demonstrate that the algebra morphism
\begin{equation}\label{EndDM}
\uEnd(X^{\otimes d})^{S_d}\;\to\; \uEnd(D_M(X))
\end{equation}
is an epimorphism in $\cC$. This can be demonstrated with a tedious argument.

That $X^{\otimes d}$, as a $\GL_X$-representation, has a filtration where the subquotients are given by $\{D_M(X)\mid M\}$, with $D_M(X)$ appearing $\dim_{\bk}M$ times, is a consequence of Lemma~\ref{LemFiltFrg}(2), for $F=\Frg$, and Lemma~\ref{LemSq0}(1). Thus part (3) follows from part (2).
\end{proof}

\begin{lemma}\label{Lem4Equiv}
Let $\cA,\cB$ be abelian categories. Assume that there exists a collection of objects $E_{\cB}\subset \Ob\cB$ such that every object $Y$ in $\cB$ has a presentation 
$$Y_1\to Y_0\to Y\to 0$$ 
by objects $Y_0,Y_1\in E_{\cB}$. Then a functor $F:\cA\to\cB$ is an equivalence if and only if
\begin{enumerate}
\item $F$ is exact; and
\item $F$ is faithful; and
%\item $F$ sends simple objects to simple objects; and
\item for every $Z\in E_{\cB}$, there is $Z'\in \cA$ with $F(Z')\simeq Z$ and for which
$$\Hom_{\cA}(Z',X)\xrightarrow{F}\Hom(F(Z'),F(X))$$
is an isomorphism for all $X\in\cA$.
\end{enumerate}
\end{lemma}

\begin{lemma}\label{LemProject}\label{CorGLFX}
\begin{enumerate}
\item For a tensor functor $F:\cC\to\cC'$ and $X\in\cC$, the tensor functor $\Rep_{\cC}\GL_X\to\Rep_{\cC'}\GL_{FX}$ lifting $F$ restricts to an equivalence that fits into a commutative diagram
$$\xymatrix{\mo(\mB^d_X)\ar[rr]^-{\sim}_-{H_X}\ar[rrd]^-{\sim}_-{H_{FX}}&&\Rep_{\cC}^d(\GL_X,\phi)\ar[d]^\sim&\Pol^d_{\bk}\ar[l]\ar[ld]\\
&&\Rep_{\cC'}^d(\GL_{FX},\phi).}$$
\item For a tensor category $\cD$ and $Y,Z\in\cD$, there exists a commutative diagram
$$\xymatrix{\mo(\mB^d_{Y\oplus Z})\ar[rr]^-{\sim}_-{H_{Y\oplus Z}}\ar[d]_{\Delta\mapsto \Delta|_{\mB^d_Y}}&&\Rep_{\cD}^d(\GL_{Y\oplus Z},\phi)\ar[d]&&\Pol^d_{\bk}\ar[ll]\ar[lld]\\
\mo(\mB^d_{Y})\ar[rr]^-{\sim}_-{H_Y}&&\Rep_{\cD}^d(\GL_{Y},\phi).}$$

\end{enumerate}
\end{lemma}
\begin{proof}
The right triangle in (1) is commutative by definition of universal functors. Since the left triangle is actually derived from the right triangle via \eqref{PolComDia}, it is also commutative. That two (and hence all three) arrows in the left triangle are equivalences
follows from Theorem~\ref{LemDVan}(1).

Now we prove part (2).
There is an obvious restriction functor along $\GL_Y< \GL_{Y\oplus Z}$
$$\Rep^d_{\cD}(\GL_{Y\oplus Z},\phi)\;\to\; \Rep_{\cD}\GL_Y\;\simeq\; \bigoplus_{t\in\mZ}\Rep_{\cD}^{(t)}\GL_Y$$
where the decomposition of the right-hand side is based on the action of $\mG_m=\GL_{\unit}<\GL_Y$. Even though $\Rep^d_{\cD} \GL_Y\subset \Rep^{(d)}_{\cD}\GL_Y$ is a strict inclusion, the fact that a polynomial representation of $\GL_{Y\oplus Z}$ restricts to a polynomial representation of $\GL_Y$ shows that the above functor can be interpreted as
$$\Rep^d_{\cD}(\GL_{Y\oplus Z},\phi)\;\to\; \Rep_{\cD}^d(\GL_Y,\phi)\oplus \bigoplus_{0\le t< d}\Rep_{\cD}^{t}\GL_Y.$$
We define the functor in the middle of the diagram in (2) as the projection onto the relevant summand of the above functor.

As in part (1), it suffices to prove that the right triangle is commutative.
This commutativity follows from the following observation. The degree $d$ part decomposition of $T_d(Y\oplus Z)=(Y\oplus Z)^{\otimes d}$ in the above grading of $\Rep_{\cD}\GL_Y$ is naturally $T_d(Y)$. This principle is inherited for subquotients of $(T_d)^n$, by exploiting the fact that $\Phi(X)$ is canonically a direct summand of $\Phi(X\oplus Y)$, for any universal functor $\Phi$, leading to the desired commutativity.
\end{proof}

\begin{theorem}\label{ThmPol}
\begin{enumerate}
\item The functor $\mo^{\Frg}(\bT^d)\to \Pol^d_{\bk}$ in \eqref{PolComDia} is an equivalence.
\item Evaluation \eqref{Eval3} is an equivalence
$\Pol^d_{\bk}\xrightarrow{\sim} \Rep^d_{\cC}(\GL_X,\phi)$
if and only if $\mB^d_X=\bT^d$.
\end{enumerate}
\end{theorem}
\begin{proof}
Denote the functor in part (1) again by $T$. We prove that it is an equivalence by proving the following three properties:

{\em (i) $T$ is faithful.} By Lemma~\ref{LemSer}(2), only zero morphisms in $\mo(\bT^{d})$ are sent to zero by every (downwards arrow) restriction functor in \eqref{PolComDia}.
Faithfulness thus follows from Theorem~\ref{LemDVan}(1) and commutativity of \eqref{PolComDia}.

{\em (ii) $T$ is full.} Consider $\Delta_i\in \mo(\bT^{d})$, for $i\in\{1,2\}$ and a morphism $\alpha: T_{\Delta_1}\to T_{\Delta_2}$ in $\Pol^d$, which consists of morphisms $\alpha^{\cC}_X:T^{\cC}_{\Delta_1}(X)\to T^{\cC}_{\Delta_2}(X)$ in $\cC$ for every tensor category $\cC$ and $X\in\cC$.  By Theorem~\ref{LemDVan}(1) this defines unique natural transformations $\gamma_X: \Delta_1|_{\mB^d_X}\Rightarrow \Delta_2|_{\mB^d_X}$, for every $(\cC,X)$.

We claim that for a second $(\cD,Y\in\cD)$, and $X\oplus Y$ interpreted within $\cC\boxtimes\cD$, we have 
\begin{equation}\label{eq:gamma}\gamma_{X\oplus Y}|_{\mB^d_{X}}\;=\;\gamma_X.\end{equation} It would then follow from Lemma~\ref{LemSer}(2) that the morphisms $\gamma_X$ combine to give a natural transformation $\gamma:\Delta_1\Rightarrow \Delta_2$. By construction and the fact that no morphism in $\Pol^d$ evaluates to zero on every $X\in\cC$, we find $T(\gamma)=\alpha$.

To prove the claim in \eqref{eq:gamma} we consider the diagram
\begin{equation}\label{Diabox}\xymatrix{
\mo(\mB^d_{X\oplus Y})\ar[drrr]^{\sim}\ar[dr]&\mo(\bT^d)\ar[d]\ar[l]\ar[r]^{T}&\Pol^d\ar[d]\ar[dr]&\\
&\mo(\mB^d_X)\ar[r]^\sim\ar[rd]^\sim&\Rep^d_{\cC}(\GL_X,\phi)\ar[d]^\sim&\Rep^d_{\cC\boxtimes\cD}(\GL_{X\oplus Y},\phi)\ar[dl]\\
&&\Rep^d_{\cC\boxtimes\cD}(\GL_X,\phi).
}\end{equation}
All functors come from \eqref{PolComDia} or Lemma~\ref{LemProject} and the diagram is therefore commutative. It thus follows indeed that $\gamma_X$ is the restriction of $\gamma_{X\oplus Y}$ along $\mB^d_X\subset \mB^d_{X\oplus Y}$.

{\em (iii) $T$ is essentially surjective.} By definition, every object in $\Pol^d$ is a subquotient of $(T_d)^n\simeq (T_{\Frg})^n$, for some $n\in\mN$. Up to replacing $T$ with an isomorphic functor, we can assume $(T_d)^n= (T_{\Frg})^n$. By exactness of $T$, see Lemma~\ref{LemSq0}(1), it suffices to show that every $\Phi\subset T_{d}^n$ in $\Pol^d$ is the image under $T$ of a subobject $\Delta\subset \Frg^n$ in $\mo(\bT^d)$.

Hence we fix $\Phi\subset T_{\Frg}^n$ in $\Pol^d$. For every $(\cC,X\in\cC)$, we have the corresponding $\Phi^{\cC}(X)\subset (X^{\otimes d})^n$ and, under the equivalence in part (1), a corresponding $\Delta_X\subset \Frg|_{\mB^d_X}$. One can observe that this (uniquely) defines a subobject $\Delta\subset \Frg^n$ in $\mo(\bT^d)$, with $\Delta|_{\mB^d_X}=\Delta_X$ for all $X$, because of the following consideration. For every $(\cD,Y\in\cD)$, from \eqref{Diabox} we can derive that $\Delta_{X}$ is the restriction of $\Delta_{X\oplus Y}\subset \Frg^n|_{\mB^d_{X\oplus Y}}$.
It now remains to show that $T_\Delta\subset (T_{d})^n$ and $\Phi\subset (T_d)^n$ are equal as subobjects. This is now straightforward as $T_\Delta^{\cC}(X)=\Phi^{\cC}(X)$ as subobjects of $(X^{\otimes d})^n$, for all $(\cC,X\in\cC)$.

Part (2) follows from part (1), Theorem~\ref{LemDVan}(1) and commutativity of \eqref{PolComDia}.
\end{proof}

\begin{remark}
We can use the theory of Auslander and Reiten on almost split sequences, see e.g. \cite[Proposition~4.12.6]{Be}. This shows that for any non-projective indecomposable $M\in\Rep S_d$, and any $X\in\cC$, we have an exact sequence in $\cC$: 
$$0\to (N^\ast\otimes_{\bk}X^{\otimes d})^{S_d}\to (E^\ast\otimes_{\bk}X^{\otimes d})^{S_d}\to (M^\ast\otimes_{\bk}X^{\otimes d})^{S_d}\to D_M(X)\to 0,$$
coming from a short exact sequence $M\hookrightarrow E\tto N$ in $\Rep S_d$. Note that the above is a projective resolution of the (non-zero provided $M\in\bB^d_X$) simple $D_M(X)$ in $\Rep^d\GL_X$ only if $E,N\in \bB^d_X$. For $p=2=d$ and $M=\bk$, we get
$$0\to \Gamma^2X\to X^{\otimes 2}\to \Gamma^2(X)\to X^{(1)}\to 0.$$
Similarly, if $M$ is projective in $\Rep S_d$, we have a short exact sequence
$$0\to (M_1^\ast\otimes_{\bk}X^{\otimes d})^{S_d}\to (M^\ast\otimes_{\bk}X^{\otimes d})^{S_d}\to D_M(X)\to 0,$$
where $M_1$ is the quotient of $M$ by its socle.
\end{remark}

\subsection{Representations of the symmetric group and strict polynomial functors}

\subsubsection{}\label{ComplicFun} Now we define a $\bk$-linear composite functor $\mT:\Delta\mapsto\mT_\Delta$,
$$\Fun_{\bk}(\Rep S_d,\Vecc)\;\to\; \Fun^{\cC}_{\bk}(\cC\boxtimes\Rep S_d,\cC)\;\to\;\Fun_{\cC}(\underline{\cC\boxtimes\Rep S_d},\uC)\;\to\;\Fun_{\cC}(\Gamma^d\uC,\uC)=\SPol^d\cC.$$
The first functor is as in \ref{MoreNon}, the second functor is an example of Lemma~\ref{LemModtoEnr}, and the third functor is simply the restriction to the subcategory $\Gamma^d\uC $ from Remark~\ref{RemGamEnr}.

We abbreviate again $\mT_{\Delta_M}$ to $\mD_M$, for $M\in\Indec\Rep S_d$.

\begin{lemma}\label{LemmT}
The functor $\mT$ is exact and, for every $X\in\cC$, admits a commutative diagram
\begin{equation}\label{pathmT}\xymatrix{\Fun_{\bk}(\Rep S_d,\Vecc)\ar[rr]^-{\mT}\ar[d]&&\SPol^d\cC\ar[dr]^{\eqref{Eval2}}\\
\Fun_{\bk}(\mB^d_X,\Vecc)=\mo(\mB^d_X)\ar[rr]^-{\sim}_-{H_X}&&\Rep^d_{\cC}(\GL_X,\phi)\ar@{}[r]|-*[@]{\subset}&\Rep^d_{\cC}\GL_X.
}\end{equation}
\end{lemma}
\begin{proof}
To prove exactness of $\mT$, by the discussion following Definition~\ref{DefSPol}, it is sufficient to show that for each $X\in\cC$, the upper path in \eqref{pathmT} (or equivalently, its composition with the forgetful functor $\Rep^d_{\cC}\GL_X\to \cC$) is exact. We consider therefore the commutative diagram
$$\xymatrix{
\Fun_{\bk}(\Rep S_d,\Vecc)\ar[r]\ar[d]_{\eqref{Funmo0}}&\Fun^{\cC}_{\bk}(\cC\boxtimes\Rep S_d,\cC)\ar[r]\ar[rd]&\Fun_{\cC}(\underline{\cC\boxtimes\Rep S_d},\uC)\ar[d]\\
\UFun^d\ar[r]^{\Phi\mapsto \Phi^{\cC}(X)}&\cC&\ar[l]\Mod_{\cC}\cS(X,d).
}$$
Here, the upper path from top left to bottom right is simply the upper path in~\eqref{pathmT}, up to the equivalence in Lemma~\ref{lem:SchurPoly}. The diagonal arrow with commutative triangle is an application of Lemma~\ref{LemModtoEnr}, and that lemma also shows that the left-hand side of the diagram is commutative. In conclusion, exactness of $\mT$ follows from Lemma~\ref{LemSq0}(1).

It now follows from Lemma~\ref{LemSer}(1) that the upper path in \eqref{pathmT} factors as
$$
\xymatrix{
\mo^{\Frg}(\Rep S_d)\ar@{^{(}->}[r]\ar[rd]&\Fun_{\bk}(\Rep S_d,\Vecc)\ar[rr]^-{\mT}\ar[d]&&\SPol^d\cC\ar[d]^{\eqref{Eval2}}\\
&\mo(\mB^d_X)=\Fun_{\bk}(\mB^d_X,\Vecc)\ar@{-->}[rr]&&\Rep^d_{\cC}\GL_X.
}$$
To prove that the dashed arrow is as in the lemma, it suffices to show that the above functor(s) from $\mo(\Rep S_d)$ to $\Rep^d_{\cC}\GL_X$ is/are isomorphic to the one(s) in \eqref{PolComDia}. Also this follows from Lemma~\ref{LemModtoEnr} and Remark~\ref{RemUniqueAct}.
\end{proof}

\begin{corollary}\label{Cor-Rev}
For indecomposable $S_d$-representation $M$ and $X\in\cC$, we have $D_M(X)=0$ if and only if $\mD_M(X)=0$. Consequently $\mT$ factors through and restricts to exact faithful functors
$$\Fun(\mB^d[\cC],\Vecc)\;\to\; \SPol^d\cC\quad\mbox{and}\quad\mo^{\Frg}(\mB^d[\cC])\;\to\; \SPol^d_{\circ}\cC,$$
\end{corollary}
\begin{proof}
Lemma~\ref{LemmT} implies the first claim, producing the first functor. This restricts to the second functor since $\mT_{\Frg}\simeq\mT^d$.
\end{proof}

\subsubsection{}\label{AnotherDiagram}
We can now consider the exact functor 
\begin{equation}\label{mTup}
\cC\boxtimes \mo^{\Frg}(\mB^d[\cC])\;\xrightarrow{\cC\boxtimes\mT}\;\cC\boxtimes\SPol^d_{\circ}\cC\;\to\; \SPol^d_{\circ}\cC,
\end{equation}
where the second functor corresponds to the canonical $\cC$-module structure $(Y,\mT)\mapsto Y\otimes\mT$ of $\SPol^d\cC$. By Lemma~\ref{LemmT}
we have a commutate diagram
$$\xymatrix{
\cC\boxtimes \mo^{\Frg}(\mB^d[\cC])\ar[rr]^-{\eqref{mTup}}\ar[d]&&\SPol^d_{\circ}\cC\ar[d]^{\eqref{Eval2}}\\
\cC\boxtimes \mo(\mB^d_X)\ar[rr]^{\sim}&&\Mod_{\cC}\cS(X,d)
}$$
where the lower horizontal arrow is the equivalence in (the proof of) Theorem~\ref{LemDVan}(1).

\begin{theorem}\label{ThmSPol}
\begin{enumerate}
\item The functor \eqref{mTup} yields an equivalence
$$\cC\boxtimes\mo^{\Frg}(\mB^d[\cC])\;\xrightarrow\; \SPol^d_{\circ}\cC.$$

\item If $\mB^d[\cC]=\bT^d$, then 
$\SPol^d_{\circ}\cC\simeq \Pol^d_{\bk}\boxtimes\cC.$
\end{enumerate}
\end{theorem}
\begin{proof}
Part (1) can be proved using the same argument as in the proof of Theorem~\ref{ThmPol}(1).

Part (2) follows from the combination of part (1) and Theorem~\ref{ThmPol}(1).
\end{proof}

\begin{remark}
The equivalence in Theorem~\ref{ThmSPol}(1) can be extended to 
$$\cC\boxtimes\Fun_{\bk}(\mB^d[\cC],\Vecc)\;\xrightarrow\; \SPol^d\cC.$$
In particular, if $\Indec\mB^d[\cC]$ is finite ({\it i.e.} $\Indec\mB^d[\cC]=\mB^d_X$ for some $X\in\cC$), $\SPol^d\cC$ is equivalent to $\SPol^d_{\circ}\cC$.
\end{remark}

\subsection{Discerning objects and invariant theory}
Let $\cC$ be a tensor category with a fixed object $X$.

\begin{theorem}\label{ThmDisc}
The following conditions on $X\in\cC$ are equivalent:
\begin{enumerate}
\item We have $\mB^d_X=\mB^d[\cC]$.
\item For all $Y\in\cC$, the morphism in $\cC$
$$\Gamma^d(X^{\ast}\otimes Y)\,\otimes\, \Gamma^d(Y^\ast\otimes X)\;\to\; \uEnd(Y^{\otimes d})^{S_d}$$
coming from evaluation of $X^{\otimes d}$ yields an epimorphism.
\item The evaluation at $X$ in \eqref{Eval2} yields an equivalence
$$\SPol^d_{\circ}\cC\;\simeq\; \Mod_{\cC}\cS(X,d)=\Rep^d\GL_X,\quad \mT\mapsto \mT(X).$$
\item We have $D_M(X)\not=0$ for all $M\in\Indec\mB^d[\cC]$.
\end{enumerate}
\end{theorem}
\begin{proof}%[Proof of Theorem~\ref{ThmDisc}]
Equivalence of (1) and (4) follows from Theorem~\ref{LemDVan}(2).

That (3) implies (4) follows from Corollary~\ref{Cor-Rev}. %by considering $\mD_M$ in $\SPol^d\cC$. Clearly this is zero in $\SPol^d\cC$ if and only if $0=\mD_M(Y)=D_M(Y)$ for all $Y\in\cC$, while it is sent zero in $\cS(X,d)\mbox{{\rm-mod}}$ if and only if $D_M(X)=0$. Hence the conclusion follows again from Theorem~\ref{LemDVan}(2).

That (1) implies (3) follows from the diagram in \ref{AnotherDiagram} and Theorem~\ref{ThmSPol}(1).

Finally, we show equivalence between (1) and (2). To keep notation more transparent, we assume that $\cC$ has projective objects. We can rewrite property (1) as the property that, for each $Y\in\cC$ and $j,l\in\Irr\cC$, the composition morphism
\begin{eqnarray*}
&&\bigoplus_{i}\Hom_{S_d}(\Hom(P_i, X^{\otimes d}),\Hom(P_j,Y^{\otimes d}))\otimes \Hom_{S_d}(\Hom(P_l, Y^{\otimes d}),\Hom(P_i,X^{\otimes d}))\\
& &\to \;\Hom_{S_d}(\Hom(P_l, Y^{\otimes d}),\Hom(P_j,Y^{\otimes d}))
\end{eqnarray*}
is surjective. 

By using Lemma~\ref{LemDual}, and by denoting the projective cover of $L_i^\ast$ by $P_i'$, the above morphism can be rewritten as
\begin{eqnarray*}
&&\bigoplus_{i}\left(\Hom(P_i', (X^\ast)^{\otimes d})\otimes\Hom(P_j,Y^{\otimes d})\right)^{S_d}\otimes \left(\Hom(P_l', (Y^\ast)^{\otimes d})\otimes\Hom(P_i,X^{\otimes d})\right)^{S_d}\\
& &\to \;\left(\Hom(P_l', (Y^\ast)^{\otimes d})\otimes \Hom(P_j,Y^{\otimes d})\right)^{S_d}.
\end{eqnarray*}

Surjectivity of these maps for all $j,l$ can be written as a condition for every $P\in\Proj\cC$, by taking a direct sum of tensor products with the vector spaces $\Hom(P,L_l^\ast\otimes L_j)$. By \eqref{EqNat} we can then identify the target of the map as
$$\Hom(P, (Y^\ast\otimes Y)^{\otimes d})^{S_d}\;=\; \Hom(P,\uEnd(Y)^{S_d}).$$
Moreover, for the source, we can instead take a direct sum of tensor products with
$$\Hom(P, L_l^\ast\otimes L_i\otimes L_i^\ast\otimes L_j)$$
and adapt the map by using evaluation at $L_i$. Evaluation being an epimorphism, this does not change whether the overall map is surjective or not. Moreover, we can then use two different labels for the $i$ in $L_i$ and $L_i^\ast$, understanding that the map is zero on cases with two distinct labels. After all this, the source becomes, again via \eqref{EqNat},
$$\Hom(P, (X^\ast\otimes Y \otimes Y^\ast\otimes X))^{S_d\times S_d}\;\simeq\;\Hom(P,\Gamma^d(X^\ast\otimes Y)\otimes \Gamma^d(Y^\ast\otimes X)).$$
By construction, the map we now obtain is $\Hom(P,-)$ of the map in (2).
\end{proof}

\begin{question}
What are the (minimal) relatively $d$-discerning objects in $\sVec$ and $\Ver_p$? For $\mathrm{char}(\bk)=0$ the condition for $\bk^{m|n}\in\sVec$ to be relatively $d$-discerning is $d<(m+1)(n+1)$.
\end{question}

Below, faithfulness as a $\bk S_d$-module refers to the annihilator ideal in $\bk S_d$ being zero and is thus stronger than faithfulness of the $S_d$-action.

\begin{prop}\label{PropInvTh}
The following conditions on $X$ are equivalent:
\begin{enumerate}
\item $\beta_X^d:\bk S_d\to\End_{\GL_X}(X^{\otimes d})$ is an isomorphism;
\item $\beta^d_X:\bk S_d\to\End_{\cC}(X^{\otimes d})$ is injective;
\item $\mB^d_X$ contains a faithful $\bk S_d$-module;
\item $\bk S_d\in\mB^d_X$;
\item $D_{Y^\lambda}(X)\not=0$ for all $p$-restricted $\lambda\vdash d$.
\end{enumerate}
\end{prop}
\begin{proof}
Obviously (1) implies (2).

To show that (2) implies (3) we assume for convenience that $\cC$ has projective objects. Let $P\tto X^{\otimes d}$ be a projective cover in $\cC$. Then $\Hom_{\cC}(P,X^{\otimes d})$ is a faithful $\End_{\cC}(X^{\otimes d})$-module. Hence under assumption (2), the corresponding representation in $\mB^d_X$ is faithful.

That (3) implies (4) is well-known, see for instance \cite[Lemma~1.1.1(ii)]{CEKO}.

We show that (4) implies (1). Fully faithfulness of the equivalence in Theorem~\ref{LemDVan}(1) yields an isomorphism
$$\End(\Frg)\;\xrightarrow{\sim}\; \End_{\GL_X}(X^{\otimes d})$$
for the endomorphism algebra of $\Frg:\mB^d_X\to\Vecc$. Under condition (4), $\Frg$ is representable by $\bk S_d$, so that the Yoneda lemma yields an isomorphism with $\bk S_d$. Tracing through the morphisms shows that the resulting composite isomorphism is $\beta^d_X$.

Equivalence between (4) and (5) follows from Theorem~\ref{LemDVan}(2), since these $Y^\lambda$ are precisely the indecomposable projective modules.
\end{proof}

\begin{remark}
\begin{enumerate}
\item Note that $\bk S_d\to\End_{\GL_X}(X^{\otimes d})$  is not always surjective when $\mathrm{char}(\bk)>0$, see for example \cite[Theorem~B]{CEKO}. For $\mathrm{char}(\bk)=0$ the morphism is always surjective by \cite[\S 10.11]{Del07}.
\item Together with \cite[Theorem~A]{CEKO}, the result in Proposition~\ref{PropInvTh} completely determines when the braid action yields an isomorphism for $\cC=\sVec$.
\item Relatively discernable objects satisfy \ref{PropInvTh}(5) and hence all conditions in \Cref{PropInvTh}.
\item The techniques in the proof of Proposition~\ref{PropInvTh} demonstrate that, for arbitrary $X$ and with $M:=\Omega^d(X)\in\Rep S_d$, we have an isomorphism and commutative triangle
$$\xymatrix{
\End_{\End_{S_d}(M)}(M)\ar[rr]^{\sim}&& \End_{\GL_X}(X^{\otimes d}).\\
\bk S_d\ar[u]\ar[rru]_{\beta^d_X}
}$$ 

\end{enumerate}

\end{remark}

\subsection*{Acknowledgement} The author thanks the referee for many useful comments. The author thanks Pavel Etingof, Johannes Flake, Alexander Kleshchev, Rapha\"el Rouquier and Andrew Snowden for stimulating discussions. The research was partly supported by ARC grants DP210100251 and FT220100125.

\end{document}